\documentclass[a4paper]{article}

\usepackage[latin1]{inputenc}
\usepackage[ngerman]{babel}
\usepackage[T1]{fontenc}
\usepackage{lmodern}
\usepackage{amsthm,amsmath,amsfonts,amssymb}
\usepackage{mathtools}
\usepackage{mathcomp}
\usepackage{nicefrac}
\usepackage{pxfonts}
\usepackage[all]{xy}
\usepackage{graphicx}
\usepackage{array}

\DeclareGraphicsExtensions{.eps,.pdf,.jpg}

\title{Eine Spektralsequenz in ungerader Khovanov-Homologie}
\author{Masterarbeit von Simon Beier\\ Betreuer: Prof. Dr. Thomas Schick\\ Zweitgutachter: Prof. Dr. Ralf Meyer}
\date{September 2011}

\newcommand{\eps}{\varepsilon}
\newcommand{\p}{\piup}
\newcommand{\g}{\textup{G}}
\newcommand{\V}{\textup{V}}
\newcommand{\gl}{\g(C,\eps)}
\newcommand{\gle}{\g(C,\etaup)}
\newcommand{\defeq}{\mathrel{\vcentcolon=}}

\theoremstyle{plain}
\newtheorem{satz}[equation]{Satz}
\newtheorem{prop}[equation]{Proposition}
\newtheorem{kor}[equation]{Korollar}
\newtheorem{lemma}[equation]{Lemma}

\theoremstyle{definition}
\newtheorem{defi}[equation]{Definition}
\newtheorem*{dank}{Danksagung}

\theoremstyle{remark}

\newtheorem{bemmn}[equation]{Bemerkung}

\usepackage{hyperref}

\begin{document}

\maketitle

\thispagestyle{empty}

\newpage

Hiermit versichere ich, dass ich diese Arbeit selbständig verfasst und nur die angegebenen Quellen und Hilfsmittel benutzt habe.\\\\
Göttingen, im September 2011

\thispagestyle{empty}

\newpage

\tableofcontents

\thispagestyle{empty}

\setcounter{page}{0}

\newpage

\section{Einleitung}
In den späten 1990ern entwickelte Khovanov eine Verschlingungsinvariante, die Khovanov-Homologie (\cite{kh}, \cite{barnat}). Sie ist eine bigraduierte abelsche Gruppe, die als Homologie des Khovanov-Kettenkomplexes, der einem Verschlingungsdiagramm zugeordnet wird, entsteht. Die Eulercharakteristik der Khovanov-Homologie ist das Jones-Polynom.\\
Ozsvath, Rasmussen und Szabo konstruierten 2007 eine weitere Verschlingungsinvariante, die ungerade Khovanov-Homologie (\cite{ors}). Auch hierbei handelt es sich um eine bigraduierte abelsche Gruppe, die als Homologie eines Kettenkomplexes gewonnen wird. Die Konstruktion ist ähnlich zu der von Khovanov und tatsächlich stimmen die beiden Invarianten über $\mathbb{Z}/2\mathbb{Z}$ überein. Ozsvath, Rasmussen und Szabo ordnen einem Verschlingungsdiagramm zwei Isomorphieklassen von Kettenkomplexen zu, die ungeraden Khovanov-Kettenkomplexe vom Typ X und vom Typ Y. Sie behaupten die beiden Isomorphieklassen wären gleich, ohne allerdings einen Isomorphismus anzugeben. Die Homologiegruppen vom Typ X und vom Typ Y sind Verschlingungsinvarianten. Wir werden zeigen, dass sie isomorph sind.\\
2010 konstruierte Szabo eine Spektralsequenz über $\mathbb{Z}/2\mathbb{Z}$, die ab dem $E^2$-Term eine Verschlingungsinvariante ist (\cite{szabo}). Der $E^2$-Term ist die Khovanov-Homologie mit $\mathbb{Z}/2\mathbb{Z}$-Koeffizienten. Die Spektralsequenz konvergiert gegen eine weitere Verschlingungshomologie. Szabo geht dabei wie folgt vor: Khovanovs Differential erhöht im Khovanov-Kettenkomplex den $h$-Grad um 1. Szabo konstruiert nun für alle natürlichen Zahlen $n$ einen Homomorphismus $d_n$ auf dem Khovanov-Kettenkomplex, der den $h$-Grad um $n$ erhöht. Sein $d_1$ entspricht dem Khovanov-Differential. Szabo definiert dann $d\defeq\sum_{n\ge 1}d_n$ und zeigt, dass $d\circ d=0$. Aus dem resultierenden filtrierten Kettenkomplex erhält Szabo dann seine Spektralsequenz, die gegen die Homologie des Komplexes konvergiert. Szabo gibt eine zweite Version seiner Spektralsequenz an (\cite{szabo}, Anfang von Kapitel 8). Wir zeigen, dass die beiden Spektralsequenzen ab dem $E^2$-Term, sowie die beiden Verschlingungshomologien, gegen die sie konvergieren, isomorph sind.\\
Das Hauptanliegen dieser Arbeit ist einen Lift von Szabos Komplex in die ganzen Zahlen zu konstruieren, so dass unser $d_1$ dem ungeraden Khovanov-Differential entspricht. Wir erhalten damit eine Spektralsequenz über $\mathbb{Z}$, die ab dem $E^2$-Term eine Verschlingungsinvariante ist und deren $E^2$-Term die ungerade Khovanov-Homologie ist. Die Spektralsequenz konvergiert dann gegen eine Verschlingungshomologie über $\mathbb{Z}$, aus der man mit dem universellen Koeffizienten-Theorem die Verschlingungshomologie von Szabo berechnen kann.
\begin{dank}
Ich möchte Prof. Dr. Thomas Schick für das interessante Thema und die gute Betreuung danken und dafür, dass ich seit meinem ersten Semester in seinen Vorlesungen und Seminaren sehr viel lernen durfte.\\
Ich möchte meinen Eltern danken, die mich immer unterstützt haben.
\end{dank}

\section{Geometrische Grundlagen}
Betrachte die $S^2$ als orientierte glatte Untermannigfaltigkeit des $\mathbb{R}^3$.
\begin{defi}
Ein Kreis sei eine glatte Untermannigfaltigkeit der $S^2$, die diffeomorph zur $S^1$ ist. Ein Bogen sei eine glatte Untermannigfaltigkeit mit Rand der $S^2$, die diffeomorph zum Intervall $[0,1]$ ist.
\end{defi}
\begin{defi}
Für $k\in\mathbb{N}_0$ sei ein $k$-dimensionaler Konfigurationsvertreter ein Tripel $(T,\sigma,\tau)$, dabei sei
\begin{description}
\item[-]
$T$ eine Teilmenge der $S^2$, so dass es endlich viele Kreise und $k$ Bögen gibt, so dass Folgendes gilt:
\begin{enumerate}
\item 
$T$ ist die Vereinigung der Kreise und Bögen.
\item 
Die Bögen sind paarweise disjunkt.
\item 
Die Kreise sind paarweise disjunkt.
\item 
Das Innere jedes Bogens ist disjunkt zu allen Kreisen.
\item 
Die Randpunkte der Bögen liegen auf den Kreisen.
\item 
Ein Bogen und ein Kreis verlaufen an einem gemeinsamen Punkt nicht tangential.
\end{enumerate}
Wir können also aus $T$ die Kreise und Bögen zurückgewinnen.
\item[-]
$\sigma$ eine Ordnung auf den Bögen von $T$.
\item[-]
$\tau$ eine Ordnung auf den Kanten von $T$. Dabei sind die Kanten von $T$ die Zusammenhangskomponenten von $T\setminus\bigcup_{i=1}^k\gamma_i$, wobei $\gamma_1,\dots,\gamma_k$ die Bögen von $T$ seien.
\end{description}
\end{defi}
Beachte, dass $\tau$ eine Ordnung auf den Kreisen von $T$ induziert. Ordne nämlich jedem Kreis $K$ die kleinste Kante (im Sinne von $\tau$) zu, die Teilmenge von $K$ ist. $\tau$ induziert nun eine Ordnung auf den Kanten, die wir den Kreisen zugeordnet haben und somit auch eine Ordnung auf den Kreisen.
\begin{defi}
Ein $k$-dimensionaler orientierter Konfigurationsvertreter ist ein $k$-dimensionaler Konfigurationsvertreter zusammen mit einer Orientierung auf den Bögen.
\end{defi}
\begin{defi}
Zwei $k$-dimensionale Konfigurationsvertreter $(T,\sigma,\tau)$ und $(\hat{T},\hat{\sigma},\hat{\tau})$ nennen wir äquivalent, wenn es einen orientierungserhaltenden Diffeomorphismus $f:S^2\rightarrow S^2$ gibt, so dass $f(T)=\hat{T}$. Dabei soll $f$ verträglich mit den beiden Ordnungen, die zu einem Konfigurationsvertreter gehören, sein. Für die Äquivalenz von orientierten Konfigurationsvertretern verlangen wir zusätzlich, dass $f$ die Orientierung auf den Bögen erhält. Die Äquivalenzklassen nennen wir $k$-dimensionale (orientierte) Konfigurationen.
\end{defi}
\begin{bemmn}
Unsere Definitionen erlauben es uns, von Bögen, Kanten und Kreisen von Konfigurationen zu sprechen. Zum Beispiel sei der $i$-te Bogen einer Konfiguration die Abbildung, die jedem Vertreter der Konfiguration den Bogen zuordnet, der in der Ordnung der Bögen an $i$-ter Stelle steht.
\end{bemmn}
\begin{defi}
Sei $C$ eine orientierte Konfiguration. Dann sei 
\begin{description}
\item[-]
$\overline{C}$ die (nicht orientierte) Konfiguration, die man erhält, wenn man die Orientierung der Bögen vergisst.
\item[-]
$C^*$ die orientierte Konfiguration, die man erhält, wenn man $C$ in kleinen Umgebungen der Bögen wie in Abbildung \ref{stern} verändert. Die Bögen werden also um 90° gegen den Uhrzeigersinn gedreht.
\item[-]
$r(C)$ die orientierte Konfiguration, die man erhält, wenn man die Orientierung aller Bögen umkehrt. Es gilt $(C^*)^*=r(C)$.
\item[-]
$m(C)$ die orientierte Konfiguration, die man erhält, wenn man die Orientierung der $S^2$ umkehrt. Es gilt $(m(C^*))^*=m(C)$.
\end{description}
$C^*$ und $m(C)$ seien für (nicht orientierte) Konfigurationen $C$ analog zu obiger Definition definiert.
\end{defi}
\begin{figure}[htbp]
\centering
\includegraphics{graph.1}
\caption{$C\rightarrow C^*$}
\label{stern}
\end{figure}
\begin{defi}
Sei $C$ eine (orientierte) Konfiguration. Dann nennen wir
\begin{description}
\item[-]
die Kreise von $C$ auch die Startkreise von $C$.
\item[-]
die Kreise von $C^*$ auch die Endkreise von $C$.
\item[-]
Die Kreise von $C$, die disjunkt zu allen Bögen sind, die passiven Kreise von $C$. Die passiven Kreise von $C$ bilden die nulldimensionale Konfiguration $\textup{pass}(C)$.
\item[-]
alle anderen Kreise von $C$ die aktiven Kreise von $C$.
\end{description}
Wir erhalten den aktiven Teil $\textup{akt}(C)$ von $C$, wenn wir alle passiven Kreise weglassen. Wir nennen $C$ aktiv, wenn es gleich seinem aktiven Teil ist, also wenn es keine passiven Kreise hat. Die passiven Kreise von $C$ sind gleich den passiven Kreisen von $C^*$.
\end{defi}
\begin{defi}
Für $n\le k\in\mathbb{N}_0$ sei $\sum(k,n)$ die Teilmenge von $\{0,1,*\}^k$, die alle Elemente enthält, bei denen genau $n$ Komponenten ein $*$ sind. $\sum(k,0)$, $\sum(k,1)$ und $\sum(k,n)$ sind die Ecken, Kanten und $n$-dimensionalen Seiten des Hyperwürfels $[0,1]^k$. Für $\eps\in\sum(k,n)$ entstehen $\eps^0,\eps^1\in\sum(k,0)$ aus $\eps$, indem wir alle $*$ durch $0$ beziehungsweise $1$ ersetzen. Wir nennen $\eps^0$/$\eps^1$ den Anfangs-/Endpunkt von $\eps$.\\
Für eine $k$-dimensionale (orientierte) Konfiguration $C$ und $\eps\in\sum(k,n)$ sei $\g(C,\eps)$ die $n$-dimensionale (orientierte) Konfiguration, die man wie folgt aus $C$ erhält: Für alle $i\in\{1,\dots k\}$, für die $\eps_i=1$ ist, verändere $C$ in einer kleinen Umgebung des $i$-ten Bogens wie in Abbildung \ref{stern}. Anschließend entferne alle Bögen, für die $\eps_i\ne*$ ist.\\
Wie bekommen wir eine Ordnung auf den Kanten von $\g(C,\eps)$? Es kann sein, dass mehrere Kanten von $C$ zu einer Kante in $\g(C,\eps)$ zusammenfallen. Ordne einer Kante in $\g(C,\eps)$ die kleinste der Kanten aus $C$, die zu unserer Kante zusammenfallen, zu. Dadurch erhalten wir eine Ordnung der Kanten von $\g(C,\eps)$.
\end{defi}

\section{Konstruktion der ungeraden Khovanov-Homologie}
\begin{defi}
Sei $C$ eine (orientierte) Konfiguration. Dann sei $\V(C)$ die freie abelsche Gruppe, die von den Kreisen von $C$ erzeugt wird.
\end{defi}
\begin{defi}
Wir ordnen einer $k$-dimensionalen (orientierten) Konfiguration $C$ die abelsche Gruppe
\[\Gamma(C)\defeq\bigoplus_{\eps\in\sum(k,0)}\Lambda\V(\g(C,\eps))\]
zu. Dabei steht $\Lambda$ für die äußere Algebra.
\end{defi}
Für eine aktive orientierte eindimensionale Konfiguration $C$ möchten wir einen $\mathbb{Z}$-Modul-Homomorphismus
\[\partial_C:\Lambda\V(C)\rightarrow\Lambda\V(C^*)\]
definieren. Für $C$ gibt es die zwei Möglichkeiten aus Abbildung \ref{eindim} (bis auf die Ordnung auf den Kanten). $C^*$ ist dann jeweils der andere Fall. Hier seien $x_1$, $x_2$, $y$ Bezeichnungen für die Kreise, die unabhängig von der Ordnung auf den Kreisen vergeben wurden.\\
\begin{figure}[htbp]
\centering
\includegraphics{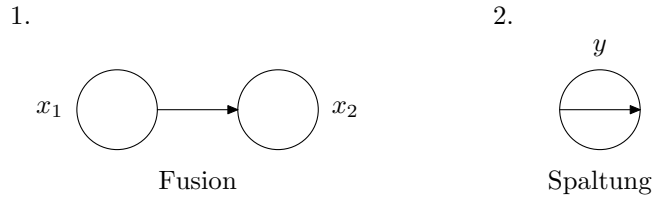}
\caption{aktive orientierte eindimensionale Konfigurationen}
\label{eindim}
\end{figure}
Im 1. Fall sei
\[\partial_C(1)=1,\qquad\partial_C(x_1)=y,\qquad
\partial_C(x_2)=y,\qquad\partial_C(x_1\wedge x_2)=0.\]
Im 2. Fall sei
\[\partial_C(1)=x_1-x_2,\qquad\partial_C(y)=x_1\wedge x_2.\]
Als Nächstes möchten wir $\partial_C$ auch für den Fall definieren, dass $C$ nicht unbedingt aktiv ist. Für ein beliebiges $k\in\mathbb{N}_0$ seien $z_1,\dots,z_k$ beliebige passive Kreise von $C$ und $\omega\defeq z_1\wedge\cdots\wedge z_k$. Für den aktiven Teil von $C$ gibt es die beschriebenen zwei Möglichkeiten. Im 1. Fall sei
\[\partial_C(\omega)=\omega,\qquad\partial_C(x_1\wedge\omega)=
y\wedge\omega,\qquad\partial_C(x_2\wedge\omega)=y\wedge\omega,
\qquad\partial_C(x_1\wedge x_2\wedge\omega)=0.\]
Im 2. Fall sei
\[\partial_C(\omega)=(x_1-x_2)\wedge\omega,\qquad
\partial_C(y\wedge\omega)=x_1\wedge x_2\wedge\omega.\]
$\partial_C$ lässt sich auch noch auf andere Weise beschreiben. Im 1. Fall sei
\[f_C:\{\textup{Kreise von $C$}\}\to\{\textup{Kreise von $C^*$}\}\]
eine Abbildung, die die passiven Kreise identisch abbildet. Weiterhin sei $f_C(x_1)=y=f_C(x_2)$. Für ein beliebiges $k\in\mathbb{N}_0$ seien $z_1,\dots,z_k$ beliebige Kreise von $C$. Dann ist
\[\partial_C(z_1\wedge\cdots\wedge z_k)=f_C(z_1)\wedge\cdots\wedge f_C(z_k).\]
Im 2. Fall seien
\[f_C^i:\{\textup{Kreise von $C$}\}\to\{\textup{Kreise von $C^*$}\}\]
für $i=1,2$ Abbildungen, die die passiven Kreise identisch abbilden. Weiterhin sei $f_C^i(y)=x_i$. Für ein beliebiges $k\in\mathbb{N}_0$ seien $z_1,\dots,z_k$ beliebige Kreise von $C$. Dann ist
\[\partial_C(z_1\wedge\cdots\wedge z_k)=(x_1-x_2)\wedge f_C^i(z_1)\wedge\cdots\wedge f_C^i(z_k)\]
für $i=1,2$.\\
Sei $C$ eine $k$-dimensionale orientierte Konfiguration und $\eps\in\sum(k,1)$. Dann erhalten wir den $\mathbb{Z}$-Modul-Homomorphismus
\[\partial_{\g(C,\eps)}:\Lambda\underbrace{\V(\gl)}_{=\V(\g(C,\eps^0))}\to
\Lambda\underbrace{\V((\gl)^*)}_{=\V(\g(C,\eps^1))}.\]
Aus diesen Abbildungen möchten wir ein Differential auf $\Gamma(C)$ erhalten. Dazu müssen wir folgendes untersuchen: Sei $C$ eine aktive orientierte zweidimensionale Konfiguration. Seien
\[a\defeq(*,0),\qquad b\defeq(1,*),\qquad c\defeq(0,*),\qquad d\defeq(*,1).\]
Wie verhalten sich dann $\partial_{\g(C,b)}\circ\partial_{\g(C,a)}$ und $\partial_{\g(C,d)}\circ\partial_{\g(C,c)}$ zueinander? Bis auf die Ordnung auf den Bögen und Kanten gibt es für $C$ die Möglichkeiten aus den Abbildungen \ref{zweidim1} und \ref{zweidim2}.\\
\begin{figure}[htbp]
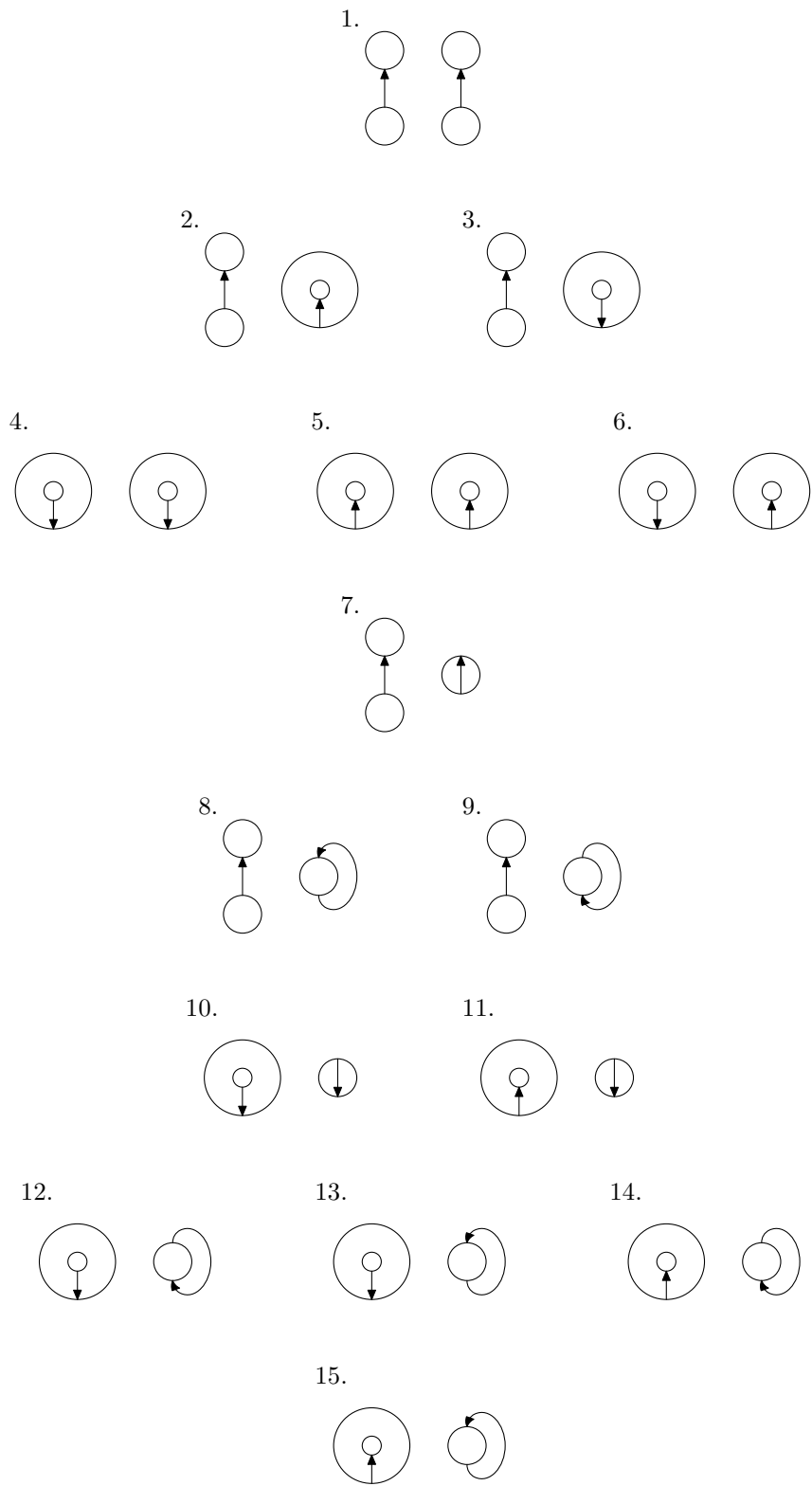

\centering
\includegraphics{bild.1}\\\ \\\ \\
\includegraphics{bild.2}\qquad\qquad
\includegraphics{bild.3}\\\ \\\ \\
\includegraphics{bild.4}\qquad\qquad
\includegraphics{bild.5}\qquad\qquad
\includegraphics{bild.6}\\\ \\\ \\
\includegraphics{bild.7}\\\ \\\ \\
\includegraphics{bild.8}\qquad\qquad
\includegraphics{bild.9}\\\ \\\ \\
\includegraphics{bild.10}\qquad\qquad
\includegraphics{bild.11}\\\ \\\ \\
\includegraphics{bild.12}\qquad\qquad
\includegraphics{bild.13}\qquad\qquad
\includegraphics{bild.14}\\\ \\\ \\
\includegraphics{bild.15}
\caption{aktive orientierte zweidimensionale Konfigurationen}
\label{zweidim1}
\end{figure}
\begin{figure}[htbp]
\centering
\includegraphics{bild.16}\qquad
\includegraphics{bild.17}\qquad
\includegraphics{bild.18}\\\ \\\ \\
\includegraphics{bild.19}\qquad\qquad
\includegraphics{bild.20}\qquad\qquad
\includegraphics{bild.21}\\\ \\\ \\
\includegraphics{bild.22}\\\ \\\ \\
\includegraphics{bild.23}\qquad\qquad
\includegraphics{bild.24}\\\ \\\ \\
\includegraphics{bild.25}\qquad\qquad
\includegraphics{bild.26}\qquad\qquad
\includegraphics{bild.27}\\\ \\\ \\
\includegraphics{bild.28}\qquad\qquad
\includegraphics{bild.29}\\\ \\\ \\
\includegraphics{bild.30}\qquad
\includegraphics{bild.31}\qquad
\includegraphics{bild.32}\qquad
\includegraphics{bild.33}\\\ \\\ \\
\includegraphics{bild.34}\qquad
\includegraphics{bild.35}\qquad
\includegraphics{bild.36}\qquad
\includegraphics{bild.37}\\\ \\\ \\
\includegraphics{bild.38}\qquad\qquad
\includegraphics{bild.39}\qquad\qquad
\includegraphics{bild.40}\\\ \\\ \\
\includegraphics{bild.41}\qquad\qquad
\includegraphics{bild.42}\qquad\qquad
\includegraphics{bild.43}\\\ \\\ \\
\includegraphics{bild.44}\qquad\qquad
\includegraphics{bild.45}
\caption{aktive orientierte zweidimensionale Konfigurationen}
\label{zweidim2}
\end{figure}
Man rechnet leicht nach, dass $\partial_{\g(C,b)}\circ\partial_{\g(C,a)}
=\partial_{\g(C,d)}\circ\partial_{\g(C,c)}$ in den Fällen 1-21 und 29-37. Wir nennen $C$ dann vom Typ K. Weiterhin gilt $\partial_{\g(C,b)}\circ\partial_{\g(C,a)}
=-\partial_{\g(C,d)}\circ\partial_{\g(C,c)}$ in den Fällen 22-28 und 38-43. Wir nennen $C$ dann vom Typ A. In den Fällen 44/45 nennen wir $C$ vom Typ X/Y. Dann gilt $\partial_{\g(C,b)}\circ\partial_{\g(C,a)}=0
=\partial_{\g(C,d)}\circ\partial_{\g(C,c)}$. Wir nennen eine orientierte zweidimensionale Konfiguration vom Typ K,A,X,Y, falls der aktive Teil vom Typ K,A,X,Y ist. Auf diese Weise lässt sich auch einer der Fälle 1-45 zuordnen.
\begin{defi}
Sei $k\in\mathbb{N}_0$. Eine Abbildung $s:\sum(k,1)\to\{-1,1\}$ nennen wir eine Kantenzuordnung. Für $\eps\in\sum(k,2)$ seien $a,b,c,d\in\sum(k,1)$ die an $\eps$ angrenzenden Kanten. Dann sei $p(\eps,s)\defeq s(a)s(b)s(c)s(d)$. Sei $C$ eine $k$-dimensionale orientierte Konfiguration. Dann nennen wir eine Kantenzuordnung $s$ bezüglich $C$ vom Typ X, falls für alle $\eps\in\sum(k,2)$ gilt
\[p(\eps,s)=\begin{cases}
1\qquad&\textup{falls $\gl$ vom Typ A oder X ist,}\\
-1\qquad&\textup{falls $\gl$ vom Typ K oder Y ist.}
\end{cases}\]
Wir nennen $s$ bezüglich $C$ vom Typ Y, falls für alle $\eps\in\sum(k,2)$ gilt
\[p(\eps,s)=\begin{cases}
1\qquad&\textup{falls $\gl$ vom Typ A oder Y ist,}\\
-1\qquad&\textup{falls $\gl$ vom Typ K oder X ist.}
\end{cases}\]
\end{defi}
\begin{satz}\label{exi}
Für jede orientierte Konfiguration $C$ gibt es eine Kantenzuordnung vom Typ X und eine Kantenzuordnung vom Typ Y.
\end{satz}
Diesen Satz werden wir später beweisen. Für $s$ vom Typ X,Y erhalten wir den Kettenkomplex
\[\Gamma(C,s)\defeq\left(\Gamma(C)=\bigoplus_{\eps\in\sum(k,0)}
\Lambda\V(\g(C,\eps)),\ \partial(C,s)\defeq\bigoplus_{\etaup\in\sum(k,1)}
s(\etaup)\partial_{\gle}\right).\]
Wir erhalten eine Bigraduierung auf $\Gamma(C)$ folgendermaßen: Für $\eps\in\sum(k,0)$ sei $|\eps|\defeq\sum_{i=1}^k\eps_i$. Dann sei
\[\Gamma(C)^h\defeq\bigoplus_{\eps\in\sum(k,0),|\eps|=h}\Lambda\V(\g(C,\eps)).\]
Für eine (orientierte) Konfiguration $C$ sei $|C|$ die Anzahl der Kreise von $C$. Es sei
\[\Gamma(C)_\delta\defeq\bigoplus_{m\in\mathbb{N}_0,\eps\in\sum(k,0):\ |\gl|-2m-|\eps|=\delta}\Lambda^m\V(\g(C,\eps)).\]
Weiterhin sei $\Gamma(C)_{h,\delta}\defeq\Gamma(C)^h\cap\Gamma(C)_\delta$. Es gilt $\Gamma(C)=\bigoplus_{h,\delta\in\mathbb{Z}}\Gamma(C)_{h,\delta}$. Das Differential $\partial(C,s)$ erhöht den $h$-Grad um 1 und senkt den $\delta$-Grad um 2.
\begin{satz}\label{inv}
Seien $C$ und $D$ orientierte Konfigurationen, so dass $\overline{C}=\overline{D}$. Seien $s$ und $t$ Kantenzuordnungen, so dass $s$ bezüglich $C$ den gleichen Typ (X oder Y) hat, wie t bezüglich $D$. Dann sind $\Gamma(C,s)$ und $\Gamma(D,t)$ isomorph als bigraduierte Kettenkomplexe.
\end{satz}
Diesen Satz werden wir später beweisen.\\
Sei $D$ nun ein Verschlingungsdiagramm.
\begin{defi}
Es seien $n_+(D)$ und $n_-(D)$ die Anzahlen der Plus- und Minus-Kreuzungen von $D$, siehe Abbildung \ref{pm}.
\end{defi}
\begin{figure}[htbp]
\centering
\includegraphics{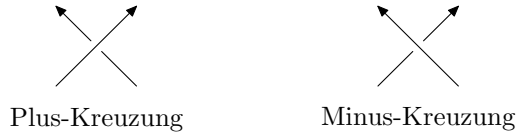}
\caption{Kreuzungen im Verschlingungsdiagramm}
\label{pm}
\end{figure}
\begin{defi}
Man kann aus $D$ eine Konfiguration $C$ erhalten, indem man die Orientierung auf $D$ vergisst und dann $D$ in kleinen Umgebungen der Kreuzungen wie in Abbildung \ref{nullgl} verändert. $C$ ist bis auf die Ordnungen auf den Bögen und Kanten eindeutig festgelegt. Wir nennen $C$ eine Nullglättung von $D$.
\end{defi}
\begin{figure}[htbp]
\centering
\includegraphics{1bild.2}
\caption{Nullglättung}
\label{nullgl}
\end{figure}
Sei $C$ eine orientierte Konfiguration, so dass $\overline{C}$ eine Nullglättung von $D$ ist. Es sei $\Gamma(D,C)^h\defeq\Gamma(C)^{h+n_-(D)}$ und $\Gamma(D,C)_\delta\defeq\Gamma(C)_{\delta-n_+(D)}$. Weiterhin sei $s$ eine Kantenzuordnung vom Typ X oder vom Typ Y bezüglich C. Dann sei $\Gamma(D,C,s)$ der bigraduierte Kettenkomplex, der sich nur wie beschrieben in der Bigraduierung von $\Gamma(C,s)$ unterscheidet. Wir nennen ihn den ungeraden Khovanov-Kettenkomplex zu $(D,C,s)$.
\begin{satz}\label{rminv}
Seien $D_1$ und $D_2$ Verschlingungsdiagramme, die durch endlich viele Reidemeister-Bewegungen auseinander hervorgehen. Dann gibt es orientierte Konfigurationen $C_1$, $C_2$, so dass $\overline{C_1}$, $\overline{C_2}$ Nullglättungen von $D_1$, $D_2$ sind und Kantenzuordnungen $s_1$, $s_2$, so dass $s_1$ bezüglich $C_1$ und $s_2$ bezüglich $C_2$ Typ X hat und so dass die Homologiegruppen $H_{h,\delta}(\Gamma(D_1,C_1,s_1))$ und $H_{h,\delta}(\Gamma(D_2,C_2,s_2))$ isomorph sind. Der Satz gilt analog für Typ Y.
\end{satz}
Wir werden den Satz später beweisen.\\
Wir ordnen also einer Verschlingung zwei Isomorphieklassen von bigraduierten abelschen Gruppen zu, die ungerade Khovanov-Homologie vom Typ X und die ungerade Khovanov-Homologie vom Typ Y. Das nächste Korollar zeigt, dass beide gleich sind.
\begin{kor}\label{xy}
Sei $C$ orientierte Konfiguration. Sei $s$ Kantenzuordnung vom Typ X und $t$ Kantenzuordnung vom Typ Y bezüglich $C$. Dann sind die Homologiegruppen $H_{h,\delta}(\Gamma(C,s))$ und $H_{h,\delta}(\Gamma(C,t))$ isomorph.
\end{kor}
\begin{proof}
Sei $k$ die Dimension von $C$. Dann gilt für alle $\eps\in\sum(k,2)$: Ist $\gl$ vom Typ A,K,X,Y, so ist $\g(r(m(C)),\eps)$ vom Typ A,K,Y,X. Also ist $s$ vom Typ Y bezüglich $r(m(C))$. Man sieht leicht: Der kanonische Isomorphismus der bigraduierten abelschen Gruppen $\Gamma(C)$ und $\Gamma(r(m(C)))$ liefert auch einen Isomorphismus der Kettenkomplexe $\Gamma(C,s)$ und $\Gamma(r(m(C)),s)$. Sei $D_1$ ein Verschlingungsdiagramm, so dass $\overline{C}$ Nullglättung von $D_1$ ist. Sei $D_2$ das Verschlingungsdiagramm, das man aus $D_1$ erhält, wenn man die Orientierung der $S^2$ umkehrt und bei allen Kreuzungen oben und unten vertauscht. Sei $L$ eine Verschlingung, die durch Parallelprojektion auf $D_1$ abgebildet wird. Kehrt man die Projektionsrichtung um, so wird $L$ auf $D_2$ abgebildet. Also geht $D_2$ durch endlich viele Reidemeister-Bewegungen aus $D_1$ hervor. Weiterhin ist $\overline{m(C)}$ Nullglättung von $D_2$. Es folgt
\begin{multline*}
H_{h,\delta}(\Gamma(C,s))\cong H_{h,\delta}(\Gamma(r(m(C)),s))\cong H_{h-n_-(D_2),\delta+n_+(D_2)}(\Gamma(D_2,r(m(C)),s))\\
\cong H_{h-n_-(D_1),\delta+n_+(D_1)}(\Gamma(D_1,C,t))\cong H_{h,\delta}(\Gamma(C,t)).
\end{multline*}
\end{proof}
Zur besseren Unterscheidung bezeichnen wir im Folgenden den Original-Khovanov-Kettenkomplex als geraden Khovanov-Kettenkomplex. Man sieht leicht:
\begin{satz}
Sei $D$ ein Verschlingungsdiagramm, $C$ eine orientierte Konfiguration, so dass $\overline{C}$ eine Nullglättung von $D$ ist und $s$ eine Kantenzuordnung vom Typ X oder vom Typ Y bezüglich $C$. Dann ist der bigraduierte Kettenkomplex $\Gamma(D,C,s)\otimes_\mathbb{Z}\mathbb{Z}/2\mathbb{Z}$ isomorph zum geraden Khovanov-Kettenkomplex von $D$ mit $\mathbb{Z}/2\mathbb{Z}$-Koeffizienten mit $(h,\delta)$-Graduierung.
\end{satz}
Durch $q=\delta+2h$ erhalten wir auf dem geraden wie auf dem ungeraden Khovanov-Kettenkomplex eine weitere Graduierung. Das Differential lässt den $q$-Grad unverändert. Bilden wir bezüglich der $h$-Graduierung die Euler-Charakteristik der geraden Khovanov-Homologie einer Verschlingung, so erhalten wir durch die $q$-Graduierung ein Laurent-Polynom. Dies ist das Jones-Polynom der Verschlingung. Aus vorherigem Satz folgt, dass dies auch für ungerade Khovanov-Homologie gilt.

\section{Existenz und Invarianz der ungeraden Khovanov-Homologie}
Sei $C$ eine orientierte dreidimensionale Konfiguration. Dann sei
\[a_C\defeq\#\left\{\eps\in\sum(3,2)\ \ |\ \ \textup{$\gl$ ist vom Typ A.}\right\}.\]
Entsprechend seien $k_C$, $x_C$ und $y_C$ für die Typen K, X und Y definiert. Es sei $\textup{Typ}_C\defeq(a_C,k_C,x_C,y_C)$.
\begin{lemma}\label{gersum}
Für jede orientierte dreidimensionale Konfiguration $C$ gilt: Die Summen $a_C+x_C$ und $a_C+y_C$ sind gerade Zahlen.
\end{lemma}
\begin{defi}
Für zwei nichtorientierte Konfigurationen $C_1$, $C_2$ sagen wir, diese gehen durch Rotation auseinander hervor, falls wir $C_2$ durch endlich viele lokale Veränderungen wie in Abbildung \ref{rot} aus $C_1$ erhalten können.
\end{defi}
\begin{figure}[htbp]
\centering
\includegraphics{1bild.3}\\\ \\\ \\
\includegraphics{1bild.4}
\caption{Rotation}
\label{rot}
\end{figure}
\begin{proof}[Beweis von Lemma \ref{gersum}]
Man sieht leicht: Sind $C_1$, $C_2$ zwei nichtorientierte dreidimensionale Konfigurationen, die durch Rotation auseinander hervorgehen, dann gilt: Für jede Orientierung auf $C_1$ gibt es eine Orientierung auf $C_2$, so dass $\textup{Typ}_{C_1}=\textup{Typ}_{C_2}$.\\
Für eine orientierte dreidimensionale Konfiguration $C$ hängt $\textup{Typ}_C$ nur vom aktiven Teil von $C$ ab.\\
In Abbildung \ref{dreidim} ist aufgelistet, welche Möglichkeiten es für eine zusammenhängende, nichtorientierte, dreidimensionale Konfiguration modulo Rotation gibt. Für jede Möglichkeit ist angegeben, welche Möglichkeiten es für $\textup{Typ}_C$ gibt, wenn wir eine Orientierung wählen. Dies lässt sich leicht überprüfen.\\
\begin{figure}[htbp]
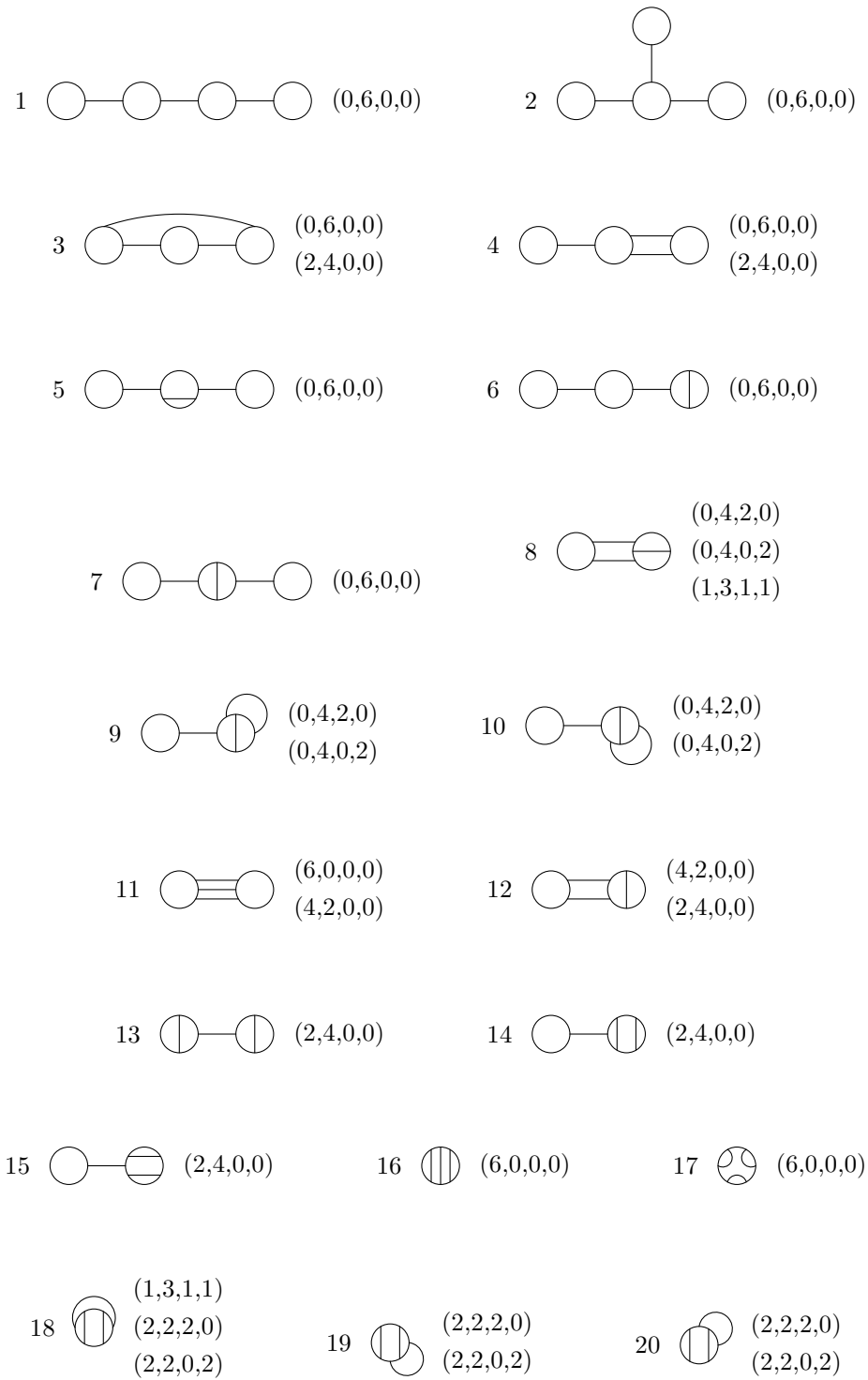

\centering
\includegraphics{2bild.1}\qquad\qquad
\includegraphics{2bild.2}\\\ \\\ \\\ \\
\includegraphics{2bild.3}\qquad\qquad
\includegraphics{2bild.4}\\\ \\\ \\\ \\
\includegraphics{2bild.5}\qquad\qquad
\includegraphics{2bild.6}\\\ \\\ \\\ \\
\includegraphics{2bild.7}\qquad\qquad
\includegraphics{2bild.8}\\\ \\\ \\\ \\
\includegraphics{2bild.9}\qquad\qquad
\includegraphics{2bild.10}\\\ \\\ \\\ \\
\includegraphics{2bild.11}\qquad\qquad
\includegraphics{2bild.12}\\\ \\\ \\\ \\
\includegraphics{2bild.13}\qquad\qquad
\includegraphics{2bild.14}\\\ \\\ \\\ \\
\includegraphics{2bild.15}\qquad\qquad
\includegraphics{2bild.16}\qquad\qquad
\includegraphics{2bild.17}\\\ \\\ \\\ \\
\includegraphics{2bild.18}\qquad\qquad
\includegraphics{2bild.19}\qquad\qquad
\includegraphics{2bild.20}
\caption{zusammenhängende dreidimensionale Konfigurationen modulo Rotation mit Möglichkeiten für $\textup{Typ}_C$}
\label{dreidim}
\end{figure}
Bleibt noch der Fall zu betrachten, dass $C$ eine aktive, nicht zusammenhängende, orientierte, dreidimensionale Konfiguration ist. Dann gibt es eine Komponente, die genau einen Bogen enthält. Ohne Einschränkung sei dies der erste Bogen in der Ordnung der Bögen. Dann ist $\g(C,(0,*,*))$ vom gleichen Typ wie $\g(C,(1,*,*))$.\\
Ist $\g(C,(*,0,0))$ eine Fusion, so sind $\g(C,(*,*,0))$, $\g(C,(*,*,1))$, $\g(C,(*,0,*))$ und $\g(C,(*,1,*))$ alle vom Typ K.\\
Sei nun $\g(C,(*,0,0))$ eine Spaltung. Dann ist $\g(C,(*,*,0))$ vom Typ K, falls $\g(C,(0,*,0))$ eine Fusion ist, sonst vom Typ A. Analog für $(*,*,1)$ und $(0,*,1)$, für $(*,0,*)$ und $(0,0,*)$, für $(*,1,*)$ und $(0,1,*)$. Von $\g(C,(0,*,0))$, $\g(C,(0,*,1))$, $\g(C,(0,0,*))$ und $\g(C,(0,1,*))$ sind aber die Anzahlen der Fusionen und Spaltungen jeweils gerade.
\end{proof}
\begin{proof}[Beweis von Satz \ref{exi}]
Wir zeigen den Satz für Typ X, für Typ Y funktioniert es analog. Sei $F$ die multiplikative Gruppe mit Elementen $-1$, $1$. Sei $C$ orientierte $k$-dimensionale Konfiguration. Wir erhalten eine CW-Struktur auf dem Hyperwürfel $Q\defeq[0,1]^k$ durch die $n$-dimensionalen Seiten $\sum(k,n)$ für $n\le k$. Wir definieren eine zelluläre 2-Kokette $\phi\in C_\textup{cell}^2(Q;F)$, die jedem $\eps\in\sum(k,2)$ das Element $1\in F$ zuordnet, falls $\gl$ vom Typ A oder X ist und das Element $-1\in F$, falls $\gl$ vom Typ K oder Y ist. Wegen Lemma \ref{gersum} ist $\phi$ ein Kozykel. Weil $Q$ zusammenziehbar ist, ist $\phi$ Korand. Also gibt es eine Kantenzuordnung vom Typ X.
\end{proof}
\begin{lemma}\label{inv1}
Sei $C$ orientierte Konfiguration. Seien $s$ und $t$ Kantenzuordnungen vom gleichen Typ (X oder Y) bezüglich $C$. Dann sind $\Gamma(C,s)$ und $\Gamma(C,t)$ isomorph als bigraduierte Kettenkomplexe.
\end{lemma}
\begin{proof}
Seien die Bezeichnungen wie im Beweis von Satz \ref{exi}. Dann ist $st\in C_\textup{cell}^1(Q;F)$ ein Kozykel und somit auch Korand. Es gibt also $\etaup:\sum(k,0)\to\{-1,1\}$, so dass für alle $\eps\in\sum(k,1)$ gilt $\etaup(\eps^0)\etaup(\eps^1)=s(\eps)t(\eps)$. Wir erhalten den gesuchten Isomorphismus, indem wir für alle $\theta\in\sum(k,0)$ auf $\Lambda\V(\g(C,\theta))$ mit $\etaup(\theta)$ multiplizieren.
\end{proof}
\begin{lemma}\label{inv2}
Seien $C$ und $D$ orientierte Konfigurationen, so dass $\overline{C}=\overline{D}$ und $s$ eine Kantenzuordnung vom Typ X oder Y bezüglich $C$. Dann gibt es eine Kantenzuordnung $t$, so dass $t$ bezüglich $D$ den gleichen Typ hat, wie $s$ bezüglich $C$ und so dass die bigraduierten Kettenkomlexe $\Gamma(C,s)$ und $\Gamma(D,t)$ gleich sind.
\end{lemma}
\begin{proof}
Wir zeigen den Satz für Typ X, für Typ Y funktioniert es analog. Es reicht, den Fall zu betrachten, wenn sich $C$ und $D$ nur durch die Orientierung eines Bogens unterscheiden. Ohne Einschränkungen sei dies der erste Bogen in der Ordnung der Bögen. Sei $\alpha:\sum(k,1)\to\{-1,1\}$ definiert durch
\[\alpha(\eps)=\begin{cases}
-1\qquad&\textup{falls $\eps_1=*$ und $\gl$ eine Spaltung ist,}\\
1\qquad&\textup{sonst.}
\end{cases}\]
Für alle $\eps\in\sum(k,1)$ gilt $\partial_{\gl}=\alpha(\eps)\partial_{\g(D,\eps)}$. Sei $\phi:\sum(k,2)\to\{-1,1\}$ definiert durch
\[\phi(\eps)=\begin{cases}
1\qquad&\textup{falls $\gl$ vom Typ A oder X ist,}\\
-1\qquad&\textup{falls $\gl$ vom Typ K oder Y ist.}
\end{cases}\]
Analog sei $\psi:\sum(k,2)\to\{-1,1\}$ für $D$ definiert. Mit den Bezeichnungen aus dem Beweis von Satz \ref{exi} sind $\phi,\psi\in C_\textup{cell}^2(Q;F)$ und $\alpha\in C_\textup{cell}^1(Q;F)$. Es gilt $\psi=\phi d^1(\alpha)$. Für $\eps\in\sum(k,2)$, für die die erste Komponente ungleich $*$ ist, ist dies klar. Sei nun die erste Komponente von $\eps$ gleich $*$. Ist $\gl$ einer der Fälle 1-27 oder 30-43 aus den Abbildungen \ref{zweidim1} und \ref{zweidim2}, so haben $\gl$ und $\g(D,\eps)$ den gleichen Typ und $\alpha$ ist auf genau null oder genau zwei der vier an $\eps$ angrenzenden Kanten $-1$. Ist $\gl$ Fall 28, 29, 44, 45, so ist $\g(D,\eps)$ Fall 29, 28, 45, 44 und $\alpha$ ist auf genau einer an $\eps$ angrenzenden Kante $-1$.\\
Es sei $t\defeq\alpha s$. Aus $d^1(s)=\phi$ folgt $d^1(t)=\psi$, somit ist $t$ eine Kantenzuordnung vom Typ X bezüglich $D$, so dass $\partial(C,s)=\partial(D,t)$.
\end{proof}
Mit Lemma \ref{inv1} und Lemma \ref{inv2} ist Satz \ref{inv} bewiesen. Als Nächstes möchten wir Satz \ref{rminv} beweisen.
\begin{prop}[RM1-Invarinanz]\label{p1}
Seien $D_1$ und $D_2$ Verschlingungsdiagramme, die sich wie in Abbildung \ref{rm1} lokal unterscheiden. Dann gilt die Aussage aus Satz \ref{rminv}.
\end{prop}
\begin{figure}[htbp]
\centering
\includegraphics{graph.3}
\caption{Reidemeister-Bewegung 1}
\label{rm1}
\end{figure}
\begin{proof}
Seien $C_1$, $s_1$ beliebig. Eine Nullglättung von $D_2$ sieht wie in Abbildung \ref{ngd2} aus.
\begin{figure}[htbp]
\centering
\includegraphics{1bild.5}
\caption{Nullglättung von $D_2$}
\label{ngd2}
\end{figure}
Der im Bild vorkommende Kreis sei mit $z$ der Bogen mit $\gamma$ bezeichnet. $C_2$ sei so orientiert, dass alle Bögen außer $\gamma$ die gleiche Orientierung haben wie in $C_1$. Die Orientierung von $\gamma$ sei beliebig. Sei $\gamma$ der letzte Bogen in der Ordnung der Bögen, alle anderen Bögen seien geordnet wie in $C_1$. Sei $k$ die Dimension von $C_1$. Für $\eps\in\sum(k,0)$ ist $\g(C_2,(\eps,1))=\g(C_1,\eps)$ und $\g(C_2,(\eps,0))$ ist wie $\g(C_1,\eps)$ mit der lokalen Änderung wie in Abbildung \ref{gc2}.
\begin{figure}[htbp]
\centering
\includegraphics{graph.4}
\caption{}
\label{gc2}
\end{figure}
Für $a\in\sum(k,1)$ sei $s_2(a,0)\defeq s_1(a)$ und $s_2(a,1)\defeq-s_1(a)$. Für $b\in\sum(k+1,1)$ mit $b_{k+1}=*$ sei $s_2(b)\defeq 1$. Man prüft leicht nach, dass $s_2$ bezüglich $C_2$ den gleichen Typ hat wie $s_1$ bezüglich $C_1$. Für $\eps\in\sum(k+1,0)$ mit $\eps_{k+1}=0$ sei $\V_z(\g(C_2,\eps))$ die freie abelsche Gruppe, die von allen Kreisen von $\g(C_2,\eps)$ außer $z$ erzeugt wird. Es sei
\[A\defeq\bigoplus_{\eps\in\sum(k+1,0)\textup{ mit }\eps_{k+1}=0}\Lambda\V_z(\g(C_2,\eps))\]
und $\partial_A$ die Einschränkung von
\[\bigoplus_{\etaup\in\sum(k+1,1)\textup{ mit }\etaup_{k+1}=0}s_2(\etaup)\partial_{\g(C_2,\etaup)}\]
auf $A$. Weiter sei
\[B\defeq\bigoplus_{\eps\in\sum(k+1,0)\textup{ mit }\eps_{k+1}=1}\Lambda\V(\g(C_2,\eps))\]
und
\[\partial_B\defeq\bigoplus_{\etaup\in\sum(k+1,1)\textup{ mit }\etaup_{k+1}=1}s_2(\etaup)\partial_{\g(C_2,\etaup)}\]
und $\partial_{A\to B}$ die Einschränkung von
\[\bigoplus_{\etaup\in\sum(k+1,1)\textup{ mit }\etaup_{k+1}=*}s_2(\etaup)\partial_{\g(C_2,\etaup)}\]
auf $A$. $\partial_{A\to B}$ ist ein Isomorphismus der abelschen Gruppen $A$ und $B$ und es gilt $\partial_{A\to B}\circ\partial_A=-\partial_B\circ\partial_{A\to B}$. Somit ist der Unterkomplex $(A\oplus B, \partial_A\oplus\partial_B\oplus\partial_{A\to B})$ von $\Gamma(D_2,C_2,s_2)$ azyklisch. Der Quotietenkomplex ist isomorph als bigraduierter Kettenkomplex zu $\Gamma(D_1,C_1,s_1)$, da $D_2$ eine Plus-Kreuzung mehr hat als $D_1$.
\end{proof}
\begin{prop}[RM2-Invarinanz]\label{p2}
Seien $D_1$ und $D_2$ Verschlingungsdiagramme, die sich wie in Abbildung \ref{rm2} lokal unterscheiden. Dann gilt die Aussage aus Satz \ref{rminv}.
\end{prop}
\begin{figure}[htbp]
\centering
\includegraphics{graph.5}
\caption{Reidemeister-Bewegung 2}
\label{rm2}
\end{figure}
\begin{proof}
Eine Nullglättung von $D_2$ sieht wie in Abbildung \ref{ngd22} aus.
\begin{figure}[htbp]
\centering
\includegraphics{graph.6}
\caption{Nullglättung von $D_2$}
\label{ngd22}
\end{figure}
Seien $C_2$, $s_2$ beliebig, so dass $\alpha$ der erste und $\beta$ der zweite Bogen in der Ordnung der Bögen seien. Sei $k+2$ die Dimension von $C_2$. Dann gilt:
\begin{multline*}
\Gamma(D_2,C_2,s_2)\\
=\xymatrix{ & {\Gamma(\g(C_2,(0,1,*,\dots,*)),s_2|)} \ar[dr] \\ {\Gamma(\g(C_2,(0,0,*,\dots,*)),s_2|)} \ar[ur] \ar[dr] && {\Gamma(\g(C_2,(1,1,*,\dots,*)),s_2|)} \\ & {\Gamma(\g(C_2,(1,0,*,\dots,*)),s_2|)} \ar[ur] }\\
=\xymatrix{ & {\includegraphics{1bild.6}} \ar[dr] \\ {\includegraphics{1bild.7}} \ar[ur] \ar[dr] && {\includegraphics{1bild.8}} \\ & {\includegraphics{1bild.9}} \ar[ur] }
\end{multline*}
Es sei $C_1\defeq\g(C_2,(0,1,*,\dots,*))$ und für alle $a\in\sum(k,1)$ sei $s_1(a)\defeq s_2(0,1,a)$. Im Beweis der RM1-Invarianz haben wir eine abelsche Gruppe $A$ definiert. Analog sei nun $A$ als Untergruppe von $\Gamma(\g(C_2,(1,0,*,\dots,*)))$ definiert. Dann ist $A\oplus\Gamma(\g(C_2,(1,1,*,\dots,*)))$ azyklischer Unterkomplex von $\Gamma(D_2,C_2,s_2)$. Den Quotientenkomplex nennen wir $B$. In $B$ haben wir den Unterkomplex $C\defeq\Gamma(\g(C_2,(0,1,*,\dots,*)),s_2|)$. Der Quotientenkomplex ist azyklisch. $C$ ist isomorph als bigraduierter Kettenkomplex zu $\Gamma(D_1,C_1,s_1)$, da $D_2$ eine Plus- und eine Minus-Kreuzung mehr hat als $D_1$.
\end{proof}
\begin{prop}[RM3-Invarinanz]\label{p3}
Seien $D_1$ und $D_2$ Verschlingungsdiagramme, die sich wie in Abbildung \ref{rm3} lokal unterscheiden. Dann gilt die Aussage aus Satz \ref{rminv}.
\end{prop}
\begin{figure}[htbp]
\centering
\includegraphics{graph.7}
\caption{Reidemeister-Bewegung 3}
\label{rm3}
\end{figure}
\begin{proof}
$C_1$ und $C_2$ sind in Abbildung \ref{c1c2} dargestellt.
\begin{figure}[htbp]
\centering
\includegraphics{graph.8}
\caption{}
\label{c1c2}
\end{figure}
$\alpha_1$, $\alpha_2$, $\alpha_3$, sowie $\beta_1$, $\beta_2$, $\beta_3$ sollen in den Ordnungen der Bögen an den ersten drei Stellen stehen. Die übrigen Bögen seien in $C_1$ und $C_2$ gleich geordnet und gleich orientiert. $s_1$ und $s_2$ seien beliebig, so dass $s_1$ bezüglich $C_1$ und $s_2$ bezüglich $C_2$ Typ X haben. Für Typ Y funktioniert der Beweis analog. Sei $k$ die Dimension von $C_1$ und $C_2$. $\Gamma(C_1,s_1)$ ist in Abbildung \ref{gac1} dargestellt.
\begin{figure}[htbp]
\centering
\[\xymatrix{ & {\includegraphics{3bild.1}} \ar[dr] \ar[r] & {\includegraphics{3bild.2}} \ar[dr] & \\ {\includegraphics{3bild.3}} \ar[ur] \ar[dr] \ar[r] & {\includegraphics{3bild.4}} \ar[ur] \ar[dr] & {\includegraphics{3bild.5}} \ar[r] & {\includegraphics{3bild.6}} \\ & {\includegraphics{3bild.7}} \ar[r] \ar[ur] & {\includegraphics{3bild.8}} \ar[ur] & }\]
\caption{$\Gamma(C_1,s_1)$}
\label{gac1}
\end{figure}
Sei $A$ analog zur RM1-Invarianz als Untergruppe von $\Gamma(\g(C_1,(0,1,0,*,\dots,*)))$ definiert. Sei
\[R\defeq A\oplus\bigoplus_{\eps\in\sum(k,0)\textup{ mit }(0,0,0)\ne(\eps_1,\eps_2,\eps_3)\ne(0,1,0)}\Lambda\V(\g(C_1,\eps)).\]
$R$ ist Unterkomplex von $\Gamma(C_1,s_1)$ und der Quotientenkomplex ist azyklisch. Sei $\partial_R$ die Randabbildung von $R$. Dann ist $\{a+\partial_R(b)\ \ |\ \ a,b\in A\}$ azyklischer Unterkomplex von $R$. Der Quotientenkomplex sei mit $P$ bezeichnet. $\Gamma(C_2,s_2)$ ist in Abbildung \ref{gac2} dargestellt.
\begin{figure}[htbp]
\centering
\[\xymatrix{ & {\includegraphics{4bild.1}} \ar[dr] \ar[r] & {\includegraphics{4bild.2}} \ar[dr] & \\ {\includegraphics{4bild.3}} \ar[ur] \ar[dr] \ar[r] & {\includegraphics{4bild.4}} \ar[ur] \ar[dr] & {\includegraphics{4bild.5}} \ar[r] & {\includegraphics{4bild.6}} \\ & {\includegraphics{4bild.7}} \ar[r] \ar[ur] & {\includegraphics{4bild.8}} \ar[ur] & }\]
\caption{$\Gamma(C_2,s_2)$}
\label{gac2}
\end{figure}
Analog zu $R$ können wir einen Unterkomplex $S$ von $\Gamma(C_2,s_2)$ definieren. Analog zu $P$ können wir einen Quotientenkomplex $Q$ von $S$ definieren, dessen Homologie isomorph zur Homologie von $\Gamma(C_2,s_2)$ ist. Bleibt noch zu zeigen, dass die Komplexe $P$ und $Q$ isomorph sind. Es sei $C_3\defeq\g(C_1,(1,*,\dots,*))=\g(C_2,(*,*,1,*,\dots,*))$ und $C_4\defeq\g(C_1,(*,*,1,*,\dots,*))=\g(C_2,(1,*,\dots,*))$, siehe Abbildung \ref{c3c4}.
\begin{figure}[htbp]
\centering
\includegraphics{graph.9}
\caption{}
\label{c3c4}
\end{figure}
$s_1$ und $s_2$ induzieren Kantenzuordnungen $s_1^3$, $s_2^3$ vom Typ X bezüglich $C_3$ und $s_1^4$, $s_2^4$ vom Typ X bezüglich $C_4$. Der Beweis von Lemma \ref{inv1} zeigt: Es gibt $\etaup_3,\etaup_4:\sum(k-1,0)\to\{-1,1\}$, so dass für alle $\eps\in\sum(k-1,1)$ gilt:
\begin{align*}
\etaup_3(\eps^0)\etaup_3(\eps^1)&=s_1^3(\eps)s_2^3(\eps)\\
\etaup_4(\eps^0)\etaup_4(\eps^1)&=s_1^4(\eps)s_2^4(\eps)
\end{align*}
Für alle $\eps\in\sum(k-3,1)$ gilt:
\begin{equation}\label{g1}
s_1^3(1,1,\eps)=s_1^4(1,1,\eps)
\end{equation}
\begin{equation}\label{g2}
s_1^3(0,1,\eps)=s_1^4(1,0,\eps)
\end{equation}
\begin{equation}\label{g3}
s_2^3(1,1,\eps)=s_2^4(1,1,\eps)
\end{equation}
\begin{equation}\label{g4}
s_2^3(1,0,\eps)=s_2^4(0,1,\eps)
\end{equation}
Für alle $\eps\in\sum(k-3,0)$ gilt:
\begin{equation}\label{g5}
s_1^3(*,1,\eps)=s_1^4(1,*,\eps)
\end{equation}
\begin{equation}\label{g6}
s_2^3(1,*,\eps)=s_2^4(*,1,\eps)
\end{equation}
Aus Gleichung (\ref{g1}) und (\ref{g3}) folgt: Entweder gilt $\etaup_3(1,1,\eps)=\etaup_4(1,1,\eps)$ für alle $\eps\in\sum(k-3,0)$ oder für alle $\eps\in\sum(k-3,0)$ gilt $\etaup_3(1,1,\eps)=-\etaup_4(1,1,\eps)$. Im zweiten Fall können wir $\etaup_4$ mit $-1$ multiplizieren, so dass wir stets vom ersten Fall ausgehen können. Dann gilt für alle $\eps\in\sum(k-3,0)$:
\begin{equation}\label{g7}
\etaup_3(0,1,\eps)\etaup_4(1,0,\eps)=s_2^3(*,1,\eps)s_2^4(1,*,\eps)
=s_2(*,1,1,\eps)s_2(1,1,*,\eps)\textup{ (wegen Gleichung (\ref{g5})}
\end{equation}
\begin{equation}\label{g8}
\etaup_3(1,0,\eps)\etaup_4(0,1,\eps)=s_1^3(1,*,\eps)s_1^4(*,1,\eps)
=s_1(1,1,*,\eps)s_1(*,1,1,\eps)\textup{ (wegen Gleichung (\ref{g6})}
\end{equation}
Für $i=1,2$ ist ein Element aus $\Gamma(C_i)$ ein Paar $(\omega,\eps)$ mit $\eps\in\sum(k,0)$ und $\omega\in\Lambda\V(\g(C_i,\eps))$. Nun können wir einen Isomorphismus von Kettenkomplexen $P\to Q$ angeben, wir bilden für alle $\eps\in\sum(k-3,0)$ wie folgt ab:
\begin{align*}
[(\omega,(1,0,0,\eps))]&\mapsto\etaup_3(0,0,\eps)[(\omega,(0,0,1,\eps))]\\
[(\omega,(0,0,1,\eps))]&\mapsto\etaup_4(0,0,\eps)[(\omega,(1,0,0,\eps))]\\
[(\omega,(1,0,1,\eps))]&\mapsto\etaup_3(0,1,\eps)[(\omega,(0,1,1,\eps))]
=\etaup_4(1,0,\eps)[(\omega,(1,1,0,\eps))]\\
&\textup{(weil $\g(C_2,(*,1,*,\eps))$ vom Typ K ist und wegen Gleichung (\ref{g7}))}\\
[(\omega,(1,1,0,\eps))]&\mapsto\etaup_3(1,0,\eps)[(\omega,(1,0,1,\eps))]\\
[(\omega,(0,1,1,\eps))]&\mapsto\etaup_4(0,1,\eps)[(\omega,(1,0,1,\eps))]\\
&\textup{Die beiden letzten Abbildungsvorschriften passen zusammen,}\\
&\textup{weil $\g(C_1,(*,1,*,\eps))$ vom Typ K ist und wegen Gleichung (\ref{g8}).}\\
[(\omega,(1,1,1,\eps))]&\mapsto\etaup_3(1,1,\eps)[(\omega,(1,1,1,\eps))]
=\etaup_4(1,1,\eps)[(\omega,(1,1,1,\eps))]
\end{align*}
\end{proof}

\section{Spektralsequenzen}
In diesem Kapitel werden wir uns an dem Buch von Weibel orientieren (\cite{weibel}).\\
Sei $R$ ein kommutativer Ring mit $1$.
\begin{defi}
Eine Spektralsequenz vom Grad $k\in\mathbb{Z}$ besteht aus folgenden Daten:
\begin{enumerate}
\item 
einer Familie $E_{p,q}^r$ von $R$-Moduln definiert für $p,q\in\mathbb{Z}$ und $r\in\mathbb{N}_0$ (für ein festes r werden die $E_{p,q}^r$ der $E^r$-Term der Spektralsequenz genannt)
\item
Homomorphismen $d_{p,q}^r:E_{p,q}^r\to E_{p-r,q+r+k}^r$, die $d^r\circ d^r=0$ erfüllen 
\item 
Isomorphismen $f_{p,q}^r:H_{p,q}(E^r,d^r)\defeq\ker(d_{p,q}^r)/\textup{im}(d_{p+r,q-r-k}^r)\to E_{p,q}^{r+1}$
\end{enumerate}
\end{defi}
\begin{defi}
Sei $(E,d,f)$ eine Spektralsequenz vom Grad $k$. Dann sei ${\p_{p,q}^r:\ker(d_{p,q}^r)\to H_{p,q}(E^r,d^r)}$ die kanonische Projektion. Für $r\ge 1$ definiere
\begin{align*}
Z_{p,q}^{r,r-1}&\defeq\ker(d_{p,q}^{r-1})\subset E_{p,q}^{r-1}\textup{ und }\\
Z_{p,q}^{r,s}&\defeq(f_{p,q}^s\circ\p_{p,q}^s)^{-1}(Z_{p,q}^{r,s+1})\subset E_{p,q}^s
\end{align*}
für $0\le s\le r-2$. Es sei $Z_{p,q}^r\defeq Z_{p,q}^{r,0}\subset E_{p,q}^0$. Für $r=0$ sei $Z_{p,q}^0\defeq E_{p,q}^0$. Definiere für $r\ge 1$ weiterhin
\begin{align*}
B_{p,q}^{r,r-1}&\defeq\textup{im}(d_{p+r-1,q-r-k+1}^{r-1})\subset E_{p,q}^{r-1}\textup{ und }\\
B_{p,q}^{r,s}&\defeq(f_{p,q}^s\circ\p_{p,q}^s)^{-1}(B_{p,q}^{r,s+1})\subset E_{p,q}^s
\end{align*}
für $0\le s\le r-2$. Es sei $B_{p,q}^r\defeq B_{p,q}^{r,0}\subset E_{p,q}^0$. Für $r=0$ sei $B_{p,q}^0\defeq\{0\}$.
\end{defi}
Offensichtlich gilt
\[\{0\}=B_{p,q}^0\subset B_{p,q}^1\subset\cdots\subset B_{p,q}^r\subset B_{p,q}^{r+1}\subset\cdots\subset Z_{p,q}^{r+1}\subset Z_{p,q}^r\subset\cdots\subset Z_{p,q}^1\subset Z_{p,q}^0=E_{p,q}^0\]
und $E_{p,q}^r\cong Z_{p,q}^r/B_{p,q}^r$.
\begin{defi}
\begin{align*}
B_{p,q}^\infty&\defeq\bigcup_{r=0}^\infty B_{p,q}^r\\
Z_{p,q}^\infty&\defeq\bigcap_{r=0}^\infty Z_{p,q}^r\\
E_{p,q}^\infty&\defeq Z_{p,q}^\infty/B_{p,q}^\infty
\end{align*}
Die $E_{p,q}^\infty$ werden der $E^\infty$-Term der Spektralsequenz genannt.
\end{defi}
\begin{bemmn}
Sei $r_0\in\mathbb{N}_0$. Gilt für alle feste $p,q\in\mathbb{Z}$ und alle $r\ge r_0$ die Gleichung $d_{p,q}^r=0=d_{p+r,q-r-k}^r$, so ist $Z_{p,q}^{r_0}=Z_{p,q}^{r_0+1}=\cdots=Z_{p,q}^\infty$ und $B_{p,q}^{r_0}=B_{p,q}^{r_0+1}=\cdots=B_{p,q}^\infty$. Es folgt $E_{p,q}^\infty\cong E_{p,q}^r$ für $r\ge r_0$.
\end{bemmn}
Wir erhalten kurze exakte Sequenzen der Form
\[0\to Z_{p,q}^{r+1}/B_{p,q}^r\to Z_{p,q}^r/B_{p,q}^r\xrightarrow{\tilde{d}_{p,q}^r}B_{p-r,q+r+k}^{r+1}/B_{p-r,q+r+k}^r\to 0,\]
wobei $\tilde{d}_{p,q}^r$ von $d_{p,q}^r$ induziert ist. Dadurch erhalten wir Isomorphismen $Z_{p,q}^r/Z_{p,q}^{r+1}\cong B_{p-r,q+r+k}^{r+1}/B_{p-r,q+r+k}^r$.\\
Haben wir umgekehrt $R$-Moduln $B_{p,q}^r$ und $Z_{p,q}^r$ gegeben, die
\[\{0\}=B_{p,q}^0\subset B_{p,q}^1\subset\cdots\subset B_{p,q}^r\subset B_{p,q}^{r+1}\subset\cdots\subset Z_{p,q}^{r+1}\subset Z_{p,q}^r\subset\cdots\subset Z_{p,q}^1\subset Z_{p,q}^0\]
erfüllen und zusätzlich Isomorphismen $g_{p,q}^r:Z_{p,q}^r/Z_{p,q}^{r+1}\to B_{p-r,q+r+k}^{r+1}/B_{p-r,q+r+k}^r$ für ein $k\in\mathbb{Z}$, so erhalten wir eine Spektralsequenz vom Grad $k$ wie folgt: Es sei $E_{p,q}^r\defeq Z_{p,q}^r/B_{p,q}^r$. Durch $g_{p,q}^r$ erhalten wir kurze exakte Sequenzen der Form
\[0\to Z_{p,q}^{r+1}/B_{p,q}^r\to Z_{p,q}^r/B_{p,q}^r\to B_{p-r,q+r+k}^{r+1}/B_{p-r,q+r+k}^r\to 0,\]
wodurch $d_{p,q}^r$ definiert ist. Offensichtlich gilt $d^r\circ d^r=0$. Die Isomorphismen $f_{p,q}^r$ erhalten wir durch
\[\ker(d_{p,q}^r)/\textup{im}(d_{p+r,q-r-k}^r)
=(Z_{p,q}^{r+1}/B_{p,q}^r)/(B_{p,q}^{r+1}/B_{p,q}^r)\cong Z_{p,q}^{r+1}/B_{p,q}^{r+1}=E_{p,q}^{r+1}.\]
\begin{defi}\label{defmor}
Seien $(E,d,f)$ und $(\hat{E},\hat{d},\hat{f})$ Spektralsequenzen vom Grad $k$. Ein Morphismus zwischen diesen Spektralsequenzen besteht aus $R$-Modul-Homomorphismen $\varphi_{p,q}^r:E_{p,q}^r\to\hat{E}_{p,q}^r$ für $p,q\in\mathbb{Z}$ und $r\in\mathbb{N}_0$, die Folgendes erfüllen:
\begin{enumerate}
\item
$\hat{d}_{p,q}^r\circ\varphi_{p,q}^r=\varphi_{p-r,q+r+k}^r\circ d_{p,q}^r$
\item 
$\varphi^{r+1}$ ist die von $\varphi^r$ auf Homologie induzierte Abbildung.
\end{enumerate}
\end{defi}
Somit ist $\varphi$ schon durch $\varphi^0$ festgelegt. Diese Definition macht die Spektralsequenzen vom Grad $k$ zu einer Kategorie $\textup{Spek}(k)$. Ist $\varphi$ ein Morphismus in $\textup{Spek}(k)$ und ist $\varphi_{p,q}^0$ ein Isomorphismus von $R$-Moduln für alle $p,q\in\mathbb{Z}$, so ist $\varphi$ ein Isomorphismus. Für einen Morphismus $\varphi$ in $\textup{Spek}(k)$ gilt $\varphi_{p,q}^0(Z_{p,q}^r)\subset\hat{Z}_{p,q}^r$ und $\varphi_{p,q}^0(B_{p,q}^r)\subset\hat{B}_{p,q}^r$. Es folgt $\varphi_{p,q}^0(Z_{p,q}^\infty)\subset\hat{Z}_{p,q}^\infty$ und $\varphi_{p,q}^0(B_{p,q}^\infty)\subset\hat{B}_{p,q}^\infty$. Wir erhalten also $R$-Modul-Homomorphismen $\varphi_{p,q}^\infty:E_{p,q}^\infty\to\hat{E}_{p,q}^\infty$. Somit erhalten wir für $k\in\mathbb{Z}$ und $r\in\mathbb{N}_0\cup\{\infty\}$ den Funktor $\mathfrak{T}^r(k):\textup{Spek}(k)\to\textup{BG-$R$-Mod}$ mit Ziel in den bigraduierten $R$-Moduln, der einer Spektralsequenz vom Grad $k$ ihren $E^r$-Term zuordnet.
\begin{lemma}\label{spekiso}
Ist $\varphi$ ein Morphismus in $\textup{Spek}(k)$ und ist $\varphi_{p,q}^{r_0}$ ein Isomorphismus für ein $r_0\in\mathbb{N}_0$ und für alle $p,q\in\mathbb{Z}$ (dann ist auch $\varphi_{p,q}^r$ ein Isomorphismus für $r>r_0$), so sind die $\varphi_{p,q}^\infty$ Isomorphismen.
\end{lemma}
\begin{proof}
Definiere
\begin{align*}
B_{p,q}^{\infty,s}&\defeq\bigcup_{r=s+1}^\infty B_{p,q}^{r,s}\subset E_{p,q}^s,\\
Z_{p,q}^{\infty,s}&\defeq\bigcap_{r=s+1}^\infty Z_{p,q}^{r,s}\subset E_{p,q}^s,\\
E_{p,q}^{\infty,s}&\defeq Z_{p,q}^{\infty,s}/B_{p,q}^{\infty,s}.
\end{align*}
Es gilt $B_{p,q}^{\infty,0}=B_{p,q}^\infty$, $Z_{p,q}^{\infty,0}=Z_{p,q}^\infty$ und somit $E_{p,q}^{\infty,0}=E_{p,q}^\infty$. Außerdem ist $B_{p,q}^{\infty,s}=(f_{p,q}^s\circ\p_{p,q}^s)^{-1}(B_{p,q}^{\infty,s+1})$ und $Z_{p,q}^{\infty,s}=(f_{p,q}^s\circ\p_{p,q}^s)^{-1}(Z_{p,q}^{\infty,s+1})$. Es folgt $E_{p,q}^{\infty,s}\cong E_{p,q}^{\infty,s+1}$. $\varphi_{p,q}^{r_0}$ induziert einen Isomorphismus $E_{p,q}^{\infty,r_0}\cong\hat{E}_{p,q}^{\infty,r_0}$. Zusammen mit der von $\varphi_{p,q}^{r_0-1}$ induzierten Abbildung ergibt sich folgendes kommutatives Diagramm:
\[\xymatrix{ {\hat{E}_{p,q}^{\infty,r_0-1}} \ar[r]^{\cong} & {\hat{E}_{p,q}^{\infty,r_0}} \\
{E_{p,q}^{\infty,r_0-1}} \ar[r]^{\cong} \ar[u] & {E_{p,q}^{\infty,r_0}} \ar[u]_{\cong} }\]
Somit ist die linke Abbildung auch ein Isomorphismus. Induktiv folgt, dass $\varphi_{p,q}^0$ einen Isomorphismus $E_{p,q}^{\infty,0}\cong\hat{E}_{p,q}^{\infty,0}$ induziert, also dass $\varphi_{p,q}^\infty$ ein Isomorphismus ist.
\end{proof}
Wir erhalten eine weitere Kategorie $\textup{BZ}(k)$ für $k\in\mathbb{Z}$ wie folgt: Ein Objekt besteht aus $R$-Moduln $B_{p,q}^r$ und $Z_{p,q}^r$ für $p,q\in\mathbb{Z}$ und $r\in\mathbb{N}_0$, die
\[\{0\}=B_{p,q}^0\subset B_{p,q}^1\subset\cdots\subset B_{p,q}^r\subset B_{p,q}^{r+1}\subset\cdots\subset Z_{p,q}^{r+1}\subset Z_{p,q}^r\subset\cdots\subset Z_{p,q}^1\subset Z_{p,q}^0\]
erfüllen und zusätzlich aus Isomorphismen $g_{p,q}^r:Z_{p,q}^r/Z_{p,q}^{r+1}\to B_{p-r,q+r+k}^{r+1}/B_{p-r,q+r+k}^r$. Ein Morphismus $(B,Z,g)\to(\hat{B},\hat{Z},\hat{g})$ besteht aus $R$-Modul-Homomorphismen $\psiup_{p,q}:Z_{p,q}^0\to\hat{Z}_{p,q}^0$, der folgendes erfüllt:
\begin{enumerate}
\item 
$\psiup_{p,q}(Z_{p,q}^r)\subset\hat{Z}_{p,q}^r$
\item 
$\psiup_{p,q}(B_{p,q}^r)\subset\hat{B}_{p,q}^r$
\item 
Die von $\psiup$ induzierten Abbildungen kommutieren mit $g$ und $\hat{g}$.
\end{enumerate}
Wir haben beschrieben wie wir einer Spektralsequenz vom Grad $k$ ein Objekt aus $\textup{BZ}(k)$ zuordnen und umgekehrt. Mit den kanonischen Abbildungen zwischen den Morphismenmengen erhalten wir Funktoren $\mathfrak{F}(k):\textup{Spek}(k)\to\textup{BZ}(k)$ und $\mathfrak{G}(k):\textup{BZ}(k)\to\textup{Spek}(k)$. Man kann leicht nachrechnen, dass gilt:
\begin{lemma}
$\mathfrak{G}(k)\circ\mathfrak{F}(k)$ ist natürlich isomorph zu $\textup{id}_{\textup{Spek}(k)}$ und $\mathfrak{F}(k)\circ\mathfrak{G}(k)=\textup{id}_{\textup{BZ}(k)}$. Somit sind $\textup{Spek}(k)$ und $\textup{BZ}(k)$ äquivalente Kategorien.
\end{lemma}
\begin{kor}
Sei $\psiup$ ein Morphismus in der Kategorie $\textup{BZ}(k)$ und seien alle $\psiup_{p,q}$ Isomorphismen von $R$-Moduln, dann ist $\psiup$ ein Isomorphismus.
\end{kor}
\begin{kor}
Sei $\psiup\in\textup{Mor}_{\textup{BZ}(k)}((B,Z,g),(\hat{B},\hat{Z},\hat{g}))$ und sei die induzierte Abbildung $Z_{p,q}^{r_0}/B_{p,q}^{r_0}\to\hat{Z}_{p,q}^{r_0}/\hat{B}_{p,q}^{r_0}$ ein Isomorphismus für ein $r_0\in\mathbb{N}_0$ und für alle $p,q\in\mathbb{Z}$. Dann ist die induzierte Abbildung $Z_{p,q}^{r}/B_{p,q}^{r}\to\hat{Z}_{p,q}^{r}/\hat{B}_{p,q}^{r}$ ein Isomorphismus für alle $r>r_0$ und $p,q\in\mathbb{Z}$. Dies gilt auch für $r=\infty$.
\end{kor}
\begin{defi}
Wir definieren die Kategorie Filt der filtrierten graduierten $R$-Moduln wie folgt: Ein Objekt besteht aus $R$-Moduln $F_iA_j$ für $i,j\in\mathbb{Z}$, so dass $\cdots\subset F_iA_j\cdots\subset F_{i+1}A_j\cdots\subset F_{i+2}A_j\subset\cdots$. Es sei $A_j\defeq\bigcup_{i\in\mathbb{Z}}F_iA_j$, $A\defeq\bigoplus_{j\in\mathbb{Z}}A_j$ und $F_iA\defeq\bigoplus_{j\in\mathbb{Z}}F_iA_j$. Ein Morphismus $\alpha:(A,F)\to(\hat{A},\hat{F})$ besteht aus Homomorphismen $\alpha_j:A_j\to\hat{A}_j$ mit $\alpha_j(F_iA_j)\subset\hat{F}_i\hat{A}_j$.
\end{defi}
\begin{defi}
Die Kategorie $\textup{FKK}(k)$ der filtrierten Kettenkomplexe vom Grad $k\in\mathbb{Z}$ sei wie folgt definiert. Ein Objekt besteht aus $(A,F)\in\textup{Obj(Filt)}$ und Homomorphismen $d_j:A_j\to A_{j+k}$ mit $d_{j+k}\circ d_j=0$ und $d_j(F_iA_j)\subset F_iA_{j+k}$. Es sei $d\defeq\bigoplus_{j\in\mathbb{Z}}d_j:A\to A$. Ein Morphismus $\alpha:(A,F,d)\to(\hat{A},\hat{F},\hat{d})$ ist ein Morphismus in Filt, der zusätzlich noch $\hat{d}_j\circ\alpha_j=\alpha_{j+k}\circ d_j$ erfüllt.
\end{defi}
Nun konstruieren wir einen Funktor $\mathfrak{L}(k):\textup{FKK}(k)\to\textup{BZ}(k)$: Sei $(A,F,d)$ ein filtrierter Kettenkomplex vom Grad $k$. Dann sei $\etaup_{i,j}:F_iA_j\to F_iA_j/F_{i-1}A_j$ die kanonische Projektion und
\[K_{p,q}^r\defeq\{a\in F_pA_{p+q}\ \ |\ \ d_{p+q}(a)\in F_{p-r}A_{p+q+k}\}\]
für $i,j,p,q,r\in\mathbb{Z}$. Definiere für $r\ge 0$:
\begin{align*}
Z_{p,q}^r&\defeq\etaup_{p,p+q}(K_{p,q}^r)\\
B_{p,q}^r&\defeq\etaup_{p,p+q}(d_{p+q-k}(K_{p+r-1,q-r-k+1}^{r-1}))
\end{align*}
Offensichtlich gilt
\[\{0\}=B_{p,q}^0\subset B_{p,q}^1\subset\cdots\subset B_{p,q}^r\subset B_{p,q}^{r+1}\subset\cdots\subset Z_{p,q}^{r+1}\subset Z_{p,q}^r\subset\cdots\subset Z_{p,q}^1\subset Z_{p,q}^0.\]
Es gilt $K_{p,q}^r\cap F_{p-1}A_{p+q}=K_{p-1,q+1}^{r-1}$ und somit $Z_{p,q}^r\cong K_{p,q}^r/K_{p-1,q+1}^{r-1}$. Es folgt
\[Z_{p,q}^r/Z_{p,q}^{r+1}
\cong\frac{K_{p,q}^r/K_{p-1,q+1}^{r-1}}{K_{p,q}^{r+1}/K_{p-1,q+1}^{r}}
\cong\frac{K_{p,q}^r}{K_{p,q}^{r+1}+K_{p-1,q+1}^{r-1}}.\]
Weiterhin gilt $d_{p+q}(K_{p,q}^r)\cap F_{p-r-1}A_{p+q+k}=d_{p+q}(K_{p,q}^{r+1})$ und somit \[B_{p-r,q+r+k}^{r+1}=\etaup_{p-r,p+q+k}(d_{p+q}(K_{p,q}^r))\cong d_{p+q}(K_{p,q}^r)/d_{p+q}(K_{p,q}^{r+1}).\]
Es folgt
\[B_{p-r,q+r+k}^{r+1}/B_{p-r,q+r+k}^{r}
\cong\frac{d_{p+q}(K_{p,q}^r)/d_{p+q}(K_{p,q}^{r+1})}
{d_{p+q}(K_{p-1,q+1}^{r-1})/d_{p+q}(K_{p-1,q+1}^{r})}
\cong\frac{d_{p+q}(K_{p,q}^r)}{d_{p+q}(K_{p,q}^{r+1}+K_{p-1,q+1}^{r-1})}.\]
Wegen $\ker(d_{p+q})\cap K_{p,q}^r\subset K_{p,q}^{r+1}$ erhalten wir einen Isomorphismus $g_{p,q}^r:Z_{p,q}^r/Z_{p,q}^{r+1}\to B_{p-r,q+r+k}^{r+1}/B_{p-r,q+r+k}^r$. Somit haben wir einem Objekt aus $\textup{FKK}(k)$ ein Objekt aus $\textup{BZ}(k)$ zugeordnet. Es gilt $Z_{p,q}^r=F_pA_{p+q}/F_{p-1}A_{p+q}$. Somit induziert ein $\alpha\in\textup{Mor}_{\textup{FKK}(k)}((A,F,d),(\hat{A},\hat{F},\hat{d}))$ $R$-Modul-Homomorphismen $Z_{p,q}^0\to\hat{Z}_{p,q}^0$. Man prüft leicht nach, dass dies einen Morphismus $(B,Z,g)\to(\hat{B},\hat{Z},\hat{g})$ ergibt. Mit diesen Konstruktionen erhält man einen Funktor $\mathfrak{L}(k):\textup{FKK}(k)\to\textup{BZ}(k)$.
\begin{defi}\label{konvdef}
Wir sagen eine Spektralsequenz $(E,d,f)$ vom Grad $k$ konvergiert gegen ein $(A,F)\in\textup{Obj(Filt)}$, falls $E_{p,q}^\infty\cong F_pA_{p+q}/F_{p-1}A_{p+q}$ für alle $p,q$. Wir sagen ein $\varphi\in\textup{Mor}_{\textup{Spek}(k)}((E,d,f),(\hat{E},\hat{d},\hat{f}))$ konvergiert gegen ein $\alpha\in\textup{Mor}_\textup{Filt}((A,F),(\hat{A},\hat{F}))$, falls es Isomorphismen $E_{p,q}^\infty\cong F_pA_{p+q}/F_{p-1}A_{p+q}$ und $\hat{E}_{p,q}^\infty\cong\hat{F}_p\hat{A}_{p+q}/\hat{F}_{p-1}\hat{A}_{p+q}$ gibt, so dass folgendes Diagramm kommutiert
\[\xymatrix{ {E_{p,q}^\infty} \ar[r]^{\cong} \ar[d]_{\varphi_{p,q}^\infty} & {F_pA_{p+q}/F_{p-1}A_{p+q}} \ar[d] \\
{\hat{E}_{p,q}^\infty} \ar[r]^{\cong} & {\hat{F}_p\hat{A}_{p+q}/\hat{F}_{p-1}\hat{A}_{p+q},} }\]
wobei die rechte Abbildung von $\alpha$ induziert sei.
\end{defi}
\begin{lemma}\label{konvlem}
Sei $\varphi$ ein Morphismus in $\textup{Spek}(k)$ und seien alle $\varphi_{p,q}^\infty$ Isomorphismen von $R$-Moduln. Weiterhin konvergiere $\varphi$ gegen $\alpha\in\textup{Mor}_\textup{Filt}((A,F),(\hat{A},\hat{F}))$. Es gebe $s:\mathbb{Z}\to\mathbb{Z}$, so dass $F_{s(j)}A_j=\{0\}$ und $\hat{F}_{s(j)}\hat{A}_j=\{0\}$ für alle $j\in\mathbb{Z}$. Dann ist $\alpha$ ein Isomorphismus.
\end{lemma}
\begin{proof}
Nach Definition von Konvergenz sind die von $\alpha$ induzierten Abbildungen $\beta_{p,q}^1:F_pA_{p+q}/F_{p-1}A_{p+q}
\to\hat{F}_p\hat{A}_{p+q}/\hat{F}_{p-1}\hat{A}_{p+q}$ Isomorphismen von $R$-Moduln. Wir zeigen nun induktiv, dass die von $\alpha$ induzierten $\beta_{p,q}^t:F_pA_{p+q}/F_{p-t}A_{p+q}
\to\hat{F}_p\hat{A}_{p+q}/\hat{F}_{p-t}\hat{A}_{p+q}$ füt $t\in\mathbb{N}$ Isomorphismen sind: Wir haben folgendes kommutatives Diagramm mit exakten Zeilen:
\[\xymatrix{ {0} \ar[r] & {F_{p-1}A_{p+q}/F_{p-1-t}A_{p+q}} \ar[r] \ar[d]^{\beta_{p-1,q+1}^t} & {F_pA_{p+q}/F_{p-(t+1)}A_{p+q}} \ar[r] \ar[d]^{\beta_{p,q}^{t+1}} & {F_pA_{p+q}/F_{p-1}A_{p+q}} \ar[r] \ar[d]^{\beta_{p,q}^1} & {0} \\
{0} \ar[r] & {\hat{F}_{p-1}\hat{A}_{p+q}/\hat{F}_{p-1-t}\hat{A}_{p+q}} \ar[r] & {\hat{F}_p\hat{A}_{p+q}/\hat{F}_{p-(t+1)}\hat{A}_{p+q}} \ar[r] & {\hat{F}_p\hat{A}_{p+q}/\hat{F}_{p-1}\hat{A}_{p+q}} \ar[r] & {0} }\]
Sind $\beta_{p-1,q+1}^t$ und $\beta_{p,q}^1$ Isomorphismen, so ist auch $\beta_{p,q}^{t+1}$ ein Isomorphismus. Somit sind alle $\beta_{p,q}^t$ Isomorphismen. Insbesondere sind auch die $\beta_{p,q}^t:F_pA_{p+q}\to\hat{F}_p\hat{A}_{p+q}$ für $t=\textup{max}(1,p-s(p+q))$ Isomorphismen. Damit ist $\alpha$ ein Isomorphismus.
\end{proof}
\begin{satz}\label{konvsatz}
Sei $\alpha\in\textup{Mor}_{\textup{FKK}(k)}((A,F,d),(\hat{A},\hat{F},\hat{d}))$ gegeben und sei weiterhin $s:\mathbb{Z}\to\mathbb{Z}$ gegeben, so dass $F_{s(j)}A_j=\{0\}$ und $\hat{F}_{s(j)}\hat{A}_j=\{0\}$ für alle $j\in\mathbb{Z}$. Dann konvergiert $\mathfrak{G}(k)\circ\mathfrak{L}(k)(\alpha)$ gegen den Morphismus $H(\alpha)$, der von $\alpha$ in Homologie induziert wird. Dabei erhalten wir eine Filtrierung $G$ auf Homologie wie folgt: Sei $p_j:\ker d_j\to H_j(A,d)$ die kanonische Projektion, dann sei $G_iH_j(A,d)\defeq p_j(F_iA_j\cap\ker d_j)$.
\end{satz}
\begin{proof}
Es gilt $K_{p,q}^r=F_pA_{p+q}\cap\ker d_{p+q}$ für $r\ge p-s(p+q+k)$. Definiere $K_{p,q}^\infty\defeq\bigcap_{r\in\mathbb{Z}}K_{p,q}^r=F_pA_{p+q}\cap\ker d_{p+q}$. Es gilt $Z_{p,q}^\infty=\etaup_{p,p+q}(K_{p,q}^\infty)$. Weiterhin gilt $\bigcup_{r\in\mathbb{Z}}d_{p+q-k}(K_{p+r-1,q-r-k+1}^{r-1})
=F_pA_{p+q}\cap\textup{im}(d_{p+q-k})$ und damit $B_{p,q}^\infty=\etaup_{p,p+q}(F_pA_{p+q}\cap\textup{im}(d_{p+q-k}))$. Außerdem ist $K_{p-1,q+1}^\infty=K_{p,q}^\infty\cap\ker\etaup_{p,p+q}$. Wir erhalten
\begin{align*}
G_pH_{p+q}(A,d)/G_{p-1}H_{p+q}(A,d)
&\cong\frac{K_{p,q}^\infty}{K_{p-1,q+1}^\infty+(F_pA_{p+q}\cap\textup{im}(d_{p+q-k}))}\\
&\cong\frac{\etaup_{p,p+q}(K_{p,q}^\infty)}
{\etaup_{p,p+q}(F_pA_{p+q}\cap\textup{im}(d_{p+q-k}))}
=Z_{p,q}^\infty/B_{p,q}^\infty=E_{p,q}^\infty.
\end{align*}
Entsprechend erhält man Isomorphismen $\hat{G}_pH_{p+q}(\hat{A},\hat{d})/\hat{G}_{p-1}H_{p+q}(\hat{A},\hat{d})
\cong\hat{E}_{p,q}^\infty$. Man prüft leicht nach, dass die von $\mathfrak{G}(k)\circ\mathfrak{L}(k)(\alpha)$ induzierte Abbildung $E_{p,q}^\infty\to\hat{E}_{p,q}^\infty$ unter den gegebenen Isomorphismen mit der von $H_{p+q}(\alpha)$ induzierten Abbildung übereinstimmt.
\end{proof}
\begin{bemmn}
Wir erhalten weitere Kategorien $\textup{Spek}_l(k)$ für $l\in\mathbb{N}_0$, $k\in\mathbb{Z}$, deren Objekte gleich den Objekten von $\textup{Spek}(k)$ sind. Ein Morphismus $\varphi:(E,d,f)\to(\hat{E},\hat{d},\hat{f})$ besteht aus $R$-Modul-Homomorphismen $\varphi_{p,q}^r:E_{p,q}^r\to\hat{E}_{p,q}^r$ für $p,q\in\mathbb{Z}$ und $r\in\mathbb{N}_{\ge l}$, die 1. und 2. aus Definition \ref{defmor} erfüllen. $\varphi$ induziert $R$-Modul-Homomorphismen $\varphi_{p,q}^\infty:E_{p,q}^\infty\to\hat{E}_{p,q}^\infty$, wie der Beweis von Lemma \ref{spekiso} zeigt. Ist $\varphi_{p,q}^{r_0}$ ein Isomorphismus für ein $r_0\ge l$ und für alle $p,q\in\mathbb{Z}$, so sind die $\varphi_{p,q}^\infty$ Isomorphismen. Man kann analog zu Definition \ref{konvdef} Konvergenz definieren und auch die Analogie zu Lemma \ref{konvlem} gilt. Einen Isomorphismus in $\textup{Spek}_l(k)$ nennen wir Isomorphismus ab dem $E^l$-Term.
\end{bemmn}
\begin{defi}
Sei die Katgorie $\textup{P}(k)$ für $k\in\mathbb{Z}$ wie folgt definiert: Ein Objekt besteht aus $R$-Moduln $X=\bigoplus_{i,j\in\mathbb{Z}}X_{i,j}$. Zusätzlich haben wir Homomorphismen $d_{i,j}:X_{i,j}\to X_{i+1,j+k}$ mit $d_{i+1,j+k}\circ d_{i,j}=0$. Ein Morphismus $\gammaup:(X,d)\to(\hat{X},\hat{d})$ besteht aus Homomorphismen $\gammaup_{i,j}:X_{i,j}\to\hat{X}_{i,j}$ mit $\hat{d}_{i,j}\circ\gammaup_{i,j}=\gammaup_{i+1,j+k}\circ d_{i,j}$.
\end{defi}
\begin{defi}
Die Kategorie $Q(k)$ für $k\in\mathbb{Z}$ sei wie folgt definiert: Ein Objekt besteht aus $R$-Moduln $X=\bigoplus_{i,j\in\mathbb{Z}}X_{i,j}$. Zusätzlich haben wir einen Homomorphismus $d:X\to X$ mit $d\circ d=0$. Wir verlangen $d(X_{i,j})\subset\bigoplus_{l>i}X_{l,j+k}$. Ein Morphismus $(X,d)\to(\hat{X},\hat{d})$ ist ein Homomorphismus $\gammaup:X\to\hat{X}$ mit $\hat{d}\circ\gammaup=\gammaup\circ d$. Wir verlangen $\gammaup(X_{i,j})\subset\bigoplus_{l\ge i}\hat{X}_{l,j}$.
\end{defi}
Wir erhalten einen Funktor $\mathfrak{M}(k):\textup{Q}(k)\to\textup{P}(k)$ wie folgt: Sei $(X,d)\in\textup{Ob}(\textup{Q}(k))$. Sei $\textup{pr}_{i,j}:X\to X_{i,j}$ die kanonische Projektion. Definiere $e_{i,j}\defeq\textup{pr}_{i+1,j+k}\circ d|_{X_{i,j}}:X_{i,j}\to X_{i+1,j+k}$. Dann gilt
\[e_{i+1,j+k}\circ e_{i,j}=\textup{pr}_{i+2,j+2+k}\circ d|_{X_{i+1,j+k}}\circ\textup{pr}_{i+1,j+k}\circ d|_{X_{i,j}}=\textup{pr}_{i+2,j+2+k}\circ d\circ d|_{X_{i,j}}=0.\]
Es sei $\mathfrak{M}(k)(X,d)\defeq(X,e)\in\textup{Ob}(\textup{P}(k))$. Sei nun $\gammaup\in\textup{Mor}_{\textup{Q}(k)}((X,d),(\hat{X},\hat{d}))$. Definiere $\delta_{i,j}\defeq\hat{\textup{pr}}_{i,j}\circ\gammaup|_{X_{i,j}}:
X_{i,j}\to\hat{X}_{i,j}$. Dann gilt
\begin{align*}
\hat{e}_{i,j}\circ\delta_{i,j}
&=\hat{\textup{pr}}_{i+1,j+k}\circ\hat{d}|_{\hat{X}_{i,j}}
\circ\hat{\textup{pr}}_{i,j}\circ\gammaup|_{X_{i,j}}
=\hat{\textup{pr}}_{i+1,j+k}\circ\hat{d}\circ\gammaup|_{X_{i,j}}
=\hat{\textup{pr}}_{i+1,j+k}\circ\gammaup\circ d|_{X_{i,j}}\\
&=\hat{\textup{pr}}_{i+1,j+k}\circ\gammaup|_{X_{i+1,j+k}}
\circ\textup{pr}_{i+1,j+k}\circ d|_{X_{i,j}}
=\delta_{i+1,j+k}\circ e_{i,j}.
\end{align*}
Es sei $\mathfrak{M}(k)(\gammaup)
\defeq\delta\in\textup{Mor}_{\textup{P}(k)}((X,e),(\hat{X},\hat{e}))$. Man sieht leicht, dass $\mathfrak{M}(k)$ ein Funktor ist.\\
Einen weiteren Funktor $\mathfrak{N}(k):\textup{Q}(k)\to\textup{FKK}(k)$ erhalten wir wie folgt: Sei $(X,d)\in\textup{Ob}(\textup{Q}(k))$. Definiere $F_iA_j\defeq\bigoplus_{l\ge -i}X_{l,j}$. Es gilt $A_j=\bigcup_{i\in\mathbb{Z}}F_iA_j=\bigoplus_{l\in\mathbb{Z}}X_{l,j}$ und $A=\bigoplus_{j\in\mathbb{Z}}A_j=X$. Die Abbildung $d:A\to A$ erfüllt $d\circ d=0$ und $d(F_iA_j)\subset F_{i-1}A_{j+k}$. Es sei $\mathfrak{N}(k)(X,d)\defeq(A,F,d)\in\textup{Ob}(\textup{FKK}(k))$. Sei nun $\gammaup\in\textup{Mor}_{\textup{Q}(k)}((X,d),(\hat{X},\hat{d}))$. Es gilt $\gammaup(F_iA_j)\subset\gammaup\hat{F}_i\hat{A}_j$ und $\hat{d}\circ\gammaup=\gammaup\circ d$. Es sei $\mathfrak{N}(k)(\gammaup)
\defeq\gammaup\in\textup{Mor}_{\textup{FKK}(k)}((A,F,d),(\hat{A},\hat{F},\hat{d}))$. Offensichtlich ist $\mathfrak{N}(k)$ ein Funktor.
\begin{defi}\label{chi}
Sei $\chiup:\textup{BG-$R$-Mod}\to\textup{BG-$R$-Mod}$ der Funktor, der der Familie $(A_{i,j})_{i,j\in\mathbb{Z}}$ die Familie $(A_{-i,i+j})_{i,j\in\mathbb{Z}}$ zuordnet.
\end{defi}
\begin{defi}
Sei $\mathfrak{H}(k):\textup{P}(k)\to\textup{BG-$R$-Mod}$ der Homologiefunktor.
\end{defi}
\begin{satz}\label{e2}
Die beiden Funktoren $\chiup\circ\mathfrak{H}(k)\circ\mathfrak{M}(k)$ und ${\mathfrak{T}^2(k)\circ\mathfrak{G}(k)\circ\mathfrak{L}(k)\circ\mathfrak{N}(k)}$ sind natürlich isomorph.
\end{satz}
\begin{proof}
Sei $(X,d)\in\textup{Ob}(\textup{Q}(k))$ und sei $(X,e)\defeq\mathfrak{M}(k)(X,d)$ und $(A,F,d)\defeq\mathfrak{N}(k)(X,d)$. Weiterhin sei $(B,Z,g)\defeq\mathfrak{L}(k)(A,F,d)$. Es gilt
\[Z_{p,q}^2/B_{p,q}^2=\frac{\{[a]\in F_pA_{p+q}/F_{p-1}A_{p+q}\ \ |\ \ d_{p+q}(a)\in F_{p-2}A_{p+q+k}\}}{\{[a]\in F_pA_{p+q}/F_{p-1}A_{p+q}\ \ |\ \ a\in d_{p+q-k}(F_{p+1}A_{p+q-k})\}}\cong\frac{\{x\in X_{-p,p+q}\ \ |\ \ d(x)\in\bigoplus_{l\ge-p+2}X_{l,p+q+k}\}}{\{x\in X_{-p,p+q}\ \ |\ \ x\in\textup{pr}_{-p,p+q}\circ d(X_{-p-1,p+q-k})\}}\]
und
\[(\mathfrak{H}(k)(X,e))_{p,q}=\frac{\{x\in X_{p,q}\ \ |\ \ e_{p,q}(x)=0\}}{\{x\in X_{p,q}\ \ |\ \ x\in e_{p-1,q-k}(X_{p-1,q-k})\}}=\frac{\{x\in X_{p,q}\ \ |\ \ \textup{pr}_{p+1,q+k}\circ d(x)=0\}}{\{x\in X_{p,q}\ \ |\ \ x\in\textup{pr}_{p,q}\circ d(X_{p-1,q-k})\}}.\]
Wir erhalten also Isomorphismen $(\mathfrak{H}(k)(X,e))_{-p,p+q}\cong Z_{p,q}^2/B_{p,q}^2$. Ist $\gammaup\in\textup{Mor}_{\textup{Q}(k)}((X,d),(\hat{X},\hat{d}))$, so sieht man leicht, dass $\chiup\circ\mathfrak{H}(k)\circ\mathfrak{M}(k)(\gammaup)$ und $\mathfrak{T}^2(k)\circ\mathfrak{G}(k)\circ\mathfrak{L}(k)
\circ\mathfrak{N}(k)(\gammaup)$ unter den gegebenen Isomorphismen übereinstimmen.
\end{proof}
\begin{defi}
Sei P-Quasi-Isomorphie $\sim_\textup{P}$ die Äquivalenzrelation auf $\textup{Ob}(\textup{Q}(k))$, die von Folgendem erzeugt wird: $(X,d)\sim_\textup{P}(\hat{X},\hat{d})$, falls es $\gammaup\in\textup{Mor}_{\textup{Q}(k)}((X,d),(\hat{X},\hat{d}))$ und $s:\mathbb{Z}\to\mathbb{Z}$ gibt, so dass $\mathfrak{H}(k)\circ\mathfrak{M}(k)(\gammaup)$ ein Isomorphismus ist und für alle $j\in\mathbb{Z}$ die Gleichung $\bigoplus_{l\ge-s(j)}X_{l,j}=0\bigoplus_{l\ge-s(j)}\hat{X}_{l,j}$ gilt.
\end{defi}
\begin{kor}\label{pqi}
Seien $(X,d)\sim_\textup{P}(\hat{X},\hat{d})$. Dann sind $\mathfrak{G}(k)\circ\mathfrak{L}(k)\circ\mathfrak{N}(k)(X,d)$ und $\mathfrak{G}(k)\circ\mathfrak{L}(k)\circ\mathfrak{N}(k)(\hat{X},\hat{d})$ isomorph ab dem $E^2$-Term. Weiterhin sind die Homologiegruppen $\textup{H}(\mathfrak{N}(k)(X,d))$ und $\textup{H}(\mathfrak{N}(k)(\hat{X},\hat{d}))$ isomorph als filtrierte graduierte $R$-Moduln.
\end{kor}
\begin{proof}
Satz \ref{e2}, Satz \ref{konvsatz}, Lemma \ref{konvlem}
\end{proof}
\begin{lemma}
Sei $R=\mathbb{Z}$ und sei $A$ abelsche Gruppe. Dann folgt aus $(X,d)\sim_\textup{P}(\hat{X},\hat{d})$, dass $(X,d)\otimes_\mathbb{Z}A\sim_\textup{P}(\hat{X},\hat{d})\otimes_\mathbb{Z}A$.
\end{lemma}
\begin{proof}
Universelles Koeffizienten-Theorem
\end{proof}

\section{Eine Spektralsequenz in ungerader Khovanov-Homologie}
\begin{defi}
Für $n\le k\in\mathbb{N}_0$ sei ein Kantenzug der Länge $n$ ein $n$-Tupel $\theta=({}_1\theta,\dots,{}_n\theta)\in(\sum(k,1))^n$, so dass ${}_i\theta^1={}_{i+1}\theta^0$ für $i=1,\dots,n-1$. Wir sagen, $\theta$ führt von ${}_1\theta^0$ nach ${}_n\theta^1$.
\end{defi}
\begin{defi}
Für eine Kantenzuordnung $s$ und einen Kantenzug $\theta$ der Länge $n$ sei $s(\theta)\defeq\prod_{i=1}^ns({}_i\theta)$.
\end{defi}
\begin{defi}
Für eine (orientierte) $k$-dimensionale Konfiguration $C$ und $\eps\in\sum(k,0)$ sei $\textup{sp}(C,\eps)\defeq\frac{|\gl|-|C|+|\eps|}{2}$.
\end{defi}
\begin{lemma}\label{splem}
Sei $C$ eine (orientierte) $k$-dimensionale Konfiguration, $\eps\in\sum(k,0)$ und $\theta$ ein Kantenzug von $(0,\dots,0)$ nach $\eps$. Dann gilt $\textup{sp}(C,\eps)=\sum_{i\in\{1,\dots,|\eps|\}\textup{ und $G(C,{}_i\theta)$ ist Spalt}}1$.
\end{lemma}
\begin{proof}
Induktion über $|\eps|$.
\end{proof}
\begin{defi}
Sei $C$ eine (orientierte) $k$-dimensionale Konfiguration und $\theta$ ein Kantenzug der Länge $n$, dann sei $\textup{sp}(C,\theta)\defeq\prod_{i=1}^{n-1}(-1)^{\textup{sp}(C,{}_i\theta^1)}
=(-1)^{\sum_{i=1}^{n-1}\textup{sp}(C,{}_i\theta^1)}$.
\end{defi}
Mit Lemma \ref{splem} sieht man
\[\sum_{i=1}^{n-1}\textup{sp}(C,{}_i\theta^1)\equiv\sum_{i\in\{1,\dots,n-1\}\textup{, $i$ ist ungerade und $G(C,{}_{n-i}\theta)$ ist Spalt}}1+(n+1)\textup{sp}(C,{}_1\theta^0)\qquad(\textup{mod }2).\]
Sei $C$ eine (orientierte) $k$-dimensionale Konfiguration. Dann erhalten wir zueinander inverse Bijektionen der Menge der Ordnungen der Bögen von $C$ und der Menge der Kantenzüge der Länge $k$ wie folgt. Wir benutzen dabei, dass durch $C$ bereits eine feste Ordnung $\gammaup_1,\dots,\gammaup_k$ der Bögen vorgegeben ist. Dann entsprechen die Ordnungen der Bögen den Permutationen der Menge $\{1,\dots,k\}$. Für $\sigma\in\textup{S}_n$ erhalten wir nämlich die Ordnung $\gammaup_{\sigma(1)},\dots,\gammaup_{\sigma(k)}$. Wir ordnen $\sigma$ nun einen Kantenzug $\psi(\sigma)=({}_1\psi(\sigma),\dots,{}_k\psi(\sigma))$ zu. Für $i,j\in\{1,\dots,k\}$ sei
\[{}_i\psi(\sigma)_j\defeq\begin{cases}
*\qquad&\textup{falls $\sigma(i)=j$,}\\
1\qquad&\textup{falls $\sigma^{-1}(j)<i$,}\\
0\qquad&\textup{falls $\sigma^{-1}(j)>i$.}
\end{cases}\]
Für $\eps\in\sum(k,1)$ sei $\overline{\eps}\in\{1,\dots,k\}$ die Stelle an der $*$ in $\eps$ steht. Wir ordnen dem Kantenzug $\theta=({}_1\theta,\dots,{}_k\theta)$ dann $\phi(\theta)\in\textup{S}_n$ mit $\phi(\theta)(i)\defeq{}_i\overline{\theta}$ zu. Die Abbildungen $\psi$ und $\phi$ sind invers zueinander.\\
\begin{defi}
Für einen Bogen $\gammaup$ in einer orientierten Konfiguration bezeichnen wir mit $\gammaup^0$ und $\gammaup^1$ den Anfangs- und Endpunkt von $\gammaup$.
\end{defi}
\begin{defi}
Wir nennen eine aktive orientierte Konfiguration $C$ vom Typ
\begin{description}
\item[-]
$\textup{A}_k$ für $k\in\mathbb{N}$, falls $C$ genau zwei Kreise $x_1$ und $x_2$ und genau $k$ Bögen hat, wobei alle Bögen von $x_1$ nach $x_2$ zeigen.
\item[-]
$\textup{B}_k$ für $k\in\mathbb{N}$, falls $C$ genau $k$ Kreise $x_0,\dots,x_{k-1}$ und genau $k$ Bögen $\gammaup_0,\dots,\gammaup_{k-1}$ hat, wobei $\gammaup_i$ von $x_{(i-1)\textup{mod }k}$ nach $x_i$ zeigt.
\item[-]
$\textup{C}_{p,q}$ für $p\le q\in\mathbb{N}$, falls $C$ genau einen Kreis $x$ hat und man auf $x$ eine Orietierung als Uhrzeigersinn wählen kann, wodurch man auch Inneres und Äußeres von $x$ festlegt, so dass Folgendes gilt. $C$ hat genau $p$ Bögen $\gammaup_1,\dots,\gammaup_{p}$ im Innern von $x$ und genau $q$ Bögen $\delta_1,\dots,\delta_{q}$ im Äußeren von $x$. Wenn man bei $\gammaup_1^0$ beginnt und $x$ im Uhrzeigersinn entlang läuft, erreicht man die Anfangs- und Endpunkte der Bögen in der Reihenfolge $\gammaup_1^0,\dots,\gammaup_p^0,\delta_1^0,\dots,\delta_q^0,\gammaup_p^1,\dots
\gammaup_1^1,\delta_q^1,\dots\delta_1^1$.
\item[-]
$\textup{D}_{p,q}$ für $p\le q\in\mathbb{N}$, falls Folgendes gilt. Es gibt genau einen Kreis $z$ von $C$, auf dem zwei Anfangs- und zwei Endpunkte von Bögen liegen. Man kann auf $z$ eine Orientierung als Uhrzeigersinn wählen, so dass innerhalb von $z$ die Kreise $x_1,\dots,x_{p-1}$ und außerhalb von $z$ die Kreise $y_1,\dots,y_{q-1}$ liegen. Die Bögen von $C$ seien $\gammaup_1,\dots,\gammaup_{p},\delta_1,\dots,\delta_{q}$ und $\gammaup_i$ zeige von $x_{i-1}$ nach $x_i$ für $i=2,\dots,p-1$, $\gammaup_1$ zeige von $z$ nach $x_1$, $\gammaup_p$ zeige von $x_{p-1}$ nach $z$. $\delta_i$ zeige von $y_{i-1}$ nach $y_i$ für $i=2,\dots,q-1$, $\delta_1$ zeige von $z$ nach $y_1$, $\delta_q$ zeige von $y_{q-1}$ nach $z$. Wenn man $z$ ab $\gammaup_1^0$ entlang läuft erreicht man im Uhrzeigersinn $\gammaup_1^0,\delta_1^0,\gammaup_p^1,\delta_q^1$.
\item[-]
$\textup{F}_{p,q}$ für $p,q\in\mathbb{N}_0$ mit $p+q\ge 1$, falls Folgendes gilt. Es gibt genau einen Kreis $y$ von $C$, so dass alle Bögen, die auf $y$ starten, auch auf $y$ enden. Diese Bögen seien $\delta_1,\dots,\delta_{q}$. Wählt man auf $y$ eine Orientierung als Uhrzeigersinn, so gilt für alle $\delta_i$: Liegt $\delta_i$ im Innern von $y$, so gilt: Läuft man $y$ ab $\delta_i^0$ im Uhrzeigersinn entlang, so ist der erste Start- oder Endpunkt eines Bogens, den man erreicht, $\delta_i^1$. Liegt $\delta_i$ im Äußeren von $y$, so gilt das gleiche, wobei im Uhrzeigersinn durch gegen den Uhrzeigersinn ersetzt sei. Außer $y$ habe $C$ noch die Kreise $x_1,\dots,x_p$ und außer den $\delta_i$ noch die Bögen $\gammaup_1,\dots,\gammaup_{p}$, wobei $\gammaup_i$ von $x_i$ nach $y$ zeige.
\item[-]
$\textup{G}_{p,q}$ für $p,q\in\mathbb{N}_0$ mit $p+q\ge 1$, falls $r(C)$ vom Typ $\textup{F}_{p,q}$ ist.
\end{description}
\end{defi}
Beispiele für die Typen sind in Abbildung \ref{typ} dargestellt.
\begin{figure}[htbp]
\centering
\includegraphics[scale=0.4]{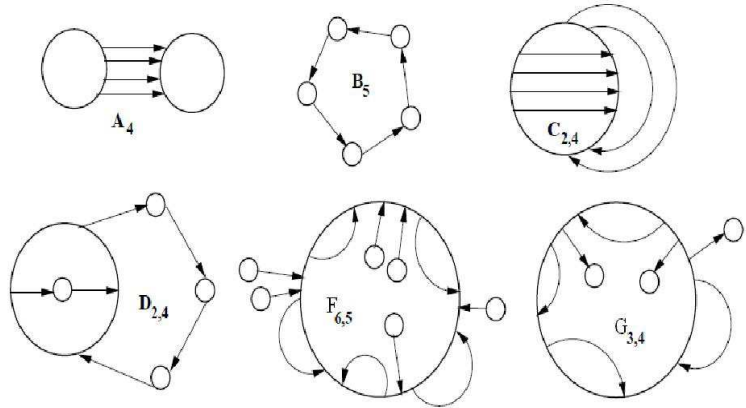}
\caption{Beispiele für die Typen (Abbildung aus \cite{szabo})}
\label{typ}
\end{figure}
Sei
\[\textup{T}\defeq\{\textup{A}_k,\ \textup{B}_k,\ \textup{C}_{p,q},\ \textup{D}_{p,q},\ \textup{F}_{r,s},\ \textup{G}_{r,s}\ \ |\ \ k,p,q\in\mathbb{N}\ ,p\le q,\ r,s\in\mathbb{N}_0,\ r+s\ge 1\}.\]
Ist $C$ nicht zusammenhängend, so gibt es kein $\tau\in\textup{T}$, so dass $C$ vom Typ $\tau$ ist.\\
Ist $C$ mindestens dreidimensional, so gibt es höchstens ein $\tau\in\textup{T}$, so dass $C$ vom Typ $\tau$ ist.\\
Ist $C$ zweidimensional, so gibt es die Fälle aus den Abbildungen \ref{zweidim1} und \ref{zweidim2}. Fall 28 ist vom Typ $\textup{A}_2$, Fall 29 ist vom Typ $\textup{B}_2$, Fall 45 ist vom Typ $\textup{C}_{1,1}$ und vom Typ $\textup{D}_{1,1}$. 17 und 20 sind vom Typ $\textup{F}_{2,0}$, 18 und 21 sind vom Typ $\textup{G}_{2,0}$, 32 und 37 sind vom Typ $\textup{F}_{1,1}$, 31 und 34 sind vom Typ $\textup{G}_{1,1}$, 39 und 43 sind vom Typ $\textup{F}_{0,2}$, 40 und 41 sind vom Typ $\textup{G}_{0,2}$. Alle anderen Fälle sind vom keinem Typ aus T.\\
Ist $C$ eindimensional, so gibt es zwei Möglichkeiten. Ist $C$ Fusion, so ist es vom Typ $\textup{A}_1$, $\textup{F}_{1,0}$ und $\textup{G}_{1,0}$. Ist $C$ Spaltung, so ist es vom Typ $\textup{B}_1$, $\textup{F}_{0,1}$ und $\textup{G}_{0,1}$.\\
Für jedes $\tau\in\textup{T}$ gibt es ein $C$ vom Typ $\tau$. Für $\tau=\textup{A}_k,\ \textup{B}_k,\ \textup{C}_{p,q},\ \textup{D}_{p,q}$ gibt es genau ein $C$ (bis auf die Ordnungen auf den Bögen und Kanten) vom Typ $\tau$.\\
$C$ ist genau dann vom Typ $\textup{A}_k,\ \textup{C}_{p,q},\ \textup{F}_{p,q}$, wenn $m(C^*)$ vom Typ $\textup{B}_k,\ \textup{D}_{p,q},\ \textup{G}_{q,p}$ ist. Definiere für $\tau\in\textup{T}$ also
\[m(\tau^*)\defeq\begin{cases}
\textup{B}_k\qquad&\textup{für $\tau=\textup{A}_k$,}\\
\textup{A}_k\qquad&\textup{für $\tau=\textup{B}_k$,}\\
\textup{D}_{p,q}\qquad&\textup{für $\tau=\textup{C}_{p,q}$,}\\
\textup{C}_{p,q}\qquad&\textup{für $\tau=\textup{D}_{p,q}$,}\\
\textup{G}_{q,p}\qquad&\textup{für $\tau=\textup{F}_{p,q}$,}\\
\textup{F}_{q,p}\qquad&\textup{für $\tau=\textup{G}_{p,q}$.}
\end{cases}\]
Eine orientierte Konfiguration sei vom Typ $\tau\in\textup{T}$, falls ihr aktiver Teil vom Typ $\tau$ ist.\\
Sei $C$ eine aktive orientierte $k$-dimensionale Konfiguration vom Typ $\tau\in\textup{T}$ und $s$ eine Kantenzuordnung vom Typ Y bezüglich $C$. Dann möchten wir den $\mathbb{Z}$-Modul-Homomorphismus $d_{C,\tau,s}:\Lambda\V(C)\to\Lambda\V(C^*)$ definieren. Auf $\Lambda\V(C)$ gibt es, bis auf Vorzeichen, eine kanonische Basis. Genau eines dieser Basiselemente soll nicht trivial abgebildet werden. Wir werden im Folgenden beschreiben, welches Element das ist und wie es abgebildet wird. Dazu wählen wir einen zulässigen Kantenzug $\theta$. Für $\tau\ne\textup{D}_{p,q}$ ist jeder Kantenzug der Länge $k$ (mit Kanten in $\sum(k,1)$) zulässig, für $\tau=\textup{D}_{p,q}$ werden wir später beschreiben, welche Kantenzüge der Länge $k$ zulässig sind. Dann definieren wir $x_{C,\tau,\theta}\in\Lambda\V(C)$ und $y_{C,\tau}\in\Lambda\V(C^*)$. Es sei
\[d_{C,\tau,s}(x_{C,\tau,\theta})=s(\theta)\textup{sp}(C,\theta)y_{C,\tau}.\]
Wir werden zeigen, dass dies unabhängig von der Wahl von $\theta$ ist.\\
Sei $\tau=\textup{A}_k$. Dann seien
\begin{align*}
x_{C,\textup{A}_k,\theta}&=1,\\
y_{C,\textup{A}_k}&=(-1)^{k+1}.
\end{align*}
Wir zeigen nun, dass $d_{C,\tau,s}$ unabhängig von $\theta$ ist. Sei also $\etaup$ ein weiterer Kantenzug der Länge $k$. Wir müssen nur den Fall $\phi(\etaup)=\phi(\theta)\circ{(i\ i+1)}$ für $i\in\{1,\dots,k-1\}$ betrachten, da die $\textup{S}_n$ von solchen Transpositionen erzeugt wird. In diesem Fall unterscheiden sich $({}_1\etaup,\dots,{}_k\etaup)$ und $({}_1\theta,\dots,{}_k\theta)$ nur in der $i$-ten und $(i+1)$-ten Kante. Sei
\[\eps_j\defeq\begin{cases}
*\qquad&\textup{für $j={}_i\overline{\theta}={}_{i+1}\overline{\etaup}$ oder $j={}_{i+1}\overline{\theta}={}_{i}\overline{\etaup}$,}\\
{}_i\etaup_j={}_{i+1}\etaup_j={}_i\theta_j={}_{i+1}\theta_j\qquad&\textup{sonst.}
\end{cases}\]
Dann ist $s(\theta)s(\etaup)=1$, falls $\gl$ vom Typ A oder Y ist und $s(\theta)s(\etaup)=-1$, falls $\gl$ vom Typ K oder X ist, da $s$ eine Kantenzuordnung vom Typ Y ist. Weil $C$ vom Typ $\textup{A}_k$ ist, ist $\gl$ Fall 22, 28 oder 38 aus den Abbildungen \ref{zweidim1} und \ref{zweidim2}. Es folgt $s(\theta)=s(\etaup)$ und $\textup{sp}(C,\theta)=\textup{sp}(C,\etaup)$.\\
Sei $\tau=\textup{B}_k$. Wir haben durch $C$ eine Ordnung $\gammaup_1,\dots,\gammaup_k$ der Bögen gegeben. Die Kreise von $C$ seien so mit $x_1,\dots,x_k$ bezeichnet, dass $\gammaup_i^1$ auf $x_i$ liegt. $C^*$ ist vom Typ $\textup{A}_k$, die Kreise von $C^*$ seien so mit $y_1$, $y_2$ bezeichnet, dass alle Bögen von $y_1$ nach $y_2$ zeigen. Weiterhin sei $\theta$ ein Kantenzug der Länge $k$. Dann seien
\begin{align*}
x_{C,\textup{B}_k,\theta}&=x_{\phi(\theta)(1)}\wedge\cdots\wedge x_{\phi(\theta)(k)},\\
y_{C,\textup{B}_k}&=y_1\wedge y_2.
\end{align*}
Verhalte sich $\etaup$ zu $\theta$ wie für $\tau=\textup{A}_k$ beschrieben und sei $\eps$ auf die gleiche Weise definiert. Es gilt $x_{C,\textup{B}_k,\theta}=-x_{C,\textup{B}_k,\etaup}$. $\gl$ ist Fall 1, 16 oder 29 aus den Abbildungen \ref{zweidim1} und \ref{zweidim2}. Es folgt $s(\theta)=-s(\etaup)$ und $\textup{sp}(C,\theta)=\textup{sp}(C,\etaup)$.\\
Sei $\tau=\textup{C}_{p,q}$. Dann seien
\begin{align*}
x_{C,\textup{C}_{p,q},\theta}&=1,\\
y_{C,\textup{C}_{p,q}}&=(-1)^k.
\end{align*}
Seien $\etaup$ und $\eps$ wie beschrieben. $\gl$ ist Fall 7,22,23,24,28,31,32,38,41,42,43 oder 45. In den Fällen 7, 31, 32 ist $s(\theta)=-s(\etaup)$ und $\textup{sp}(C,\theta)=-\textup{sp}(C,\etaup)$. In den Fällen 22,23,24,28,38,41,42,43,45 ist $s(\theta)=s(\etaup)$ und $\textup{sp}(C,\theta)=\textup{sp}(C,\etaup)$.\\
Sei $\tau=\textup{D}_{p,q}$ und sei $x$ der Kreis von $C$, auf dem zwei Anfangs- und zwei Endpunkte von Bögen liegen. Wir haben durch $C$ eine Ordnung $\gammaup_1,\dots,\gammaup_k$ der Bögen gegeben. Seien $\gammaup_a$, $\gammaup_b$ die beiden Bögen, die auf $x$ zeigen. Für $i\in\{1,\dots,k\}\setminus\{a,b\}$ sei $x_i$ der Kreis, auf den $\gammaup_i$ zeigt. Der Kreis von $C^*$ sei mit $y$ bezeichnet. Ein Kantenzug der Länge $k$ (mit Kanten in $\sum(k,1)$) sei zulässig, falls $\{\phi(\theta)(k-1),\ \phi(\theta)(k)\}=\{a,\ b\}$. Dann sei
\begin{align*}
x_{C,\textup{D}_{p,q},\theta}&=x_{\phi(\theta)(1)}\wedge\cdots\wedge x_{\phi(\theta)(k-2)}\wedge x,\\
y_{C,\textup{D}_{p,q}}&=y.
\end{align*}
Sei zunächst $\etaup$ Kantenzug der Länge $k$, so dass $\phi(\etaup)=\phi(\theta)\circ{(k-1\ k)}$. Sei $\eps$ wie beschrieben. Dann ist $\gl$ Fall 45. Es folgt $s(\theta)=s(\etaup)$ und $\textup{sp}(C,\theta)=\textup{sp}(C,\etaup)$. Sei nun $\etaup$ Kantenzug der Länge $k$, so dass $\phi(\etaup)=\phi(\theta)\circ{(i\ i+1)}$ für $i\in\{1,\dots,k-3\}$. Dann ist $x_{C,\textup{D}_{p,q},\theta}=-x_{C,\textup{D}_{p,q},\etaup}$. Sei $\eps$ wie beschrieben. $\gl$ ist Fall 1,2,16 oder 21. Es folgt $s(\theta)=-s(\etaup)$ und $\textup{sp}(C,\theta)=\textup{sp}(C,\etaup)$.\\
Sei $\tau=\textup{F}_{p,q}$. Wir haben durch $C$ eine Ordnung $\gamma_1,\dots,\gammaup_k$ der Bögen gegeben. Für $i=1,\dots,k$ sei $x_i\defeq 1\in\Lambda\V(C)$, falls $\gammaup_i^0$ und $\gammaup_i^1$ auf dem gleichen Kreis liegen. Andernfalls sei $x_i$ der Kreis, auf dem $\gammaup_i^0$ liegt. In $C^*$ gibt es genau einen Kreis $y$, so dass alle Bögen, die auf $y$ enden, auch auf $y$ starten. Dann sei
\begin{align*}
x_{C,\textup{F}_{p,q},\theta}&=x_{\phi(\theta)(1)}\wedge\cdots\wedge x_{\phi(\theta)(k)},\\
y_{C,\textup{F}_{p,q}}&=y.
\end{align*}
Sei $\etaup$ Kantenzug der Länge $k$, so dass $\phi(\etaup)=\phi(\theta)\circ{(i\ i+1)}$ für $i\in\{1,\dots,k-1\}$. Sei $\eps$ wie beschrieben. Betrachte zunächst den Fall, dass $x_{\phi(\theta)(i)}=1$ oder $x_{\phi(\theta)(i+1)}=1$. Dann gilt $x_{C,\textup{F}_{p,q},\theta}=x_{C,\textup{F}_{p,q},\etaup}$. $\gl$ ist Fall 32, 37, 39 oder 43. In den Fällen 32, 37 ist $s(\theta)=-s(\etaup)$ und $\textup{sp}(C,\theta)=-\textup{sp}(C,\etaup)$. In den Fällen 39, 43 ist $s(\theta)=s(\etaup)$ und $\textup{sp}(C,\theta)=\textup{sp}(C,\etaup)$. Betrachte nun den Fall, dass $x_{\phi(\theta)(i)}\ne 1\ne x_{\phi(\theta)(i+1)}$. Dann gilt $x_{C,\textup{F}_{p,q},\theta}=-x_{C,\textup{F}_{p,q},\etaup}$. $\gl$ ist Fall 17 oder 20. Es folgt $s(\theta)=-s(\etaup)$ und $\textup{sp}(C,\theta)=\textup{sp}(C,\etaup)$.\\
Sei $\tau=\textup{G}_{p,q}$. Wir haben durch $C$ eine Ordnung $\gamma_1,\dots,\gammaup_k$ der Bögen gegeben. Für $i=1,\dots,k$ sei $x_i\defeq 1\in\Lambda\V(C)$, falls $\gammaup_i^0$ und $\gammaup_i^1$ auf dem gleichen Kreis liegen. Andernfalls sei $x_i$ der Kreis, auf dem $\gammaup_i^1$ liegt. In $C^*$ gibt es genau einen Kreis $y$, so dass alle Bögen, die auf $y$ starten, auch auf $y$ enden. Dann sei
\begin{align*}
x_{C,\textup{G}_{p,q},\theta}&=x_{\phi(\theta)(1)}\wedge\cdots\wedge x_{\phi(\theta)(k)},\\
y_{C,\textup{G}_{p,q}}&=(-1)^{p+1}y.
\end{align*}
Dass $d_{C,\textup{G}_{p,q},s}$ unabhängig von $\theta$ ist, zeigt man analog zum Fall $\tau=\textup{F}_{p,q}$.\\
Für $x\in\Lambda^m\V(C)$ schreiben wir $\textup{gr}(x)=m$. Dann gilt
\begin{equation}\label{grgl}
\textup{gr}(y_{C,\tau})-\textup{gr}(x_{C,\tau,\theta})=\frac{|C^*|-|C|-k}{2}+1
=\textup{sp}(C,(1,\dots,1))-k+1.
\end{equation}
Als Nächstes möchten wir $d_{C,\tau,s}$ auch für den Fall definieren, dass $C$ nicht unbedingt aktiv ist. Für ein beliebiges $n\in\mathbb{N}_0$ seien $z_1,\dots,z_n$ beliebige passive Kreise von $C$ und $\omega\defeq z_1\wedge\cdots\wedge z_n$. Weiterhin sei $\theta$ ein für $\textup{akt}(C)$ zulässiger Kantenzug, wir nennen $\theta$ dann zulässig für $C$. Nun sei
\[d_{C,\tau,s}(x_{\textup{akt}(C),\tau,\theta}\wedge\omega)
=s(\theta)\textup{sp}(C,\theta)y_{\textup{akt}(C),\tau}\wedge\omega.\]
Alle kanonischen Basiselemente von $\Lambda\V(C)$, die sich nicht bis auf ein Vorzeichen in der Form $x_{\textup{akt}(C),\tau,\theta}\wedge\omega$, mit einem $\omega$ wie beschrieben, darstellen lassen, werden trivial abgebildet.\\
Es sei
\[d_{C,s}\defeq\sum_{\tau\in\textup{T},\textup{ $C$ ist vom Typ $\tau$}}d_{C,\tau,s}.\]
Ist $C$ eindimensional, so ist $d_{C,s}$ gleich dem ungeraden Khovanov-Differential $\partial(C,s)=s(*)\partial_C$.\\
Eine Kantenzuordnung $s:\sum(k,1)\to\{-1,1\}$ induziert für $\eps\in\sum(k,n)$, $1\le n\le k$, eine Kantenzuordnung $s|_\eps:\sum(n,1)\to\{-1,1\}$ wie folgt. Seien $i_1<i_2<\cdots<i_n$ so, dass $\eps_{i_j}=*$ für $j=1,\dots,n$. Für $\lambda\in\sum(n,1)$ sei $\mu\in\sum(k,1)$ mit $\mu_l\defeq\eps_l$ falls $\eps_l\ne*$ und $\mu_{i_j}\defeq\lambda_j$ für $j=1,\dots,n$. Dann sei $s|_\eps(\lambda)\defeq s(\mu)$.\\
Sei $C$ eine $k$-dimensionale orientierte Konfiguration, $s$ eine Kantenzuordnung vom Typ Y bezüglich $C$ und $1\le n\le k$. Wir erhalten für alle $\eps\in\sum(k,n)$ die Abbildung
\[d_{\g(C,\eps),s|_\eps}:\Lambda\underbrace{\V(\gl)}_{=\V(\g(C,\eps^0))}\to
\Lambda\underbrace{\V((\gl)^*)}_{=\V(\g(C,\eps^1))}.\]
Es sei
\[d_n(C,s)\defeq\bigoplus_{\eps\in\sum(k,n)}
(-1)^{|\eps^0|+(n+1)\textup{sp}(C,\eps^0)}
d_{\g(C,\eps),s|_\eps}:\Gamma(C)\to\Gamma(C)\]
und $d(C,s)\defeq\sum_{1\le n\le k}d_n(C,s)$. Dann erhöht $d_n(C,s)$ den $h$-Grad um $n$ und senkt den $\delta$-Grad um $2$.
\begin{satz}\label{ddn}
Es gilt $d(C,s)\circ d(C,s)=0$.
\end{satz}
Diesen Satz werden wir später beweisen. Wir erhalten also
\[\Omega(C,s)\defeq(\Gamma(C),d(C,s))\in\textup{Ob}(\textup{Q}(-2)).\]
Es gilt $\mathfrak{M}(-2)(\Omega(C,s))\cong\Gamma(C,s)=(\Gamma(C),\partial(C,s))$. Weiterhin ist $\Omega(C,s)\otimes_\mathbb{Z}\mathbb{Z}/2\mathbb{Z}$ isomorph zu dem von Szabo konstruierten Komplex (\cite{szabo}).
\begin{satz}\label{spekinv}
Seien $C$ und $D$ orientierte Konfigurationen, so dass $\overline{C}=\overline{D}$. Seien $s$ und $t$ Kantenzuordnungen, so dass $s$ bezüglich $C$ und $t$ bezüglich $D$ den Typ Y hat. Dann sind $\Omega(C,s)$ und $\Omega(D,t)$ isomorph.
\end{satz}
Diesen Satz werden wir später beweisen. Sei $D$ nun ein Verschlingungsdiagramm und $C$ eine orientierte Konfiguration, so dass $\overline{C}$ eine Nullglättung von $D$ ist. Weiterhin sei $s$ eine Kantenzuordnung vom Typ Y bezüglich $C$. Dann sei $\Omega(D,C,s)\in\textup{Ob}(\textup{Q}(-2))$ wie $\Omega(C,s)$ nur die Bigraduierung unterscheide sich, wie bei ungerader Khovanov-Homologie beschrieben.
\begin{satz}\label{spekrminv}
Seien $D_1$ und $D_2$ Verschlingungsdiagramme, die durch endlich viele Reidemeister-Bewegungen auseinander hervorgehen. Dann gibt es orientierte Konfigurationen $C_1$, $C_2$, so dass $\overline{C_1}$, $\overline{C_2}$ Nullglättungen von $D_1$, $D_2$ sind und Kantenzuordnungen $s_1$, $s_2$ vom Typ Y bezüglich $C_1$, $C_2$, so dass $\Omega(D_1,C_1,s_1)\sim_\textup{P}\Omega(D_2,C_2,s_2)$.
\end{satz}
Diesen Satz werden wir später beweisen.\\
Nach Korollar \ref{pqi} ordnen wir einer Verschlingung $L$ eine filtrierte graduierte abelsche Gruppe $\hat{H}(L)$ zu (bis auf Isomorphie). Weiterhin ordnen wir $L$ eine Spektralsequenz $\hat{E}(L)$ vom Grad $-2$ zu (bis auf Isomorphie ab dem $E^2$-Term). $\hat{E}^2(L)$ ist die ungerade Khovanov-Homologie von $L$. Dabei ist die Graduierung mittels $\chiup$ aus Definition \ref{chi} geändert. $\hat{E}(L)$ konvergiert gegen $\hat{H}(L)$.
\begin{lemma}
Sei $L$ eine Verschlingung, deren ungerade Khovanov-Homologie gleich der des Unknotens ist. Dann ist $\hat{H}(L)=\hat{H}(\textup{Unknoten})$.
\end{lemma}
\begin{proof}
Mittels $\hat{E}(L)$.
\end{proof}
Die Verschlingungen, die durch $\hat{L}$ vom Unknoten unterschieden werden können, können also auch durch ungerade Khovanov-Homologie vom Unknoten unterschieden werden.\\
Sei $C$ eine orientierte Konfiguration und $s$ eine Kantenzuordnung vom Typ X bezüglich $C$. Dann ist $s$ vom Typ Y bezüglich $r(m(C))$. Es sei $\Omega '(C,s)\defeq\Omega(r(m(C)),s)$. Damit ist $\mathfrak{M}(-2)(\Omega '(C,s))\cong\Gamma(C,s)$ und $\Omega '(C,s)\otimes_\mathbb{Z}\mathbb{Z}/2\mathbb{Z}$ ist isomorph zu dem von Szabo beschriebenen Komplex mit Randabbildung $d'$ (\cite{szabo}, Anfang von Kapitel 8).\\
Analog zum Beweis von Korollar \ref{xy} kann man aus Satz \ref{spekrminv} folgern, dass $\Omega(C,s)\sim_\textup{P}\Omega '(C,t)$ für eine orientierte Konfiguration $C$ und Kantenzuordnungen $s$, $t$, so dass $s$ Typ Y und $t$ Typ X bezüglich $C$ hat. Somit sind die von Szabo beschriebenen $\hat{H}'(L)$ und $\hat{H}(L)$ (\cite{szabo}, Anfang von Kapitel 8), sowie die beiden zugehörigen Spektralsequenzen ab dem $E^2$-Term, isomorph.

\section{Abhängigkeit von der Orientierung}
Sei $C$ eine orientierte Konfiguration. Wir haben eine Ordnung auf den Kreisen $x_1,\dots,x_n$ gegeben. Für $1\le i_1<\cdots<i_m\le n$ nennen wir $x_{i_1}\wedge\cdots\wedge x_{i_m}\in\Lambda\V(C)$ ein Monom vom Grad $m$, das durch $x_{i_1},\dots,x_{i_m}$ teilbar ist. Die Monome bilden eine Basis der freien abelschen Gruppe $\Lambda\V(C)$. Wir haben eine Ordnung auf den Kanten $e_1,\dots,e_l$ von $C$ gegeben. Diese induziert auf kanonische Weise eine Ordnung $f_1,\dots,f_l$ der Kanten von $C^*$. Für $i\in\{1,\dots,l\}$ sei $x_C(i)$ der Kreis von $C$, in dem $e_i$ liegt. Sei $s$ eine Kantenzuordnung vom Typ Y bezüglich $C$, dann sieht man leicht, dass für $i\in\{1,\dots,l\}$ Folgendes gilt:
\begin{lemma}[Filtrierungsregel]
Ist $C$ vom Typ $\tau\in\textup{T}$ und sind $\alpha\in\Lambda\V(C),\ \beta\in\Lambda\V(C^*)$ Monome, so dass $\alpha$ durch $x_C(i)$ teilbar ist und der Koeffizient von $d_{C,\tau,s}(\alpha)$ bei $\beta$ nicht Null ist, dann ist $\beta$ durch $x_{C^*}(i)$ teilbar.
\end{lemma}
Für ein Monom $\alpha\in\Lambda\V(C)$ vom Grad $m$ sei $\alpha^*\in\Lambda\V(m(C))$ das eindeutig bestimmte Monom vom Grad $n-m$, so dass $\alpha^*$ genau durch die Kreise von $m(C)$ teilbar ist, durch deren Entsprechung in $C$ das Monom $\alpha$ nicht teilbar ist. Sei $t$ eine Kantenzuordnung vom Typ Y bezüglich $m(C^*)$. Man prüft leicht nach, dass dann gilt:
\begin{lemma}[Dualitätsregel]
Sei $C$ vom Typ $\tau\in\textup{T}$ und seien $\alpha\in\Lambda\V(C),\ \beta\in\Lambda\V(C^*)$ Monome. Dann ist der Koeffizient von $d_{C,\tau,s}(\alpha)$ bei $\beta$ bis auf ein Vorzeichen gleich dem Koeffizienten von $d_{m(C^*),m(\tau^*),t}(\beta^*)$ bei $\alpha^*$.
\end{lemma}
Die Filtrierungsregel gilt auch für $d_{C,s}=\sum d_{C,\tau,s}$. Die Dualitätsregel gilt auch für $d_{C,s}$ und $d_{m(C^*),t}$.\\
Für eine aktive orientierte eindimensionale Konfiguration $C$ möchten wir einen $\mathbb{Z}$-Modul-Homomorphismus $H_C:\Lambda\V(C)\to\Lambda\V(C^*)$ definieren. Für den 1. Fall aus Abbildung \ref{eindim} sei $H_C(1)=H_C(x_1)=H_C(x_2)=0$ und $H_C(x_1\wedge x_2)=y$. Für den 2. Fall sei $H_C(1)=1$ und $H_C(y)=0$. Ist $C$ nicht unbedingt aktiv, so sei $H_C$ analog zu $\partial_C$ definiert. $H_C$ erfüllt Filtrierungs- und Dualitätsregel.\\
Sei $C$ nun eine orientierte $k$-dimensionale Konfiguration und $s$ eine Kantenzuordnung vom Typ Y bezüglich $C$. Für $i=1,\dots,k$ sei $H_i(C,s):\Gamma(C)\to\Gamma(C)$ definiert als
\[\bigoplus_{\eps\in\sum(k,1)\textup{ mit }\eps_i=*}s(\eps)(-1)^{\textup{sp}(C,\eps^0)}H_{\gl}.\]
Es gilt $H_i(C,s)\circ H_i(C,s)=0$. Weiterhin erhöht $H_i(C,s)$ den $h$-Grad um $1$ und lässt den $\delta$-Grad unverändert.
\begin{satz}\label{dhs}
Seien $C$, $D$ orientierte $k$-dimensionale Konfigurationen, so dass $\overline{C}=\overline{D}$ und die Orientierung der Bögen sich nur beim $i$-ten Bogen unterscheide, für $i\in\{1,\dots,k\}$. Dann gibt es für jede Kantenzuordnung $s$ vom Typ Y bezüglich $C$ eine Kantenzuordnung $t$ vom Typ Y bezüglich $D$, so dass
\[d(C,s)-d(D,t)=d(C,s)\circ H_i(C,s)-H_i(C,s)\circ d(C,s).\]
\end{satz}
\begin{proof}
Ohne Einschränkungen sei $i=1$.\\
Sei $\kappa:\sum(k,1)\to\{1,-1\}$ definiert durch
\[\kappa(\eps)=\begin{cases}
-1\qquad&\textup{falls $\eps_1=*$ und $\gl$ eine Spaltung ist,}\\
1\qquad&\textup{sonst.}
\end{cases}\]
Der Beweis von Lemma \ref{inv2} zeigt, dass $t\defeq\kappa s$ eine Kantenzuordnung vom Typ Y bezüglich $D$ ist. Weiterhin ist $d_1(C,s)=d_1(D,t)$. Sei nun $k\ge 2$ und $a=(0,*,\dots,*),\ b=(1,*,\dots,*)\in\sum(k,k-1)$, sowie  $c=(*,0,\dots,0),\ d=(*,1,\dots,1)\in\sum(k,1)$. Zu zeigen ist:
\begin{equation}\label{dhgl}
d_{C,s}-d_{D,t}=(-1)^{1+k\textup{sp}(C,b^0)}s(c)d_{\g(C,b),s|_b}\circ H_{\g(C,c)}-(-1)^{\textup{sp}(C,d^0)}s(d)H_{\g(C,d)}\circ d_{\g(C,a),s|_a}
\end{equation}
Wir müssen nur den Fall betrachten, dass $C$ aktiv ist. Sei $C$ zunächst nicht zusammenhängend. Dann müssen wir nur den Fall betrachten, dass $C$ aus genau zwei Zusammenhangskomponenten besteht, wobei in der einen Zusammenhangskomponente der erste Bogen und sonst kein weiterer Bogen liegt. Sei $\theta$ ein für $\textup{akt}(\g(C,a))=\textup{akt}(\g(C,b))$ zulässiger Kantenzug. Sei zunächst $\textup{akt}(\g(C,c))=\textup{akt}(\g(C,d))$ die linke Seite aus Abbildung \ref{eindim}. Dann ist $\textup{sp}(C,b^0)=0$. Weiterhin gilt $s(d)s|_a(\theta)=(-1)^{k-1}s(c)s|_b(\theta)$, weil $\gl$ vom Typ K ist, für alle $\eps\in\sum(k,2)$ mit $\eps_1=*$. Ist $\g(C,a)$ vom Typ $\tau\in\textup{T}$, so ist
\begin{align*}
&d_{\g(C,b),\tau,s|_b}\circ H_{\g(C,c)}(x_1\wedge x_2\wedge x_{\g(C,b),\tau,\theta})\\
=&d_{\g(C,b),\tau,s|_b}(y\wedge x_{\g(C,b),\tau,\theta})\\
=&(-1)^{\textup{gr}(x_{\g(C,b),\tau,\theta})}s|_b(\theta)
\textup{sp}(\g(C,b),\theta)y_{\g(C,b),\tau}\wedge y
\end{align*}
und
\begin{align*}
&H_{\g(C,d)}\circ d_{\g(C,a),\tau,s|_a}(x_1\wedge x_2\wedge x_{\g(C,b),\tau,\theta})\\
=&s|_a(\theta)\textup{sp}(\g(C,a),\theta)H_{\g(C,d)}(y_{\g(C,a),\tau}\wedge x_1\wedge x_2)\\
=&s|_a(\theta)\textup{sp}(\g(C,b),\theta)y\wedge y_{\g(C,b),\tau}\\
=&(-1)^{\textup{gr}(y_{\g(C,b),\tau})}s|_a(\theta)\textup{sp}(\g(C,b),\theta) y_{\g(C,b),\tau}\wedge y.
\end{align*}
Wegen
\[\textup{gr}(y_{\g(C,b),\tau})-\textup{gr}(x_{\g(C,b),\tau,\theta})
=\textup{sp}(\g(C,b),(1,\dots,1))-(k-1)+1=\textup{sp}(C,d^0)-k+2\]
ist die rechte Seite von Gleichung (\ref{dhgl}) trivial.\\
Sei nun $\textup{akt}(\g(C,c))=\textup{akt}(\g(C,d))$ die rechte Seite aus Abbildung \ref{eindim}. Dann ist $\textup{sp}(C,b^0)=1$. Weiterhin ist $s(d)s|_a(\theta)=(-1)^{k-1-\textup{sp}(C,d^0)}s(c)s|_b(\theta)$, weil Folgendes gilt: Seien $\eps\in\sum(k,2),\ \etaup\in\sum(k,1)$ gleich, bis auf die erste Komponente, es sei nämlich $\eps_1=*$ und $\etaup_1=0$. Dann ist $\gl$ vom Typ K, falls $\gle$ Fusion ist und vom Typ A, falls $\gle$ Spaltung ist. Also ist die rechte Seite von Gleichung (\ref{dhgl}) trivial.\\
Sei $C$ nun zusammenhängend. Sei $C$ zuerst Fall 28 aus Abbildung \ref{zweidim2}, siehe Abbildung \ref{f28}.
\begin{figure}[htbp]
\centering
\includegraphics{graph.10}
\caption{}
\label{f28}
\end{figure}
Weil $C$ vom Typ A und $D$ vom Typ K ist, gilt
\[s(0,*)s(*,1)=s(*,0)s(1,*)=t(*,0)t(1,*)=-t(0,*)t(*,1).\]
Weiterhin ist $\textup{sp}(C,(1,0))=\textup{sp}(C,(0,1))=0$. Es ist $d_{C,s}(1)=-s(0,*)s(*,1)$ und $d_{D,t}(x_1\wedge x_2)=t(*,0)t(1,*)y_2\wedge y_1$. Weiterhin gilt $d_{\g(C,(1,*)),s|_{(1,*)}}\circ H_{\g(C,(*,0))}(x_1\wedge x_2)=s(1,*)y_2\wedge y_1$ und $H_{\g(C,(*,1))}\circ d_{\g(C,(0,*)),s|_{(0,*)}}(1)=s(0,*)$. Somit ist Gleichung (\ref{dhgl}) erfüllt.\\
Ist $C$ Fall 29 aus Abbildung \ref{zweidim2}, so sind die Rollen von $C$ und $D$ und damit auch von $s$ und $t$ aus Fall 28 vertauscht und man erhält auf beiden Seiten von Gleichung (\ref{dhgl}) das Negative aus Fall 28.\\
Sei $C$ nun Fall 45, siehe Abbildung \ref{f45}.
\begin{figure}[htbp]
\centering
\includegraphics{graph.11}
\caption{}
\label{f45}
\end{figure}
Weil $C$ vom Typ Y ist, gilt $s(0,*)s(*,1)=s(*,0)s(1,*)$. Weiterhin ist $\textup{sp}(C,(1,0))=\textup{sp}(C,(0,1))=1$. Es ist $d_{C,s}(1)=-s(0,*)s(*,1)$ und $d_{C,s}(x)=-s(0,*)s(*,1)y$ und $d_{D,t}=0$. Weiterhin gilt $d_{\g(C,(1,*)),s|_{(1,*)}}\circ H_{\g(C,(*,0))}(1)=s(1,*)$ und $H_{\g(C,(*,1))}\circ d_{\g(C,(0,*)),s|_{(0,*)}}(x)=-s(0,*)y$. Somit ist Gleichung (\ref{dhgl}) erfüllt.\\
Ist $C$ Fall 44, so sind die Rollen von $C$ und $D$ und damit auch von $s$ und $t$ aus Fall 45 vertauscht und man erhält auf beiden Seiten von Gleichung (\ref{dhgl}) das Negative aus Fall 45.\\
Unser $C$ ist eine zusammenhängende, orientierte, mindestens zweidimensionale Konfiguration, es gilt also: $C$ ist einer der Fälle 28,29,44,45 aus Abbildung \ref{zweidim2} oder
\[u_C\defeq\#\{\tau\in\textup{T}\ \ |\ \ \textup{$C$ oder $D$ ist vom Typ $\tau$}\}\le 1.\]
Sei $u_C=1$. Ohne Einschränkungen gehen wir davon aus, dass es $\tau\in\textup{T}$ gibt, so dass $C$ vom Typ $\tau$ ist.\\
Sei $\tau=\textup{A}_k$, $k\ge 3$. Dann ist $\g(C,b)$ von keinem Typ aus T. Weiterhin ist $\g(C,a)$ vom Typ $\textup{A}_{k-1}$ und $\g(C,d)$ ist Spaltung. Also ist Gleichung (\ref{dhgl}) erfüllt.\\
Sei $\tau=\textup{B}_k$, $k\ge 3$. Dann ist $\g(C,a)$ von keinem Typ aus T. Weiterhin ist $\g(C,b)$ vom Typ $\textup{B}_{k-1}$ und $\g(C,c)$ ist Fusion. Es folgt $\textup{sp}(C,b^0)=0$. Wähle eine Ordnung auf den Bögen von $C$, die mit dem ersten Bogen (aus der gegebenen Ordnung der Bögen) beginnt und bei der der Bogen an zweiter Stelle steht, dessen Endpunkt auf dem Kreis liegt, auf dem der Anfangspunkt des ersten Bogens liegt. Dies liefert zulässige Kantenzüge $\theta,\ \etaup$ für $C,\ \g(C,b)$. Es gilt $s(\theta)=s(c)s|_b(\etaup)$ und $\textup{sp}(C,\theta)=\textup{sp}(\g(C,b),\etaup)$. Außerdem ist $H_{\g(C,c)}(x_{C,\textup{B}_k,\theta})=-x_{\g(C,b),\textup{B}_{k-1},\etaup}$ und $y_{\g(C,b),\textup{B}_{k-1}}=y_{C,\textup{B}_k}$. Also ist Gleichung (\ref{dhgl}) erfüllt.\\
Sei $\tau=\textup{C}_{p,q}$, $p+q\ge 3$. Die $S^2$ ohne den Kreis $x$ von $C$ besteht aus zwei Zusammenhangskomponenten. Betrachte zuerst den Fall, dass der erste Bogen von $C$ der einzige Bogen in seiner Zusammenhangskomponente von $S^2\setminus x$ ist. Dann ist $\g(C,a)$ von keinem Typ aus T. Weiterhin ist $\g(C,b)$ vom Typ $\textup{A}_{k-1}$ und $\g(C,c)$ ist Spaltung. Also ist Gleichung (\ref{dhgl}) erfüllt.\\
Betrachte nun den Fall, dass in der Zusammenhangskomponente von $S^2\setminus x$, in der der erste Bogen liegt, noch mindestens ein weiterer Bogen liegt. Dann ist $\g(C,b)$ von keinem Typ aus T oder $\g(C,b)$ ist vom Typ $\textup{F}_{1,1}$ oder vom Typ $\textup{G}_{1,1}$. Weil $\g(C,c)$ Spaltung ist, ist $d_{\g(C,b),s|_b}\circ H_{\g(C,c)}=0$. Weiterhin ist $\g(C,a)$ vom Typ $\textup{C}_{p-1,q}$ oder $\textup{C}_{p,q-1}$ und $\g(C,d)$ ist Spaltung. Also ist Gleichung (\ref{dhgl}) erfüllt.\\
Sei $\tau=\textup{D}_{p,q}$, $p+q\ge 3$. Sei $x$ der Kreis von $C$, auf dem zwei Angangs- und zwei Endpunkte von Bögen liegen. Betrachte zuerst den Fall, dass der erste Bogen von $C$ der einzige Bogen in seiner Zusammenhangskomponente von $S^2\setminus x$ ist. Dann ist $\g(C,b)$ von keinem Typ aus T. Weiterhin ist $\g(C,a)$ vom Typ $\textup{B}_{k-1}$ und $\g(C,d)$ ist Fusion. Ein zulässiger Kantenzug $\theta$ für $C$, so dass $\phi(\theta)(k)$ der erste Bogen ist, liefert einen zulässigen Kantenzug $\etaup$ für $\g(C,a)$. Es gilt $s(\theta)=s|_a(\etaup)s(d)$ und $\textup{sp}(C,\theta)=(-1)^{\textup{sp}(C,d^0)}\textup{sp}(\g(C,a),\etaup)$. Außerdem ist $x_{C,\textup{D}_{p,q},\theta}=x_{\g(C,a),\textup{B}_{k-1},\etaup}$ und $H_{\g(C,d)}(y_{\g(C,a),\textup{B}_{k-1}})=-y_{C,\textup{D}_{p,q}}$. Also ist Gleichung (\ref{dhgl}) erfüllt.\\
Betrachte nun den Fall, dass in der Zusammenhangskomponente von $S^2\setminus x$, in der der erste Bogen liegt, noch mindestens ein weiterer Bogen liegt. Dann ist $\g(C,a)$ von keinem Typ aus T oder $\g(C,a)$ ist vom Typ $\textup{F}_{1,1}$ oder vom Typ $\textup{G}_{1,1}$. Weil $\g(C,d)$ Fusion ist, ist $H_{\g(C,d)}\circ d_{\g(C,a),s|_a}=0$. Weiterhin ist $\g(C,b)$ vom Typ $\textup{D}_{p-1,q}$ oder $\textup{D}_{p,q-1}$ und $\g(C,c)$ ist Fusion. Es folgt $\textup{sp}(C,b^0)=0$. Liege zunächst der Endpunkt des ersten Bogens nicht auf $x$. Dann sei $\theta$ ein zulässiger Kantenzug für $C$, so dass $\phi(\theta)(1)$ der erste Bogen ist. $\theta$ liefert einen zulässigen Kantenzug $\etaup$ für $\g(C,b)$. Es gilt $s(\theta)=s(c)s|_b(\etaup)$ und $\textup{sp}(C,\theta)=\textup{sp}(\g(C,b),\etaup)$. Außerdem ist $H_{\g(C,c)}(x_{C,\textup{D}_{p,q},\theta})=-x_{\g(C,b),\textup{D}_{p-1,q},\etaup}$ und $y_{\g(C,b),\textup{D}_{p-1,q}}=y_{C,\textup{D}_{p,q}}$, wobei für $\textup{D}_{p-1,q}$ auch $\textup{D}_{p,q-1}$ stehen kann. Also ist Gleichung (\ref{dhgl}) erfüllt.\\
Liege nun der Endpunkt des ersten Bogens auf $x$. Dann sei $\theta$ ein zulässiger Kantenzug für $C$, so dass $\phi(\theta)(k-1)$ der erste Bogen ist und $\phi(\theta)(k-2)$ der Bogen ist, dessen Endpunkt auf dem Kreis liegt, auf dem der Anfangspunkt des ersten Bogens liegt. $\theta$ liefert einen zulässigen Kantenzug $\etaup$ für $\g(C,b)$. Es gilt $s(\theta)=(-1)^{k-2}s(c)s|_b(\etaup)$, weil $\gl$ vom Typ K ist, für alle $\eps\in\sum(k,2)$ mit $|\eps^1|<k$. Weiterhin gilt $\textup{sp}(C,\theta)=\textup{sp}(\g(C,b),\etaup)$. Außerdem ist $H_{\g(C,c)}(x_{C,\textup{D}_{p,q},\theta})
=(-1)^{k-3}x_{\g(C,b),\textup{D}_{p-1,q},\etaup}$ und $y_{\g(C,b),\textup{D}_{p-1,q}}=y_{C,\textup{D}_{p,q}}$, wobei für $\textup{D}_{p-1,q}$ auch $\textup{D}_{p,q-1}$ stehen kann. Also ist Gleichung (\ref{dhgl}) erfüllt.\\
Sei $\tau=\textup{F}_{0,2}$, siehe Abbildung \ref{f0,2}.
\begin{figure}[htbp]
\centering
\includegraphics{graph.12}
\caption{}
\label{f0,2}
\end{figure}
Weil $C$ vom Typ A ist, gilt $s(0,*)s(*,1)=s(*,0)s(1,*)$. Weiterhin ist $\textup{sp}(C,(1,0))=\textup{sp}(C,(0,1))=1$. Es ist $d_{C,s}(1)=-s(0,*)s(*,1)y_2$. Weiterhin gilt $d_{\g(C,(1,*)),s|_{(1,*)}}\circ H_{\g(C,(*,0))}(1)=s(1,*)(y_2-y_1)$ und $H_{\g(C,(*,1))}\circ d_{\g(C,(0,*)),s|_{(0,*)}}(1)=-s(0,*)y_1$. Somit ist Gleichung (\ref{dhgl}) erfüllt.\\
Sei $\tau=\textup{F}_{2,0}$, siehe Abbildung \ref{f2,0}.
\begin{figure}[htbp]
\centering
\includegraphics{graph.13}
\caption{}
\label{f2,0}
\end{figure}
Weil $C$ vom Typ K ist, gilt $s(0,*)s(*,1)=-s(*,0)s(1,*)$. Weiterhin ist $\textup{sp}(C,(1,0))=\textup{sp}(C,(0,1))=0$. Es ist $d_{C,s}(x_1\wedge x_3)=s(0,*)s(*,1)y$. Weiterhin gilt $d_{\g(C,(1,*)),s|_{(1,*)}}\circ H_{\g(C,(*,0))}(x_3\wedge x_2)=s(1,*)y$ und $H_{\g(C,(*,1))}\circ d_{\g(C,(0,*)),s|_{(0,*)}}(x_1\wedge x_3)=-s(0,*)y=H_{\g(C,(*,1))}\circ d_{\g(C,(0,*)),s|_{(0,*)}}(x_2\wedge x_3)$. Somit ist Gleichung (\ref{dhgl}) erfüllt.\\
Sei $\tau=\textup{F}_{1,1}$,so dass der Anfangs- und Endpunkt des ersten Bogens von $C$ auf dem gleichen Kreis liegen, siehe Abbildung \ref{f1,1a}.
\begin{figure}[htbp]
\centering
\includegraphics{graph.14}
\caption{}
\label{f1,1a}
\end{figure}
Weil $C$ vom Typ K ist, gilt $s(0,*)s(*,1)=-s(*,0)s(1,*)$. Weiterhin ist $\textup{sp}(C,(1,0))=1$ und $\textup{sp}(C,(0,1))=0$. Es ist $d_{C,s}(x_2)=s(0,*)s(*,1)y_1$. Weiterhin gilt $d_{\g(C,(1,*)),s|_{(1,*)}}\circ H_{\g(C,(*,0))}(1)=s(1,*)$ und $d_{\g(C,(1,*)),s|_{(1,*)}}\circ H_{\g(C,(*,0))}(x_2)=s(1,*)y_1$ und $H_{\g(C,(*,1))}\circ d_{\g(C,(0,*)),s|_{(0,*)}}(1)=s(0,*)$. Somit ist Gleichung (\ref{dhgl}) erfüllt.\\
Sei nun $\tau=\textup{F}_{1,1}$,so dass der Anfangs- und Endpunkt des ersten Bogens von $C$ auf unterschiedlichen Kreisen liegen, siehe Abbildung \ref{f1,1b}.
\begin{figure}[htbp]
\centering
\includegraphics{graph.15}
\caption{}
\label{f1,1b}
\end{figure}
Es ist $\textup{sp}(C,(1,0))=0$ und $\textup{sp}(C,(0,1))=1$ und $d_{C,s}(x_2)=s(*,0)s(1,*)y_1$. Weiterhin gilt $d_{\g(C,(1,*)),s|_{(1,*)}}\circ H_{\g(C,(*,0))}(x_2\wedge x_1)=s(1,*)y_1\wedge y_2$ und $H_{\g(C,(*,1))}\circ d_{\g(C,(0,*)),s|_{(0,*)}}(x_2)=-s(0,*)y_1$ und $H_{\g(C,(*,1))}\circ d_{\g(C,(0,*)),s|_{(0,*)}}(x_1\wedge x_2)=s(0,*)y_1\wedge y_2$. Somit ist Gleichung (\ref{dhgl}) erfüllt.\\
Sei nun $\tau=\textup{F}_{p,q}$, $p+q\ge 3$. Betrachte zuerst den Fall, dass Anfangs- und Endpunkt des ersten Bogens von $C$ auf dem gleichen Kreis liegen. Dann sind $\g(C,a)$ und $\g(C,b)$ vom Typ $\textup{F}_{p,q-1}$. Weiterhin sind $\g(C,c)$ und $\g(C,d)$ Spaltungen. Es folgt $H_{\g(C,d)}\circ d_{\g(C,a),s|_a}=0$ und $\textup{sp}(C,b^0)=1$. Ein zulässiger Kantenzug $\theta$ für $C$, so dass $\phi(\theta)(1)$ der erste Bogen ist, liefert einen zulässigen Kantenzug $\etaup$ für $\g(C,b)$. Es gilt $s(\theta)=s(c)s|_b(\etaup)$ und $\textup{sp}(C,\theta)=(-1)^{1+(k-2)}\textup{sp}(\g(C,b),\etaup)$.  Außerdem ist $H_{\g(C,c)}(x_{C,\textup{F}_{p,q},\theta})
=x_{\g(C,b),\textup{F}_{p,q-1},\etaup}$ und $y_{\g(C,b),\textup{F}_{p,q-1}}=y_{C,\textup{F}_{p,q}}$. Also ist Gleichung (\ref{dhgl}) erfüllt.\\
Betrachte nun den Fall, dass Anfangs- und Endpunkt des ersten Bogens von $C$ auf unterschiedlichen Kreisen liegen. Dann sind $\g(C,a)$ und $\g(C,b)$ vom Typ $\textup{F}_{p-1,q}$. Weiterhin sind $\g(C,c)$ und $\g(C,d)$ Fusionen. Es folgt $d_{\g(C,b),s|_b}\circ H_{\g(C,c)}=0$ und $\textup{sp}(C,d^0)=q$. Ein zulässiger Kantenzug $\theta$ für $C$, so dass $\phi(\theta)(k)$ der erste Bogen ist, liefert einen zulässigen Kantenzug $\etaup$ für $\g(C,a)$. Es gilt $s(\theta)=s|_a(\etaup)s(d)$ und $\textup{sp}(C,\theta)=(-1)^q\textup{sp}(\g(C,a),\etaup)$. Sei $x$ der Kreis von $C$, auf dem der Anfangspunkt des ersten Bogens liegt. Dann gilt $x_{C,\textup{F}_{p,q},\theta}=x_{\g(C,a),\textup{F}_{p-1,q},\etaup}\wedge x$ und $H_{\g(C,d)}(y_{\g(C,a),\textup{F}_{p-1,q}}\wedge x)=-y_{C,\textup{F}_{p,q}}$. Also ist Gleichung (\ref{dhgl}) erfüllt.\\
Dass für $\tau=\textup{G}_{p,q}$, $p+q\ge 2$, Gleichung (\ref{dhgl}) gilt, zeigt man analog zu $\tau=\textup{F}_{p,q}$.\\
Sei nun $u_C=0$. Wir müssen zeigen, dass dann die rechte Seite von Gleichung (\ref{dhgl}) trivial ist. Ergibt sich für den ersten Bogen von $C$ eine Situation wie in Abbildung \ref{sit}, so ist $\textup{akt}(\g(C,a))=\textup{akt}(\g(C,b))$.
\begin{figure}[htbp]
\centering
\includegraphics{graph.16}
\caption{}
\label{sit}
\end{figure}
Wir prüfen nun nach, dass dann
\begin{equation}\label{dhtgl}
(-1)^{1+k\textup{sp}(C,b^0)}s(c)d_{\g(C,b),\tau,s|_b}\circ H_{\g(C,c)}=(-1)^{\textup{sp}(C,d^0)}s(d)H_{\g(C,d)}\circ d_{\g(C,a),\tau,s|_a}
\end{equation}
ist, falls $\g(C,a)$ vom Typ $\tau\in\textup{T}$ ist. In Situation 1 ist $\textup{sp}(C,b^0)=0$, in Situation 2 ist $\textup{sp}(C,b^0)=1$. Sei $\theta$ ein zulässiger Kantenzug für $\g(C,a)$. In Situation 1 ist $s(c)s|_b(\theta)=(-1)^{k-1}s(d)s|_a(\theta)$, weil $\gl$ vom Typ K ist für alle $\eps\in\sum(k,2)$ mit $\eps_1=*$. In Situation 2 ist $s(c)s|_b(\theta)=(-1)^{k-1-\textup{sp}(C,d^0)}s(d)s|_a(\theta)$, weil für alle $\eps\in\sum(k,2)$ mit $\eps_1=*$ und $\mu\in\sum(k,1)$ mit
\[\mu_i=\begin{cases}
0\qquad&\textup{für }i=1,\\
\eps_i\qquad&\textup{sonst},
\end{cases}\]
gilt: Ist $\g(C,\mu)$ Spaltung, so ist $\gl$ vom Typ A. Ist $\g(C,\mu)$ Fusion, so ist $\gl$ vom Typ K. Man sieht leicht, dass in Situation 2 Gleichung (\ref{dhtgl}) gilt. In Situation gilt Gleichung (\ref{dhtgl}) wegen Gleichung (\ref{grgl}).\\
Wir werden nun zeigen: Ist $u_C=0$ und ergibt sich für den ersten Bogen von $C$ keine der beiden Situationen aus Abbildung \ref{sit}, so ist $H_{\g(C,d)}\circ d_{\g(C,a),\tau,s|_a}=0$, falls $\g(C,a)$ vom Typ $\tau$ ist.\\
Ist $\g(C,b)$ vom Typ $\sigma$, so folgt für eine Kantenzuordnung $v$ vom Typ Y bezüglich $m(C^*)$ aus $H_{\g(m(C^*),d)}\circ d_{\g(m(C^*),a),m(\sigma^*),v|_a}=0$ und der Dualitätsregel, dass $d_{\g(C,b),\sigma,s|_b}\circ H_{\g(C,c)}=0$ ist.\\
Sei der erste Bogen von $C$ mit $\gammaup$ bezeichnet.\\
Sei $\tau=\textup{A}_k$. Dann verbindet $\gammaup$ zwei unterschiedliche Kanten aus einem der beiden Kreise von $\g(C,a)$. Dann ist aber $\g(C,d)$ Fusion.\\
Sei $\tau=\textup{B}_k$. Dann verbindet $\gammaup$ zwei unterschiedliche Kreise von $\g(C,a)$. Dann ist aber $\g(C,d)$ Spaltung.\\
Sei $\tau=\textup{C}_{p,q}$. Dann verbindet $\gammaup$ zwei unterschiedliche Kanten des Kreises von $\g(C,a)$, wobei für mindestens eine dieser Kanten gilt: Die beiden Bögen, die diese Kante begrenzen, liegen auf der anderen Seite des Kreises wie $\gammaup$. Dann ist aber $\g(C,d)$ Fusion.\\
Sei $\tau=\textup{D}_{p,q}$. Dann ist $\g(C,d)$ Spaltung.\\
Sei $\tau=\textup{F}_{p,q}$. Dann gibt es zwei Möglichkeiten:
\begin{enumerate}
\item 
Für mindestens eine der beiden Kanten, die $\gammaup$ verbindet, gilt: Diese Kante wird von Anfangs- und Endpunkt des gleichen Bogens begrenzt. Dann ist $\g(C,d)$ Fusion, aber durch mindestens einen der beiden Kreise aus $\textup{akt}(\g(C,d))$ ist $y_{\g(C,a),\textup{F}_{p,q}}$ nicht teilbar. (Teilbarkeit sei wie für Monome definiert.)
\item 
Andernfalls ist $\g(C,d)$ Spaltung, aber $y_{\g(C,a),\textup{F}_{p,q}}$ ist durch den aktiven Kreis von $\g(C,d)$ teilbar.
\end{enumerate}
Den Fall $\tau=\textup{G}_{p,q}$ behandelt man analog zu $\tau=\textup{F}_{p,q}$.
\end{proof}

\section{$d\circ d=0$}
Um Satz \ref{ddn} zu beweisen, reicht es folgendes zu zeigen:
\begin{satz}\label{sddn}
Sei $C$ eine $k$-dimensionale orientierte Konfiguration und $s$ eine Kantenzuordnung vom Typ Y bezüglich $C$. Dann gilt
\[\sum_{i=1}^{k-1}d_{k-i}(C,s)\circ d_i(C,s)=0.\]
\end{satz}
Für $k=2$ gilt: $d_1(C,s)\circ d_1(C,s)=-\partial(C,s)\circ\partial(C,s)=0$.
\begin{lemma}\label{invdd}
Seien $C$, $D$ $k$-dimensionale orientierte Konfigurationen mit $\overline{C}=\overline{D}$. Wenn Satz \ref{sddn} für alle $(k-1)$-dimensionalen Konfigurationen gilt, dann gibt es für jede Kantenzuordnung $s$ vom Typ Y bezüglich $C$ eine Kantenzuordnung $t$ vom Typ Y bezüglich $D$, so dass
\[\sum_{i=1}^{k-1}d_{k-i}(C,s)\circ d_i(C,s)=\sum_{i=1}^{k-1}d_{k-i}(D,t)\circ d_i(D,t).\]
\end{lemma}
\begin{proof}
Es genügt den Fall zu betrachten, dass sich die Orientierung der Bögen nur beim $j$-ten Bogen unterscheidet, für $j\in\{1,\dots,k\}$. Nach Satz \ref{dhs} gibt es eine Kantenzuordnung $t$ vom Typ Y bezüglich $D$, so dass
\[d(C,s)-d(D,t)=d(C,s)\circ H_j(C,s)-H_j(C,s)\circ d(C,s).\]
Dann gilt:
\begin{align*}
&\sum_{i=1}^{k-1}d_{k-i}(D,t)\circ d_i(D,t)\\
=&\sum_{i=1}^{k-1}(d_{k-i}(C,s)+H_j(C,s)\circ d_{k-i-1}(C,s)-d_{k-i-1}(C,s)\circ H_j(C,s))\\
&\circ(d_i(C,s)+H_j(C,s)\circ d_{i-1}(C,s)-d_{i-1}(C,s)\circ H_j(C,s))\\
=&\sum_{i=1}^{k-1}(d_{k-i}(C,s)\circ d_i(C,s)\\
&-d_{k-i}(C,s)\circ d_{i-1}(C,s)\circ H_j(C,s)\\
&+H_j(C,s)\circ d_{k-i-1}(C,s)\circ d_i(C,s)\\
&+\underbrace{H_j(C,s)\circ d_{k-i-1}(C,s)\circ H_j(C,s)}_{=0}\circ d_{i-1}(C,s)\\
&-\underbrace{H_j(C,s)\circ d_{k-i-1}(C,s)\circ d_{i-1}(C,s)\circ H_j(C,s)}_{=0}\\
&-d_{k-i-1}(C,s)\circ\underbrace{H_j(C,s)\circ H_j(C,s)}_{=0}\circ d_{i-1}(C,s)\\
&+d_{k-i-1}(C,s)\circ\underbrace{H_j(C,s)\circ d_{i-1}(C,s)+ H_j(C,s)}_{=0})\\
=&\sum_{i=1}^{k-1}d_{k-i}(C,s)\circ d_i(C,s)\\
&-\underbrace{\left(\sum_{i=1}^{k-1}d_{k-i}(C,s)\circ d_{i-1}(C,s)\right)}_{=0}\circ H_j(C,s)\\
&+H_j(C,s)\circ\underbrace{\sum_{i=1}^{k-1}d_{k-i-1}(C,s)\circ d_{i}(C,s)}_{=0}
\end{align*}
\end{proof}
\begin{satz}\label{exddn}
Sei $D$ eine $k$-dimensionale Konfiguration und seien $\alpha\in\Lambda\V(D)$, $\beta\in\Lambda\V(D^*)$ Monome. Dann gibt es eine orientierte Konfiguration $C$, so dass $\overline{C}=D$ und so dass für alle Kantenzuordnungen $s$ vom Typ Y bezüglich $C$ der Koeffizient von
\[\sum_{i=1}^{k-1}d_{k-i}(C,s)\circ d_i(C,s)(\alpha)\]
bei $\beta$ trivial ist.
\end{satz}
Aus Lemma \ref{invdd} und Satz \ref{exddn} folgt mittels Induktion Satz \ref{sddn}.
\begin{proof}[Beweis von Satz \ref{exddn}]
Wir brauchen nur den Fall zu betrachten, dass $D$ aktiv und $k\ge 3$ ist. Sei $D$ zunächst nicht zusammenhängend. Dann brauchen wir nur den Fall zu betrachten, dass $D$ aus genau zwei Zusammenhangskomponenten besteht und für ein $1\le n\le k-1$ die ersten $n$ Bögen von $D$ in der einen Zusammenhangskomponente und die restlichen $k-n$ Bögen in der anderen Zusammenhangskomponente liegen. Seien $a,b\in\sum(k,k-n)$ mit
\[a_j=\begin{cases}
0\qquad&\textup{für}j\le n\\
*\qquad&\textup{für}j>n
\end{cases}\]
und
\[b_j=\begin{cases}
1\qquad&\textup{für}j\le n\\
*\qquad&\textup{für}j>n.
\end{cases}\]
Seien $c,d\in\sum(k,n)$ mit
\[c_j=\begin{cases}
*\qquad&\textup{für}j\le n\\
0\qquad&\textup{für}j>n
\end{cases}\]
und
\[d_j=\begin{cases}
*\qquad&\textup{für}j\le n\\
1\qquad&\textup{für}j>n.
\end{cases}\]
Zu zeigen ist, dass dann für beliebige $C,s$ gilt:
\begin{equation}\label{ddgl}
(-1)^{n+(k-n+1)\textup{sp}(C,b^0)}d_{\g(C,b),s|_b}\circ d_{\g(C,c),s|_c}
+(-1)^{n-k+(n+1)\textup{sp}(C,d^0)}d_{\g(C,d),s|_d}\circ d_{\g(C,a),s|_a}=0
\end{equation}
Es ist $\textup{akt}(\g(C,a))=\textup{akt}(\g(C,b))$ und $\textup{akt}(\g(C,c))=\textup{akt}(\g(C,d))$. Seien $\theta,\ \etaup$ zulässige Kantenzüge für $\g(C,a),\ \g(C,c)$. Dann ist
\[s|_b(\theta)s|_c(\etaup)
=(-1)^{n(k-n)-\textup{sp}(C,b^0)\textup{sp}(C,d^0)}s|_d(\etaup)s|_a(\theta).\]
Seien $\g(C,a),\ \g(C,c)$ vom Typ $\tau,\sigma\in\textup{T}$. Dann ist
\begin{align*}
&d_{\g(C,b),\tau,s|_b}\circ d_{\g(C,c),\sigma,s|_c}(x_{\g(C,c),\sigma,\etaup}\wedge x_{\g(C,b),\tau,\theta})\\
=&s|_c(\etaup)\textup{sp}(\g(C,c),\etaup)d_{\g(C,b),\tau,s|_b}
(y_{\g(C,c),\sigma}\wedge x_{\g(C,b),\tau,\theta})\\
=&(-1)^{\textup{gr}(y_{\g(C,c),\sigma})\textup{gr}(x_{\g(C,b),\tau,\theta})}
s|_b(\theta)s|_c(\etaup)\textup{sp}(\g(C,b),\theta)\textup{sp}(\g(C,c),\etaup)
y_{\g(C,b),\tau}\wedge y_{\g(C,c),\sigma}
\end{align*}
und
\begin{align*}
&d_{\g(C,d),\sigma,s|_d}\circ d_{\g(C,a),\tau,s|_a}(x_{\g(C,c),\sigma,\etaup}\wedge x_{\g(C,b),\tau,\theta})\\
=&(-1)^{\textup{gr}(x_{\g(C,c),\sigma,\etaup})\textup{gr}(x_{\g(C,b),\tau,\theta})}
s|_a(\theta)\textup{sp}(\g(C,a),\theta)d_{\g(C,d),\sigma,s|_d}(y_{\g(C,a),\tau}\wedge x_{\g(C,d),\sigma,\etaup})\\
=&(-1)^{\textup{gr}(x_{\g(C,c),\sigma,\etaup})(\textup{gr}(x_{\g(C,b),\tau,\theta})
+\textup{gr}(y_{\g(C,b),\tau}))}
s|_d(\etaup)s|_a(\theta)\textup{sp}(\g(C,d),\etaup)\textup{sp}(\g(C,a),\theta)
y_{\g(C,d),\sigma}\wedge y_{\g(C,a),\tau}\\
=&(-1)^{(\textup{gr}(x_{\g(C,c),\sigma,\etaup})+\textup{gr}(y_{\g(C,c),\sigma}))
(\textup{gr}(x_{\g(C,b),\tau,\theta})+\textup{gr}(y_{\g(C,b),\tau}))
+n(k-n)-\textup{sp}(C,b^0)\textup{sp}(C,d^0)}
d_{\g(C,b),\tau,s|_b}\circ d_{\g(C,c),\sigma,s|_c}(x_{\g(C,c),\sigma,\etaup}\wedge x_{\g(C,b),\tau,\theta}).
\end{align*}
Wegen Gleichung (\ref{grgl}) gilt:
\begin{align*}
&(\textup{gr}(x_{\g(C,c),\sigma,\etaup})+\textup{gr}(y_{\g(C,c),\sigma}))
(\textup{gr}(x_{\g(C,b),\tau,\theta})+\textup{gr}(y_{\g(C,b),\tau}))\\
\equiv&(\textup{sp}(C,b^0)-n+1)(\textup{sp}(C,d^0)-(k-n)+1)\qquad(\textup{mod }2)
\end{align*}
Es folgt Gleichung (\ref{ddgl}). Sei $D$ nun zusammenhängend.
\begin{lemma}
Gehen $D_1,\ D_2$ durch Rotation auseinander hervor (siehe Abbildung \ref{rot}), dann gilt Satz \ref{exddn} für $(D_1,\alpha,\beta)$ genau dann, wenn er für $(D_2,\alpha,\beta)$ gilt.
\end{lemma}
\begin{proof}
Sei $C_1$ orientierte Konfiguration mit $\overline{C_1}=D_1$ und $s$ Kantenzuordnung vom Typ bezüglich $C_1$. Dann sei $C_2$ orientierte Konfiguration mit $\overline{C_2}=D_2$, die sich wie in Abbildung \ref{rotor} aus $C_1$ ergibt.
\begin{figure}[htbp]
\centering
\includegraphics{graph.17}
\caption{}
\label{rotor}
\end{figure}
Dann ist $s$ Kantenzuordnung vom Typ Y bezüglich $C_2$ und $d_i(C_1,s)=d_i(C_2,s)$ für $i=1,\dots,k$.
\end{proof}
Für ein Monom $\alpha\in\Lambda\V(D)$ sagen wir Satz \ref{exddn} gilt für $(D,\alpha)$, falls Satz \ref{exddn} für alle Monome $\beta\in\Lambda\V(D^*)$ für $(D,\alpha,\beta)$ gilt.\\
Für ein Monom $\beta\in\Lambda\V(D^*)$ sagen wir Satz \ref{exddn} gilt für $(D,\beta)$, falls Satz \ref{exddn} für alle Monome $\alpha\in\Lambda\V(D)$ für $(D,\alpha,\beta)$ gilt.\\
\begin{lemma}\label{gr1}
Enthalte $D$ einen Kreis $x_1$, der nur von einem Bogen $\gammaup$ getroffen wird. Wir nennen $x_1$ dann einen Grad 1 Kreis von $D$. Sei $x_2$ der andere Kreis, der von $\gammaup$ getroffen wird und $\alpha\in\Lambda\V(D)$ Monom.
\begin{enumerate}
\item 
Ist $\alpha$ nicht durch $x_1$ teilbar, dann gilt Satz \ref{exddn} für $(D,\alpha)$.
\item 
Ist $\alpha$ durch $x_2$ teilbar, dann gilt Satz \ref{exddn} für $(D,\alpha)$.
\end{enumerate}
\end{lemma}
\begin{proof}
Ohne Einschränkung sei $\gammaup$ der erste Bogen von $D$. Seien $C,\ s$ wie in Satz \ref{exddn}.
\begin{enumerate}
\item 
Seien $a=(0,*,\dots,*),\ b=(1,*,\dots,*)\in\sum(k,k-1)$, sowie $c=(*,0,\dots,0),\ d=(*,1,\dots,1)\in\sum(k,1)$. Dann gilt:
\[\sum_{i=1}^{k-1}d_{k-i}(C,s)\circ d_i(C,s)(\alpha)=-d_{\g(C,b),s|_b}\circ d_{\g(C,c),s|_c}(\alpha)+(-1)^{k-1}d_{\g(C,d),s|_d}\circ d_{\g(C,a),s|_a}(\alpha)\]
Es ist $\textup{akt}(\g(C,a))=\textup{akt}(\g(C,b))$ und $\textup{akt}(\g(C,c))=\textup{akt}(\g(C,d))$. Sei $\theta$ ein zulässiger Kantenzug für $\g(C,a)$. Dann ist $s|_b(\theta)s(c)=(-1)^{k-1}s(d)s|_a(\theta)$. Somit ist $\sum_{i=1}^{k-1}d_{k-i}(C,s)\circ d_i(C,s)(\alpha)=0$.
\item 
Wegen 1. genügt es den Fall zu betrachten, dass $\alpha$ durch $x_1$ und $x_2$ teilbar ist. Für $\eps\in\sum(k,i)$ mit $i\in\{1,\dots,k-1\}$ und $\eps^0=(0,\dots,0)$ gilt: Ist $\eps_1=*$, so ist $d_{\gl,s|_\eps}(\alpha)=0$. Sei $\etaup\in\sum(k,k-i)$ mit $\etaup^0=\eps^1$. Ist $\eps_1\ne*$, so ist $d_{\gle,s|_\etaup}\circ d_{\gl,s|_\eps}(\alpha)=0$ wegen der Filtrierungsregel.
\end{enumerate}
\end{proof}
\begin{lemma}\label{dugr1}
Enthalte $D^*$ einen Kreis $x_1$, der nur von einem Bogen $\gammaup$ getroffen wird. Sei $x_2$ der andere Kreis, der von $\gammaup$ getroffen wird und $\beta\in\Lambda\V(D^*)$ Monom.
\begin{enumerate}
\item 
Ist $\beta$ durch $x_1$ teilbar, dann gilt Satz \ref{exddn} für $(D,\beta)$.
\item 
Ist $\beta$ nicht durch $x_2$ teilbar, dann gilt Satz \ref{exddn} für $(D,\beta)$.
\end{enumerate}
\end{lemma}
\begin{proof}
Ohne Einschränkungen sei $\gammaup$ der erste Bogen von $D^*$. Seien $C,\ s,\ \alpha$ wie in Satz \ref{exd}.
\begin{enumerate}
\item 
Seien $a,\ b,\ c,\ d$ wie im Beweis von Lemma \ref{gr1}. Dann ist der Koeffizient bei $\beta$ von $\sum_{i=1}^{k-1}d_{k-i}(C,s)\circ d_i(C,s)(\alpha)$ gleich dem Koeffizienten bei $\beta$ von
\[(-1)^{1+k}d_{\g(C,b),s|_b}\circ d_{\g(C,c),s|_c}(\alpha)+(-1)^{k-1}d_{\g(C,d),s|_d}\circ d_{\g(C,a),s|_a}(\alpha).\]
Es ist $\textup{akt}(\g(C,a))=\textup{akt}(\g(C,b))$ und $\textup{akt}(\g(C,c))=\textup{akt}(\g(C,d))$. Sei $\theta$ ein zulässiger Kantenzug für $\g(C,a)$. Dann ist $s|_b(\theta)s(c)=(-1)^{k-1-\textup{sp}(C,d^0)}s(d)s|_a(\theta)$. Mit Gleichung (\ref{grgl}) folgt, dass der Koeffizient bei $\beta$ von $\sum_{i=1}^{k-1}d_{k-i}(C,s)\circ d_i(C,s)(\alpha)$ trivial ist.
\item 
Wegen 1. genügt es den Fall zu betrachten, dass $\beta$ nicht durch $x_1$ und nicht durch $x_2$ teilbar ist. Die Aussage folgt aus dem Beweis von 2. aus Lemma \ref{gr1} und der Dualitätsregel.
\end{enumerate}
\end{proof}
\begin{lemma}\label{gr2}
Enthalte $D$ einen Kreis $x_1$, der von genau zwei Bögen $\gammaup_1,\ \gammaup_2$ getroffen wird. Weiterhin sollen $\gammaup_1,\ \gammaup_2$ jeweils nur einen Randpunkt auf $x_1$ haben. Seien $x,\ x'$ die anderen Kreise, die von $\gammaup_1,\ \gammaup_2$ getroffen werden und $\alpha\in\Lambda\V(D)$ ein Monom.
\begin{enumerate}
\item 
Ist $\alpha$ nicht durch $x_1$ teilbar, dann gilt Satz \ref{exddn} für $(D,\alpha)$.
\item 
Ist $\alpha$ durch $x_1$, $x$ und $x'$ teilbar, dann gilt Satz \ref{exddn} für $(D,\alpha)$.
\end{enumerate}
\end{lemma}
\begin{proof}
Ohne Einschränkung seien $\gammaup_1$ und $\gammaup_2$ der erste und zweite Bogen von $D$.
\begin{enumerate}
\item 
Sei $C$ so, dass $\gammaup_1^1$ und $\gammaup_2^0$ auf $x_1$ liegen. Sei $s$ Kantenzuordnung vom Typ Y bezüglich C. Seien $a=(*,0,\dots,0),\ b=(0,*,0,\dots,0)\in\sum(k,1)$, sowie $c=(1,*,\dots,*),\ d=(*,1,*,\dots,*)\in\sum(k,k-1)$. Dann gilt:
\[\sum_{i=1}^{k-1}d_{k-i}(C,s)\circ d_i(C,s)(\alpha)=-d_{\g(C,c),s|_c}\circ d_{\g(C,a),s|_a}(\alpha)-d_{\g(C,d),s|_d}\circ d_{\g(C,b),s|_b}(\alpha)\]
Liegen $x$ und $x'$ in der gleichen Zusammenhangskomponente von $S^2\setminus x_1$, dann ist $\g(C,c)=\g(C,d)$. Andernfalls gibt es 3 Möglichkeiten:
Entweder $\g(C,c)$ und $\g(C,d)$ sind beide von keinem Typ $\tau\in\textup{T}$ oder es gibt $p,\ q$, so dass $\g(C,c)$ und $\g(C,d)$ beide vom Typ $\textup{F}_{p,q}$ sind oder es gibt $p,\ q$, so dass $\g(C,c)$ und $\g(C,d)$ beide vom Typ $\textup{G}_{p,q}$ sind. Sei $\theta$ ein zulässiger Kantenzug für $\g(C,c)$. Dann ist $s|_c(\theta)s(a)=-s|_d(\theta)s(b)$ .Somit ist $\sum_{i=1}^{k-1}d_{k-i}(C,s)\circ d_i(C,s)(\alpha)=0$.
\item 
Sei $C$ so, dass $\gammaup_1^1$ und $\gammaup_2^1$ auf $x_1$ liegen. Sei $s$ Kantenzuordnung vom Typ Y bezüglich $C$. Sei $\eps\in\sum(k,i)$, mit $i\in\{1,\dots,k-1\}$ und $\eps^0=(0,\dots,0)$. Gilt $\eps_1=*$ oder $\eps_2=*$, dann folgt aus $d_{\gl,s|_\eps}(\alpha)\ne 0$ und der Tatsache, dass $\alpha$ durch $x_1,\ x,\ x'$ teilbar ist, dass $\gl$ vom Typ $\textup{B}_i$ oder $\textup{D}_{p,q}$ ist. Widerspruch. Gilt $(\eps_1,\eps_2)=(0,0)$, so sei $\etaup\in\sum(k,k-i)$ mit $\etaup^0=\eps^1$. Aus $d_{\gle,s|_\etaup}\circ d_{\gl,s|_\eps}(\alpha)\ne 0$, der Tatsache, dass $\alpha$ durch $x_1,\ x,\ x'$ teilbar ist und der Filtrierungsregel folgt, dass $\gle$ vom Typ $\textup{B}_{k-i}$ oder $\textup{D}_{p,q}$ ist. Widerspruch.
\end{enumerate}
\end{proof}
\begin{lemma}\label{dugr2}
Ersetze in Lemma \ref{gr2} $D$ durch $D^*$. Sei $\beta\in\Lambda\V(D^*)$ Monom.
\begin{enumerate}
\item 
Ist $\beta$ durch $x_1$ teilbar, dann gilt Satz \ref{exddn} für $(D,\beta)$.
\item 
Ist $\beta$ nicht durch $x_1$ und nicht durch $x$ und nicht durch $x'$ teilbar, dann gilt Satz \ref{exddn} für $(D,\beta)$.
\end{enumerate}
\end{lemma}
\begin{proof}
Ohne Einschränkung seien $\gammaup_1$ und $\gammaup_2$ der erste und zweite Bogen von $D^*$.
\begin{enumerate}
\item 
Sei $C$ so, dass $\gammaup_1^1$ und $\gammaup_2^0$ auf $x_1$ liegen. Sei $s$ Kantenzuordnung vom Typ Y bezüglich $C$ und $\alpha\in\Lambda\V(D)$ Monom. Seien $a=(*,1,\dots,1),\ b=(1,*,1,\dots,1)\in\sum(k,1)$, sowie $c=(0,*,\dots,*),\ d=(*,0,*,\dots,*)\in\sum(k,k-1)$. Wegen des Beweises von 1. aus Lemma \ref{gr2} und der Dualitätsregel ist der Koeffizient bei $\beta$ von $\sum_{i=1}^{k-1}d_{k-i}(C,s)\circ d_i(C,s)(\alpha)$ gleich dem Koeffizienten bei $\beta$ von $(-1)^{k-1}d_{\g(C,a),s|_a}\circ d_{\g(C,c),s|_c}(\alpha)+(-1)^{k-1}d_{\g(C,b),s|_b}\circ d_{\g(C,d),s|_d}(\alpha)$. Die Konfigurationen $m(\g(C,c)^*)$ und $m(\g(C,d)^*)$ wurden im Beweis des 1. Teils von Lemma \ref{gr2} besprochen. Sei $\theta$ ein zulässiger Kantenzug für $\g(C,c)$. Dann ist $s(a)s|_c(\theta)=s(b)s|_d(\theta)$. Es gibt einen kanonischen Isomorphismus $f:\Lambda\V(\g(C,a^0))\to\Lambda\V(\g(C,b^0))$. Es ist $s|_c(\theta)f\circ d_{\g(C,c),s|_c}=s|_d(\theta)d_{\g(C,d),s|_d}$. Sei $\lambda\in\Lambda\V(\g(C,a^0))$ Monom. Dann ist der Koeffizient bei $\beta$ von $s(a)d_{\g(C,a),s|_a}(\lambda)$ gleich dem Koeffizienten bei $\beta$ von $-s(b)d_{\g(C,b),s|_b}\circ f(\lambda)$.
\item 
Die Aussage folgt aus dem Beweis von 2. aus Lemma \ref{gr2} und der Dualitätsregel.
\end{enumerate}
\end{proof}
\begin{defi}
Für zwei Konfigurationen $D$, $D'$ sagen wir, diese gehen durch Tausch auseinander hervor, falls wir $D'$ durch endlich viele lokale Veränderungen wie in Abbildung \ref{tr} aus $D$ erhalten können.
\end{defi}
\begin{figure}[htbp]
\centering
\includegraphics{1bild.10}
\caption{Tausch}
\label{tr}
\end{figure}
\begin{lemma}\label{trlem}
Gehen $D$, $D'$ durch Tausch auseinander hervor, dann gilt Satz \ref{exddn} für $D$ genau dann, wenn er für $D'$ gilt.
\end{lemma}
\begin{proof}
Sei $D$ wie die linke Seite von Abbildung \ref{tr} und $D'$ wie die rechte Seite. Seien die Bezeichnungen der Bögen und Kanten wie in Abbildung \ref{trbez}.
\begin{figure}[htbp]
\centering
\includegraphics{graph.18}
\caption{}
\label{trbez}
\end{figure}
In der Ordnung der Bögen sollen $\gammaup$ und $\gammaup '$ an erster Stelle stehen. In der Ordnung der Kanten sollen $e_1/e_2$ und $e'_1/e'_2$ an erster/zweiter Stelle stehen. Seien $\alpha\in\Lambda\V(D)$, $\beta\in\Lambda\V(D^*)$, $\alpha '\in\Lambda\V(D')$, $\beta '\in\Lambda\V((D')^*)$ Monome. Nach Lemma \ref{gr1} genügt es den Fall zu betrachten, dass $\alpha$ durch $x_D(1)$, aber nicht durch $x_D(2)$ teilbar ist. Nach der Filtrierungsregel genügt es den Fall zu betrachten, dass $\beta$ durch $x_{D^*}(1)$ teilbar ist. Nach Lemma \ref{dugr1} genügt es den Fall zu betrachten, dass $\beta '$ durch $x_{(D')^*}(2)$, aber nicht durch $x_{(D')^*}(1)$ teilbar ist. Nach der Filtrierungsregel genügt es, den Fall zu betrachten, dass $\alpha '$ nicht durch $x_{D'}(1)$ teilbar ist. Seien $C$, $C'$ wie in Abbildung \ref{tror}.
\begin{figure}[htbp]
\centering
\includegraphics{graph.19}
\caption{}
\label{tror}
\end{figure}
Bei den anderen Bögen sollen sich $C$ und $C'$ nicht unterscheiden. Sei $s$ eine Kantenzuordnung vom Typ Y bezüglich $C$. Für $\eps\in\sum(k,1)$ sei
\[s'(\eps)=\begin{cases}
s(\eps)\qquad&\textup{falls }\eps_1\ne*,\\
(-1)^{\textup{sp}(C,\eps^0)}s(\eps)\qquad&\textup{falls }\eps_1=*.
\end{cases}\]
Man prüft leicht nach, dass $s'$ Kantenzuordnung vom Typ Y bezüglich $C'$ ist. Seien $a\in\sum(k,i)$ mit $i\in\{1,\dots,k-1\}$ und $b\in\sum(k,k-i)$ mit $a^1=b^0$. Sei weiterhin $\omega\in\Lambda\V(D)$, $\rho\in\Lambda\V(D^*)$ Monome, so dass $\omega$ nicht durch $x_D(1)$ oder $x_D(2)$ und $\rho$ nicht durch $x_{D^*}(1)$ teilbar ist. Wir werden nun zeigen, dass der Koeffizient bei $x_{D^*}(1)\wedge\rho$ von
\[(-1)^{(k-i+1)\textup{sp}(C,b^0)}d_{\g(C,b),s|_b}\circ d_{\g(C,a),s|_a}(x_D(1)\wedge\omega)\]
gleich dem Koeffizienten bei $x_{(D')^*}(2)\wedge\rho$ von
\[(-1)^k(-1)^{(k-i+1)\textup{sp}(C',b^0)}d_{\g(C',b),s'|_b}\circ d_{\g(C',a),s'|_a}(\omega)\]
ist.\\
Sei zunächst $a_1=*$. Dann ist $\textup{akt}(\g(C,b))=\textup{akt}(\g(C',b))$. Es gibt zwei Möglichkeiten: Entweder sind sowohl $\g(C,a)$, als auch $\g(C',a)$ von keinem Typ aus T oder es gibt $p\ge 1$, $q\ge 0$, so dass $\g(C,a)$ vom Typ $\textup{F}_{p,q}$ und $\g(C',a)$ vom Typ $\textup{F}_{p-1,q+1}$ ist. Sei $\theta$ ein zulässiger Kantenzug für $\g(C,a)$, so dass $\phi(\theta)(1)=1$. Dann ist $s|_a(\theta)=s'|_a(\theta)$ und $\textup{sp}(\g(C,a),\theta)=(-1)^{i-1}\textup{sp}(\g(C',a),\theta)$ und $\textup{sp}(C,b^0)+1=\textup{sp}(C',b^0)$. Es folgt die behauptete Gleichheit der Koeffizienten.\\
Sei nun $a_1=0$. Dann ist $\textup{akt}(\g(C,a))=\textup{akt}(\g(C',a))$. Für $\g(C,b)$ und $\g(C',b)$ gibt es die zwei beschriebenen Möglichkeiten. Sei $\theta$ ein zulässiger Kantenzug für $\g(C,b)$, so dass $\phi(\theta)(1)=1$. Dann ist $s|_b(\theta)=(-1)^{\textup{sp}(C,b^0)}s'|_b(\theta)$ und $\textup{sp}(\g(C,b),\theta)=(-1)^{k-i-1}\textup{sp}(\g(C',b),\theta)$ und $\textup{sp}(C,b^0)=\textup{sp}(C',b^0)$. Mit Gleichung (\ref{grgl}) folgt die behauptete Gleichheit der Koeffizienten.\\
Kehrt man bei $C$ die Orientierung von $\gammaup$ und bei $C'$ die Orientierung von $\gammaup '$ um, so erhält man wieder die Gleichheit der Koeffizienten, allerdings nun mit dem Vorzeichen $(-1)^{k+1}$ statt $(-1)^k$, da wir $\textup{F}_{p,q}$ und $\textup{F}_{p-1,q+1}$ durch $\textup{G}_{p,q}$ und $\textup{G}_{p-1,q+1}$ ersetzen müssen.
\end{proof}
\begin{lemma}\label{alle}
Hat $D$ mindestens zwei Kreise und ist das Monom $\alpha\in\Lambda\V(D)$ durch alle Kreise von $D$ teilbar, dann gilt Satz \ref{exddn} für $(D,\alpha)$.
\end{lemma}
\begin{proof}
Sei $x_1$ der erste Kreis von $D$. Sei $n$ die Anzahl der Bögen, die genau einen Randpunkt auf $x_1$ haben. Es gilt $n\ge 1$. Ohne Einschränkung stehen diese $n$ Bögen in der Ordnung der Bögen an den Stellen $1$ bis $n$. Sei $C$ so, dass alle diese Bögen ihren Anfangspunkt auf $x_1$ haben. Sei $s$ Kantenzuordnung vom Typ Y bezüglich $C$. Sei $\eps\in\sum(k,i)$ für $i\in\{1,\dots,k-1\}$ mit $\eps^0=(0,\dots,0)$. Aus $d_{\g(C,\eps),s|_\eps}(\alpha)\ne 0$ folgt, dass $\gl$ vom Typ $\textup{B}_i$ oder $\textup{D}_{p,q}$ ist. Also ist $\eps_j=0$ für $j\in\{1,\dots,n\}$. Sei $\etaup\in\sum(k,k-i)$ mit $\etaup^0=\eps^1$. Aus $d_{\g(C,\etaup),s|_\etaup}\circ d_{\g(C,\eps),s|_\eps}(\alpha)\ne 0$ und der Filtrierungsregel folgt, dass $\gle$ vom Typ $\textup{B}_{k-i}$ oder $\textup{D}_{p,q}$ ist. Dies ist aber ein Widerspruch zu $\eps_j=*$ für $j\in\{1,\dots,n\}$.
\end{proof}
\begin{kor}\label{eins}
Hat $D^*$ mindestens zwei Kreise, dann gilt Satz \ref{exddn} für $(D,1\in\Lambda\V(D^*))$.
\end{kor}
\begin{lemma}\label{mlem}
Enthalte $D$ mindestens eine der lokalen Konfigurationen $M_1,\ M_2,\dots,M_9$ aus Abbildung \ref{mabb}. Dann gilt Satz \ref{exddn} für $D$.
\end{lemma}
\begin{figure}[htbp]
\centering
\includegraphics[scale=0.4]{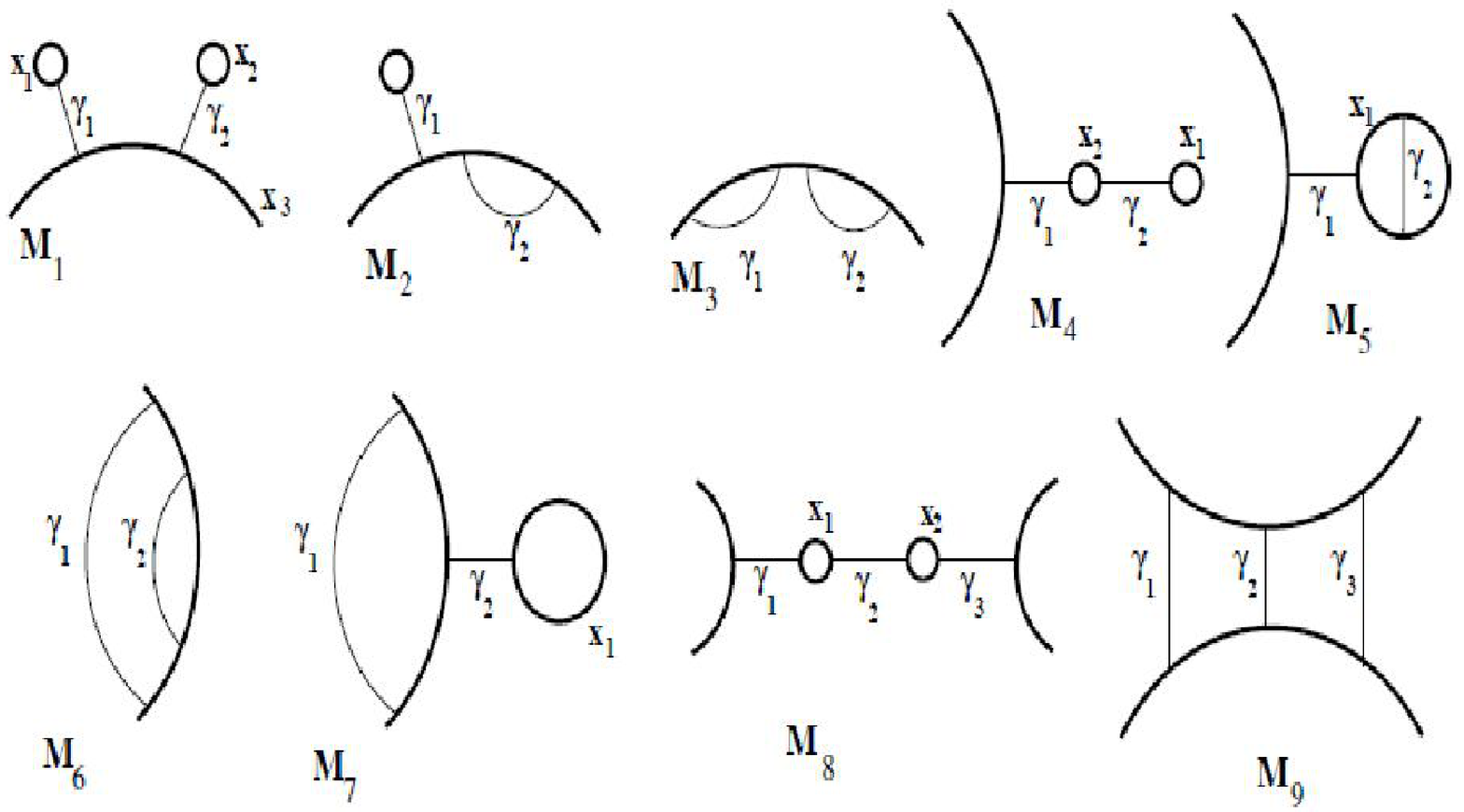}
\caption{aus \cite{szabo}}
\label{mabb}
\end{figure}
\begin{proof}
Durch Rotation und Tausch gehen $M_2$ und $M_3$ aus $M_1$, $M_5$ aus $M_4$, sowie $M_7$ aus $M_6$ hervor.\\
$M_1$: Nach Lemma \ref{gr1} reicht es den Fall zu betrachten, dass $\alpha$ durch $x_1$ und $x_2$ teilbar ist. Sei $C$ so, dass $\gammaup_1^0$ auf $x_1$ und $\gammaup_2^1$ auf $x_2$ liegt. Sei $s$ Kantenzuordnung vom Typ Y bezüglich $C$. Ohne Einschränkungen seien $\gammaup_1$ und $\gammaup_2$ der erste und zweite Bogen von $D$. Seien $\eps\in\sum(k,i)$ für $i\in\{1,\dots,k-1\}$ und $\etaup\in\sum(k,k-i)$ mit $\eps^1=\etaup^0$. Ist $\eps_1=\eps_2=*$, so ist $\gl$ von keinem Typ aus T. Ist $\etaup_1=\etaup_2=*$, so ist $\gle$ von keinem Typ aus T. Ist $(\eps_1,\eps_2)=(*,0)$ oder$(\eps_1,\eps_2)=(0,*)$, so ist wegen der Filtrierungsregel $d_{\g(C,\etaup),s|_\etaup}\circ d_{\g(C,\eps),s|_\eps}(\alpha)=0$.\\
$M_4$: Ist $\alpha$ durch $x_2$ teilbar, so kann man den 2. Teil von Lemma \ref{gr1} anwenden. Ist $\alpha$ nicht durch $x_2$ teilbar, so kann man den 1. Teil von Lemma \ref{gr2} anwenden.\\
$M_6$: 2. Teil von Lemma \ref{dugr1} und 1. Teil von Lemma \ref{dugr2}.\\
$M_8$: Nach Lemma \ref{gr2} reicht es den Fall zu betrachten, dass $\alpha$ durch $x_1$ und $x_2$ teilbar ist. Sei $C$ so, dass $\gammaup_1^1$ und $\gammaup_2^1$ auf $x_1$ und $\gammaup_3^0$ auf $x_2$ liegen. Sei $s$ Kantenzuordnung vom Typ Y bezüglich $C$. Ohne Einschränkungen seien $\gammaup_1$, $\gammaup_2$ und $\gammaup_3$ der erste, zweite und dritte Bogen von $D$. Seien $\eps\in\sum(k,i)$ für $i\in\{1,\dots,k-1\}$ und $\etaup\in\sum(k,k-i)$ mit $\eps^1=\etaup^0$. Sind mindestens 2 der ersten 3 Komponenten von $\eps$ ein $*$, dann ist $d_{\g(C,\eps),s|_\eps}(\alpha)=0$. Ist $(\eps_1,\eps_2,\eps_3)=(0,*,0)$, so ist $d_{\g(C,\eps),s|_\eps}(\alpha)=0$. Ist $(\eps_1,\eps_2,\eps_3)=(*,0,0)$ oder $(\eps_1,\eps_2,\eps_3)=(0,0,*)$, so ist $d_{\g(C,\etaup),s|_\etaup}\circ d_{\g(C,\eps),s|_\eps}(\alpha)=0$ wegen der Filtrierungsregel. Ist $\eps_1=\eps_2=\eps_3=0$, so ist $\g(C,\etaup)$ von keinem Typ aus T.\\
$M_9$: Folgt mittels Lemma \ref{dugr2} und der Dualitätsregel aus $M_8$.
\end{proof}
\begin{lemma}
Ist $D$ dreidimensional, so gilt Satz \ref{exddn} für $D$.
\end{lemma}
\begin{proof}
In Abbildung \ref{dreidim} ist aufgelistet, welche Möglichkeiten es für eine zusammenhängende, nichtorientierte, dreidimensionale Konfiguration modulo Rotation gibt. Durch Lemma \ref{mlem} werden die Fälle 1, 2, 3, 5, 6, 7, 11, 13, 14, 15, 16 und 17 abgedeckt. Sei $C$ wie in Abbildung \ref{ordreidim}.
\begin{figure}[htbp]
\centering
\includegraphics{graph.20}\qquad\qquad
\includegraphics{graph.38}\\\ \\\ \\\ \\
\includegraphics{graph.39}\qquad\qquad
\includegraphics{graph.40}\\\ \\\ \\\ \\
\includegraphics{graph.41}\qquad\qquad
\includegraphics{graph.42}
\caption{}
\label{ordreidim}
\end{figure}
Sei $s$ Kantenzuordnung vom Typ Y bezüglich $C$. Ohne Einschränkungen seien $\gammaup_1$, $\gammaup_2$ und $\gammaup_3$ der erste, zweite und dritte Bogen von $D$. Seien $\eps\in\sum(3,i)$ für $i\in\{1,2\}$ und $\etaup\in\sum(3,3-i)$ mit $\eps^1=\etaup^0$.\\
Fall 4: Nach Lemma \ref{gr1} und Lemma \ref{gr2} reicht es $\alpha=x_1\wedge x_3$ zu betrachten. Dann ist aber stets $d_{\g(C,\etaup),s|_\etaup}\circ d_{\g(C,\eps),s|_\eps}(\alpha)=0$.\\
Fall 12: Folgt mit Lemma \ref{dugr1} und Lemma \ref{dugr2} und der Dualitätsregel aus Fall 4.\\
Fall 8: Nach Lemma \ref{gr2} reicht es $\alpha=x_1$ zu betrachten. Dann ist aber stets $d_{\g(C,\etaup),s|_\etaup}\circ d_{\g(C,\eps),s|_\eps}(\alpha)=0$.\\
Fall 18: Folgt mit Lemma \ref{dugr2} und der Dualitätsregel aus Fall 8.\\
Fall 9: Nach Lemma \ref{gr1} reicht es $\alpha=x_1$ zu betrachten. Sei $y$ der Kreis von $C^*$. Es gilt:
\[d_2(C,s)\circ d_1(C,s)+d_1(C,s)\circ d_2(C,s)
=d_{\g(C,(1,*,*)),s|_{(1,*,*)}}\circ d_{\g(C,(*,0,0)),s|_{(*,0,0)}}
+d_{\g(C,(*,1,1)),s|_{(*,1,1)}}\circ d_{\g(C,(0,*,*)),s|_{(0,*,*)}}\]
In Abbildung \ref{f9} sind $\g(C,(1,*,*))$ und $\g(C,(*,1,1))$ dargestellt.
\begin{figure}[htbp]
\centering
\includegraphics{graph.21}
\caption{}
\label{f9}
\end{figure}
Es ist $d_{\g(C,(*,0,0)),s|_{(*,0,0)}}(x_1)=s(*,0,0)(z_3-z_2)\wedge z_1$ und $d_{\g(C,(1,*,*)),s|_{(1,*,*)}}((z_3-z_2)\wedge z_1)=s(1,*,0)s(1,1,*)y$. Weiterhin ist $d_{\g(C,(0,*,*)),s|_{(0,*,*)}}(x_1)=-s(0,*,0)s(0,1,*)w_1$ und $d_{\g(C,(*,1,1)),s|_{(*,1,1)}}(w_1)=s(*,1,1)y$. Außerdem gilt $s(*,0,0)s(1,*,0)s(1,1,*)=s(0,*,0)s(0,1,*)s(*,1,1)$.\\
Fall 10: Nach Lemma \ref{gr1} reicht es $\alpha=x_1$ zu betrachten. Sei $y$ der Kreis von $C^*$. Es gilt:
\[d_2(C,s)\circ d_1(C,s)+d_1(C,s)\circ d_2(C,s)
=d_{\g(C,(1,*,1)),s|_{(1,*,1)}}\circ d_{\g(C,(*,0,*)),s|_{(*,0,*)}}
+d_{\g(C,(*,1,1)),s|_{(*,1,1)}}\circ d_{\g(C,(0,*,*)),s|_{(0,*,*)}}\]
In Abbildung \ref{f10} sind $\g(C,(1,*,1))$ und $\g(C,(*,1,1))$ dargestellt.
\begin{figure}[htbp]
\centering
\includegraphics{graph.22}
\caption{}
\label{f10}
\end{figure}
Es ist $d_{\g(C,(*,0,*)),s|_{(*,0,*)}}(x_1)=s(0,0,*)s(*,0,1)z_1$ und $d_{\g(C,(1,*,1)),s|_{(1,*,1)}}(z_1)=s(1,*,1)y$. Weiterhin ist $d_{\g(C,(0,*,*)),s|_{(0,*,*)}}(x_1)=s(0,0,*)s(0,*,1)w_1$ und $d_{\g(C,(*,1,1)),s|_{(*,1,1)}}(w_1)=s(*,1,1)y$. Außerdem gilt $s(*,0,1)s(1,*,1)=-s(0,*,1)s(*,1,1)$.\\
Fall 19: Es ist $m(C^*)$ gleich dem $C$ aus Fall 10, dies liefert uns Bezeichnungen für die Kreise von $C^*$. Es gilt:
\[d_2(C,s)\circ d_1(C,s)+d_1(C,s)\circ d_2(C,s)
=d_{\g(C,(*,1,*)),s|_{(*,1,*)}}\circ d_{\g(C,(0,*,0)),s|_{(0,*,0)}}
+d_{\g(C,(1,*,*)),s|_{(1,*,*)}}\circ d_{\g(C,(*,0,0)),s|_{(*,0,0)}}\]
In Abbildung \ref{f19} sind $\g(C,(*,1,*))$ und $\g(C,(1,*,*))$ dargestellt.
\begin{figure}[htbp]
\centering
\includegraphics{graph.23}
\caption{}
\label{f19}
\end{figure}
Es ist $d_2(C,s)\circ d_1(C,s)(y)=0$.
Es ist $d_{\g(C,(0,*,0)),s|_{(0,*,0)}}(1)=s(0,*,0)(z_1-z_2)$ und $d_{\g(C,(*,1,*)),s|_{(*,1,*)}}(z_1-z_2)=-s(*,1,0)s(1,1,*)x_2$. Weiterhin ist $d_{\g(C,(*,0,0)),s|_{(*,0,0)}}(1)=s(*,0,0)(w_1-w_2)$ und $d_{\g(C,(1,*,*)),s|_{(1,*,*)}}(w_1-w_2)=-s(1,*,0)s(1,1,*)x_2$. Außerdem gilt $s(0,*,0)s(*,1,0)=-s(*,0,0)s(1,*,0)$.\\
Fall 20: Es ist $m(C^*)$ gleich dem $C$ aus Fall 9, dies liefert uns Bezeichnungen für die Kreise von $C^*$. Es gilt:
\[d_2(C,s)\circ d_1(C,s)+d_1(C,s)\circ d_2(C,s)
=d_{\g(C,(*,1,1)),s|_{(*,1,1)}}\circ d_{\g(C,(0,*,*)),s|_{(0,*,*)}}
+d_{\g(C,(1,*,*)),s|_{(1,*,*)}}\circ d_{\g(C,(*,0,0)),s|_{(*,0,0)}}\]
In Abbildung \ref{f20} sind $\g(C,(*,1,1))$ und $\g(C,(1,*,*))$ dargestellt.
\begin{figure}[htbp]
\centering
\includegraphics{graph.24}
\caption{}
\label{f20}
\end{figure}
Es ist $d_2(C,s)\circ d_1(C,s)(y)=0=d_1(C,s)\circ d_2(C,s)(y)$.
Es ist $d_{\g(C,(0,*,*)),s|_{(0,*,*)}}(1)=s(0,*,0)s(0,1,*)z_2$ und $d_{\g(C,(*,1,1)),s|_{(*,1,1)}}(z_2)=s(*,1,1)x_2$. Weiterhin ist $d_{\g(C,(*,0,0)),s|_{(*,0,0)}}(1)=s(*,0,0)(w_1-w_2)$ und $d_{\g(C,(1,*,*)),s|_{(1,*,*)}}(w_1-w_2)=-s(1,*,0)s(1,1,*)x_2$. Außerdem gilt $s(0,*,0)s(0,1,*)s(*,1,1)=s(*,0,0)s(1,*,0)s(1,1,*)$.
\end{proof}
Sei ab jetzt $k\ge 4$.
\begin{lemma}\label{1km1}
Seien $D$, $\alpha$, $\beta$ wie in Satz \ref{exddn}. Es gebe eine orientierte Konfiguration $B$, so dass $\overline{B}=D$ und eine Kantenzuordnung $t$ vom Typ Y bezüglich $B$, so dass der Koeffizient von $d_1(B,t)\circ d_{k-1}(B,t)(\alpha)$ bei $\beta$ nicht trivial ist. Dann gilt Satz \ref{exddn} für $(D,\alpha,\beta)$.
\end{lemma}
\begin{proof}
Es gibt $\eps\in\sum(k,k-1)$ und $\etaup\in\sum(k,1)$ mit $\eps^1=\etaup^0$ und ein Monom $z\in\Lambda\V(\g(D,\eps^1))$, so dass der Koeffizient bei $z$ von $d_{\g(B,\eps),t|_\eps}(\alpha)$, sowie der Koeffizient bei $\beta$ von $d_{\g(B,\etaup),t|_\etaup}(z)$ nicht trivial ist. Der Bogen von $\g(D,\etaup)$ sei mit $\delta$ bezeichnet. Es gibt 3 Fälle:
\begin{description}
\item[(i)]
$\delta$ verbindet zwei aktive Kreise von $(\g(D,\eps))^*$.
\item[(ii)]
$\delta$ verbindet einen aktiven Kreis von $(\g(D,\eps))^*$ mit einem passiven Kreis.
\item[(iii)]
$\delta$ hat beide Randpunkte auf dem gleichen aktiven Kreis von $(\g(D,\eps))^*$.
\end{description}
$\g(B,\eps)$ ist von genau einem Typ $\tau\in\textup{T}$. Unterscheide nun die möglichen Fälle.\\
Sei $\tau=\textup{A}_{k-1}$. Im Fall (i) ist $\g(D,\eps)$ aktiv. Es folgt $\alpha=1$, $z=1$, $\beta=1$. Nun können wir Korollar \ref{eins} anwenden. In den Fällen (ii) und (iii) enthält $D$ die $M_9$ Konfiguration aus Abbildung \ref{mabb}.\\
Sei $\tau=\textup{B}_{k-1}$. Im Fall (ii) können wir den 2. Teil von Lemma \ref{gr1} anwenden. In den Fällen (i) und (iii) können wir Lemma \ref{alle} anwenden.\\
Sei $\tau=\textup{C}_{p,q}$. Im Fall (i) müssen wir nur $\alpha=1$, $\beta=1$ betrachten. $\g(D,\eps^1)$ hat $k-2$ Kreise, somit hat $D^*$ $k-3$ Kreise. Für $k\ge 5$ können wir Korollar \ref{eins} anwenden. Für $k=4$ gibt es für $D$ die zwei Möglichkeiten aus Abbildung \ref{ci}.
\begin{figure}[htbp]
\centering
\includegraphics[scale=2]{1bild.11}
\caption{}
\label{ci}
\end{figure}
Sei $C$ wie in Abbildung \ref{cior} und sei $s$ eine Kantenzuordnung vom Typ Y bezüglich $C$.
\begin{figure}[htbp]
\centering
\includegraphics[scale=2]{1bild.12}
\caption{}
\label{cior}
\end{figure}
Sei $a\in\sum(4,3)$ mit $a^0=(0,0,0,0)$. Dann ist $\g(C,a)$ von keinem Typ aus T. $d_3(C,s)\circ d_1(C,s)(1)$ hat bei $1$ den Koeffizienten $0$, wegen der Filtrierungsregel. Für die rechte Seite von Abbildung \ref{cior} gilt: Ist $a\in\sum(4,2)$ mit $a^0=(0,0,0,0)$, so ist $\g(C,a)$ von keinem Typ aus T. Für die linke Seite von Abbildung \ref{cior} gilt: Es gibt genau ein $a\in\sum(4,2)$ mit $a^0=(0,0,0,0)$, so dass $\g(C,a)$ von einem Typ aus T ist. Der Typ ist $\textup{F}_{0,2}$. Somit ist der Koeffizient von $d_2(C,s)\circ d_2(C,s)(1)$ bei $1$ trivial, wegen der Filtrierungsregel.\\
Betrachte nun Fall (ii). Für $q\ge 5$ oder $p\ge 3$ enthält $D$ die Konfiguration $M_9$. Modulo Rotation können alle Möglichkeiten für $D$, die nicht $M_9$ enthalten, aus der linken Konfiguration in Abbildung \ref{cii} durch Entfernen von Bögen gewonnen werden.
\begin{figure}[htbp]
\centering
\includegraphics[scale=0.4]{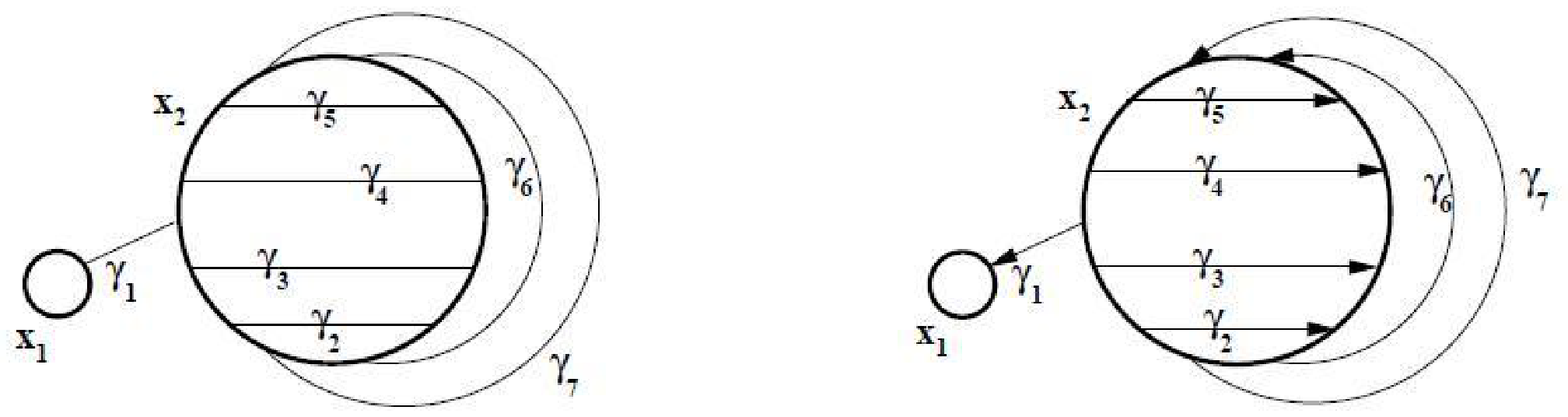}
\caption{aus \cite{szabo}}
\label{cii}
\end{figure}
Wir können sogar verlangen, dass höchstens zwei der vier Bögen aus dem Innern des großen Kreises entfernt werden. Eine orientierte Konfiguration $C$ mit $\overline{C}=D$ ergibt sich aus der rechten Seite von Abbildung \ref{cii}.
Sei $a\in\sum(k,i)$ mit $i\in\{1,\dots,k-1\}$ und $a^0=(0,\dots,0)$ und so dass $\g(C,a)$ von einem Typ aus T ist. Dann ist $i=1$ oder $a$ entspricht den Bögen $\gammaup_1$, $\gammaup_2$ oder $a$ entspricht den Bögen $\gammaup_1$, $\gammaup_3$. Sei $b\in\sum(k,k-i)$ mit $b^0=a^1$. Dann ist $\g(C,b)$ von keinem Typ aus T.\\
Im Fall (iii) lässt sich $D$ per Tausch in Fall (ii) überführen oder $D$ enthält die Konfiguration $M_9$ oder $D$ ist wie in Abbildung \ref{ciiid}.
\begin{figure}[htbp]
\centering
\includegraphics{1bild.14}
\caption{}
\label{ciiid}
\end{figure}
Dann ist $D^*$ wie in Abbildung \ref{ciiids}.
\begin{figure}[htbp]
\centering
\includegraphics{graph.26}
\caption{}
\label{ciiids}
\end{figure}
Nach Lemma \ref{dugr2} genügt es $\beta=y_1$ zu betrachten. Nach der Filtrierungsregel genügt es $\alpha=1$ zu betrachten. Sei $C$ wie in Abbildung \ref{ciiior}.
\begin{figure}[htbp]
\centering
\includegraphics{1bild.15}
\caption{}
\label{ciiior}
\end{figure}
Sei $s$ eine Kantenzuordnung vom Typ Y bezüglich $C$. Sei $a\in\sum(4,i)$ mit $i\in\{2,3\}$ und $a^0=(0,0,0,0)$. Dann ist $\g(C,a)$ von keinem Typ aus T. $d_3(C,s)\circ d_1(C,s)(1)$ hat bei $y_1$ den Koeffizienten $0$, wegen der Filtrierungsregel.\\
Sei $\tau=\textup{D}_{p,q}$. Dann ist Fall (i) nicht möglich. Im Fall (ii) können wir den 2. Teil von Lemma \ref{gr1} anwenden. Im Fall (iii) können wir Lemma \ref{alle} anwenden.\\
Sei $\tau=\textup{F}_{p,q}$ oder $\tau=\textup{G}_{p,q}$. $D$ enthalte, auch nach Rotation, keine der lokalen Konfigurationen aus Abbildung \ref{mabb}. Dann gibt es für $D$ bis auf Rotation und Tausch folgende Möglichkeiten. $D$ ist eine der vier Konfigurationen aus Abbildung \ref{fg} (wobei wir die Orientierung der Bögen vergessen) oder $D$ ist die 3. Konfiguration aus Abbildung \ref{fg}, wobei $\gammaup_2$ und $x_2$ entfernt seien oder $D$ ist die 4. Konfiguration, wobei $\gammaup_2$ und $x_2$ entfernt seien oder $D$ ist die 4. Konfiguration, wobei $\gammaup_3$ und $x_3$ entfernt seien.
\begin{figure}[htbp]
\centering
\includegraphics[scale=0.4]{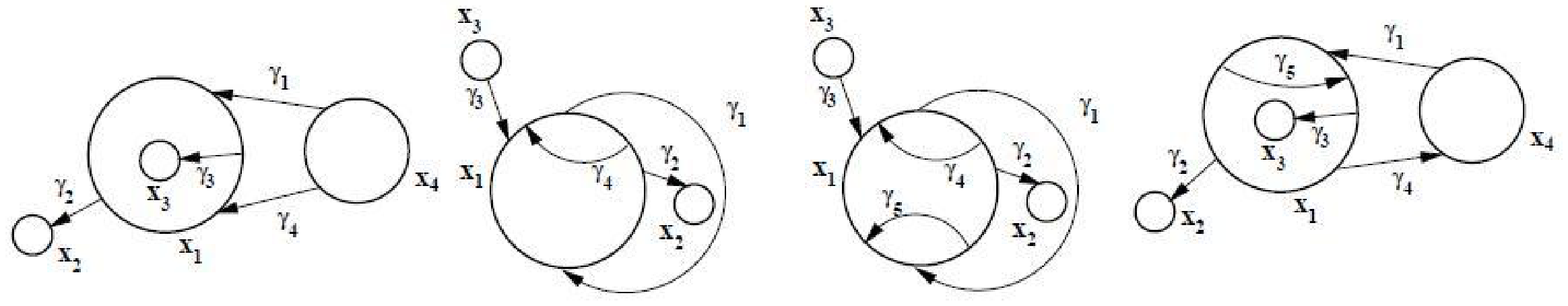}
\caption{aus \cite{szabo}}
\label{fg}
\end{figure}
Nach Lemma \ref{gr1} und Lemma \ref{gr2} müssen wir nur den Fall betrachten, dass $\alpha$ nicht durch $x_1$, aber durch alle anderen Kreise von $D$ teilbar ist. Sei $C$ wie in Abbildung \ref{fg}.
Sei $s$ Kantenzuordnung vom Typ Y bezüglich $C$. Sei $a\in\sum(k,i)$ für $i\in\{1,\dots,k-1\}$ und $b\in\sum(k,k-i)$ mit $a^1=b^0$. Man prüft leicht nach, dass $d_{\g(C,b),s|_b}\circ d_{\g(C,a),s|_a}(\alpha)$ stets trivial ist.
\end{proof}
\begin{kor}
Seien $D$, $\alpha$, $\beta$ wie in Satz \ref{exddn}. Es gebe eine orientierte Konfiguration $B$, so dass $\overline{B}=D$ und eine Kantenzuordnung $t$ vom Typ Y bezüglich $B$, so dass der Koeffizient von $d_{k-1}(B,t)\circ d_{1}(B,t)(\alpha)$ bei $\beta$ nicht trivial ist. Dann gilt Satz \ref{exddn} für $(D,\alpha,\beta)$.
\end{kor}
\begin{proof}
Beweis von Lemma \ref{1km1} und Dualitätsregel.
\end{proof}
\begin{lemma}
Seien $D$, $\alpha$, $\beta$ wie in Satz \ref{exddn}. Es gebe eine orientierte Konfiguration $B$, so dass $\overline{B}=D$ und eine Kantenzuordnung $t$ vom Typ Y bezüglich $B$ und $n\in\{2,\dots,k-2\}$, so dass der Koeffizient von $d_{k-n}(B,t)\circ d_{n}(B,t)(\alpha)$ bei $\beta$ nicht trivial ist. Dann gilt Satz \ref{exddn} für $(D,\alpha,\beta)$.
\end{lemma}
\begin{proof}
Es gibt $\eps\in\sum(k,n)$ und $\etaup\in\sum(k,k-n)$ mit $\eps^1=\etaup^0$ und ein Monom $z\in\Lambda\V(\g(D,\eps^1))$ und $Q=(\tau,\sigma)\in\textup{T}^2$, so dass der Koeffizient bei $z$ von $d_{\g(B,\eps),\tau,t|_\eps}(\alpha)$, sowie der Koeffizient bei $\beta$ von $d_{\g(B,\etaup),\sigma,t|_\etaup}(z)$ nicht trivial ist. Seien
\begin{align*}
T_1&\defeq\{\textup{A}_m,\ \textup{C}_{p,q}\ \ |\ \ m,p,q\in\mathbb{N},\ p\le q\}\subset\textup{T},\\
T_2&\defeq\{\textup{B}_m,\ \textup{D}_{p,q}\ \ |\ \ m,p,q\in\mathbb{N},\ p\le q\}\subset\textup{T}.
\end{align*}
Da $D$ zusammenhängend ist, gilt $Q\notin\textup{T}_1\times\textup{T}_2$ und $Q\notin\textup{T}_2\times\textup{T}_1$. Im Fall $Q\in\textup{T}_2\times\textup{T}_2$ ist $\alpha$ durch alle Kreise von $D$ teilbar. Nach Lemma \ref{alle} brauchen wir nur den Fall zu betrachten, dass $D$ genau einen Kreis hat. Also ist $Q=(\textup{D}_{1,1},\textup{D}_{1,1})$ und es gibt für $D$ die fünf Möglichkeiten aus Abbildung \ref{t2t2} (wobei wir die Orientierung der Bögen vergessen).
\begin{figure}[htbp]
\centering
\includegraphics[scale=0.4]{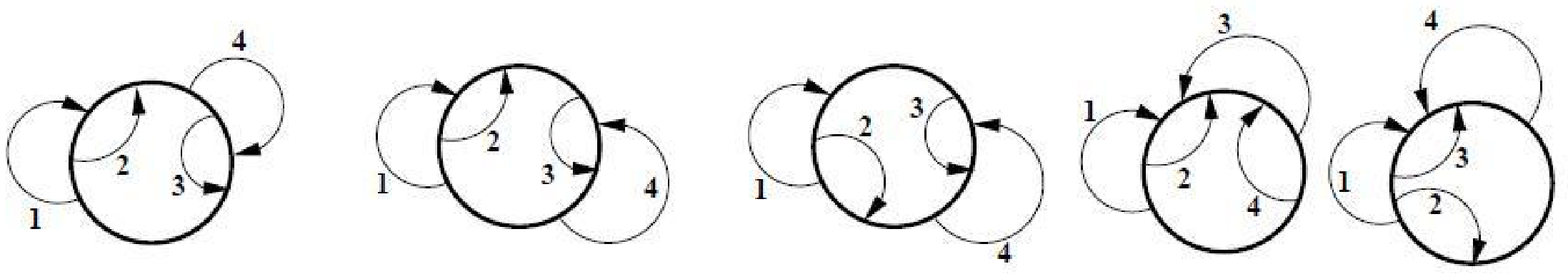}
\caption{aus \cite{szabo}}
\label{t2t2}
\end{figure}
Sei $C$ wie in Abbildung \ref{t2t2}.
Sei $a\in\sum(4,i)$ mit $i\in\{2,3\}$ und $a^0=(0,0,0,0)$. Dann ist $\g(C,a)$ von keinem Typ aus T oder vom Typ $\textup{F}_{0,2}$ oder vom Typ $\textup{G}_{0,2}$. Sei $b\in\sum(4,3)$ mit $b^1=(1,1,1,1)$. Dann ist $\g(C,b)$ von keinem Typ aus T. Aus der Behandlung des Falles $Q\in\textup{T}_2\times\textup{T}_2$ folgt der Fall $Q\in\textup{T}_1\times\textup{T}_1$. Sei
\[T_3\defeq\{\textup{F}_{p,q},\ \textup{G}_{p,q}\ \ |\ \ p,q\in\mathbb{N}_0,\ p+q\ge 1\}\subset\textup{T}.\]
Sei $Q\in\textup{T}_3\times\{\textup{A}_{k-n}\}$. Dann ist $z$ durch den Kreis $x$ teilbar, der von mindestens zwei Bögen aus $(\g(D,\eps))^*$ getroffen wird. Also ist $x$ passiver Kreis von $\g(D,\etaup)$. Ist $k-n\ge 3$, so enthält $D$ die Konfiguration $M_9$. Sei nun $k-n=2$. Sei zunächst einer der beiden aktiven Kreise von $\g(D,\etaup)$ passiv in $(\g(D,\eps))^*$. Dieser Kreis sei mit $y$ bezeichnet. Nach Lemma \ref{gr2} genügt es, den Fall zu betrachten, dass $\alpha$ durch $y$ teilbar ist. Dann ist aber auch $z$ durch $y$ teilbar. Dies ist ein Widerspruch zu $d_{\g(B,\etaup),\textup{A}_2,t|_\etaup}(z)\ne 0$. Seien nun beide aktiven Kreise von $\g(D,\etaup)$ aktiv in $(\g(D,\eps))^*$. $D$ enthalte, auch nach Rotation, keine der lokalen Konfigurationen aus Abbildung \ref{mabb}. Dann gibt es für $D$ bis auf Rotation und Tausch folgende Möglichkeiten. $D$ ist die Konfiguration aus Abbildung \ref{t3a} (wobei wir die Orientierung der Bögen vergessen) oder $D$ ist diese Konfiguration, wobei der 2. Bogen und $x_2$ entfernt seien, oder $D$ ist diese Konfiguration, wobei der 2. Bogen, der 3. Bogen, $x_2$ und $x_3$ entfernt seien.
\begin{figure}[htbp]
\centering
\includegraphics[scale=0.4]{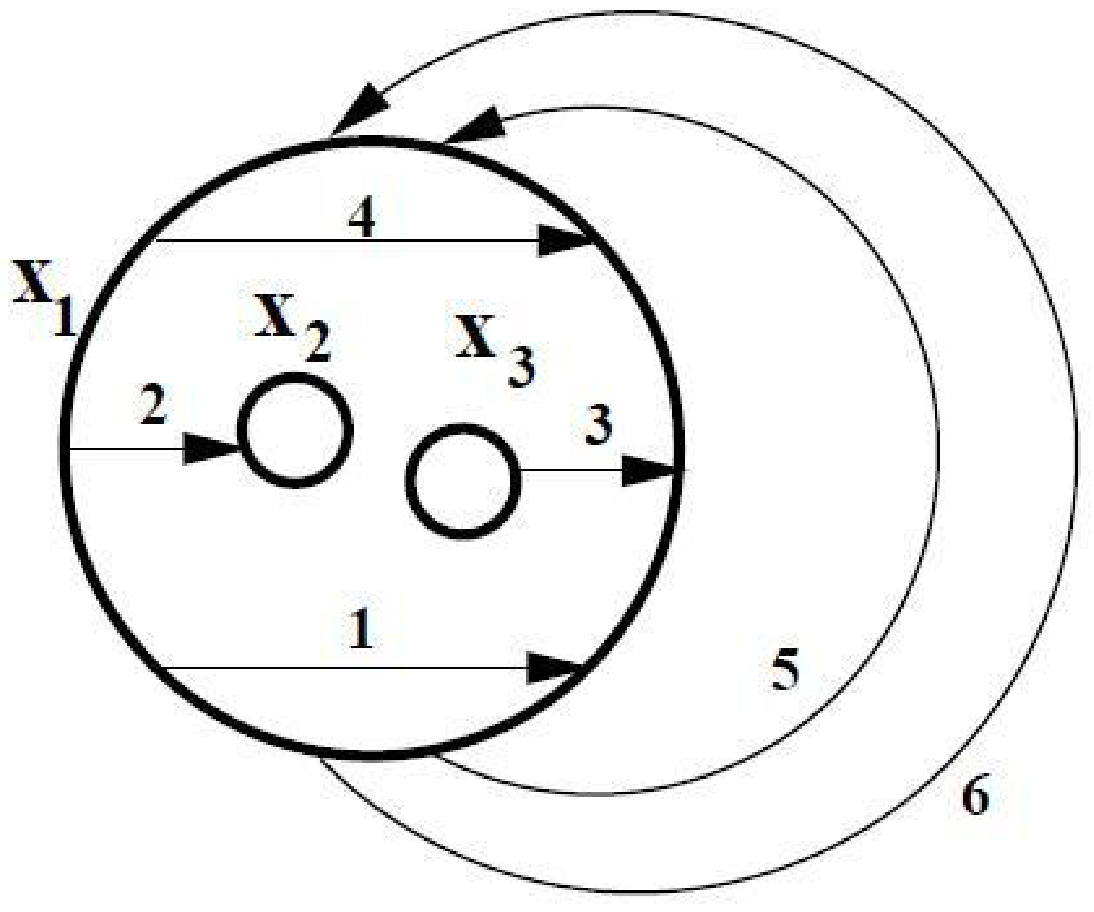}
\caption{aus \cite{szabo}}
\label{t3a}
\end{figure}
Sei $C$ wie in Abbildung \ref{t3a}.
Sei $a\in\sum(k,i)$ mit $i\in\{1,\dots,k-1\}$ und $a^0=(0,\dots,0)$ und so dass $\g(C,a)$ von einem Typ aus T ist. Dann ist $i=1$ oder $a$ entspricht den Bögen 1,2 oder $a$ entspricht den Bögen 3,4. Sei $b\in\sum(k,k-i)$ mit $b^0=a^1$. Dann ist $\g(C,b)$ von keinem Typ aus T. Es folgt der Fall $Q\in\{\textup{B}_{n}\}\times\textup{T}_3$.\\
Sei $Q\in\{\textup{A}_{n}\}\times\textup{T}_3$. Dann ist $z$ durch keinen der aktiven Kreise von $(\g(D,\eps))^*$ teilbar. Also ist genau einer dieser Kreise auch aktiver Kreis von $\g(D,\etaup)$. Für $n\ge 3$ enthält $D$ die Konfiguration $M_9$. Sei nun $n=2$. $D$ enthalte, auch nach Rotation, keine der lokalen Konfigurationen aus Abbildung \ref{mabb}. Dann ist $D$ bis auf Rotation und Tausch die Konfiguration aus Abbildung \ref{at3} (wobei wir die Orientierung der Bögen vergessen).
\begin{figure}[htbp]
\centering
\includegraphics[scale=0.4]{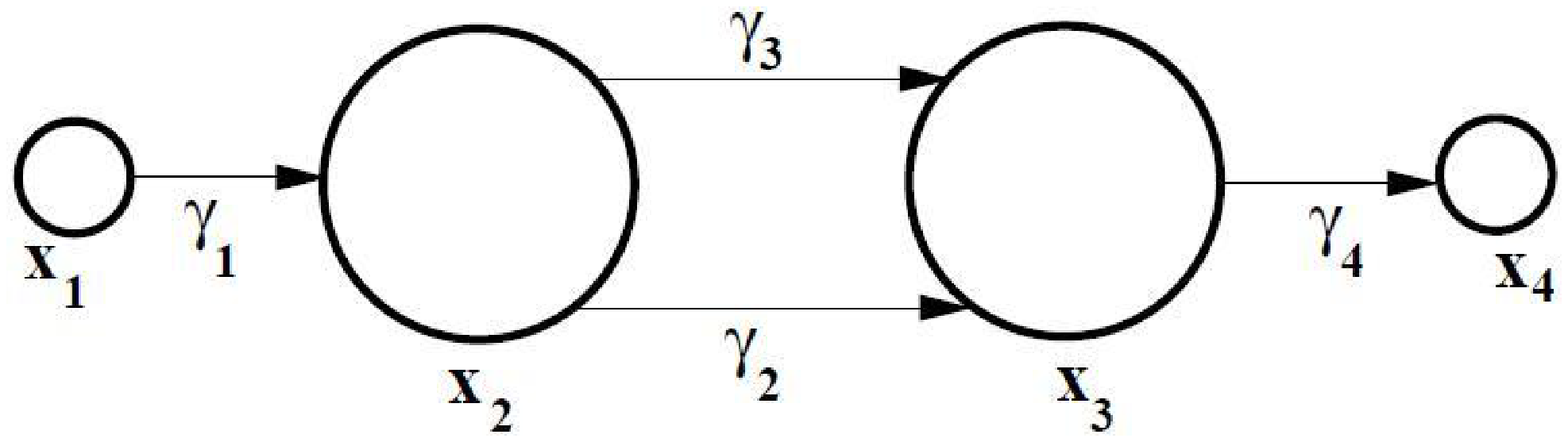}
\caption{aus \cite{szabo}}
\label{at3}
\end{figure}
Sei $C$ wie in Abbildung \ref{at3}.
Sei $a\in\sum(4,i)$ mit $i\in\{1,2,3\}$ und $a^0=(0,0,0,0)$ und so dass $\g(C,a)$ von einem Typ aus T ist. Dann ist $i=1$ oder $a$ entspricht den Bögen $\gammaup_2$, $\gammaup_3$. Sei $b\in\sum(4,4-i)$ mit $b^0=a^1$. Dann ist $\g(C,b)$ von keinem Typ aus T. Es folgt der Fall $Q\in\textup{T}_3\times\{\textup{B}_{k-n}\}$.\\
Sei $Q\in\textup{T}_3\times\{\textup{C}_{p,q}\ \ |\ \ p\le q\in\mathbb{N}\}$. Dann wird der aktive Kreis von $\g(D,\etaup)$ von genau einem Bogen aus $(\g(D,\eps))^*$ getroffen. Also gilt: Ist $n\ge 3$ so enthält $D$ nach Rotation eine der Konfigurationen $M_1$, $M_2$, $M_3$. Ist $n=2$ so enthält $D$ nach Rotation eine der Konfigurationen $M_6$, $M_7$. Es folgt der Fall $Q\in\{\textup{D}_{p,q}\ \ |\ \ p\le q\in\mathbb{N}\}\times\textup{T}_3$.\\
Sei $Q\in\textup{T}_3\times\{\textup{D}_{p,q}\}$ für $p\le q\in\mathbb{N}$. Dann ist $\alpha$ durch alle Kreise von $D$ teilbar, außer durch den Kreis $x_1$, der von mindestens zwei Bögen aus $\g(D,\eps)$ getroffen wird. Ein aktiver Kreis von $\g(D,\etaup)$ ist der Kreis, der von mindestens zwei Bögen aus $(\g(D,\eps))^*$ getroffen wird. Alle anderen aktiven Kreise von $\g(D,\etaup)$ sind passiv in $(\g(D,\eps))^*$. Enthält $D$ nicht die Konfiguration $M_8$, so bleiben  mit dem 2. Teil von Lemma \ref{gr2} für $(p,q)$ noch die Möglichkeiten $(1,1)$, $(1,2)$, $(1,3)$, $(1,4)$, $(2,2)$ zu betrachten. Genauer gilt: Man behalte von $D$ nur die Bögen, die $\etaup$ entsprechen, sowie die Bögen, die einen Randpunkt auf $x_1$ und den anderen Randpunkt auf einem anderen Kreis, der nicht nur von diesem Bogen getroffen wird, hat. Den aktiven Teil dieser Konfiguration nennen wir den Kern von $D$. Dann müssen wir für den Kern von $D$ nur die Konfigurationen aus Abbildung \ref{kern1} und \ref{kern2} betrachten.
\begin{figure}[htbp]
\centering
\includegraphics[scale=0.4]{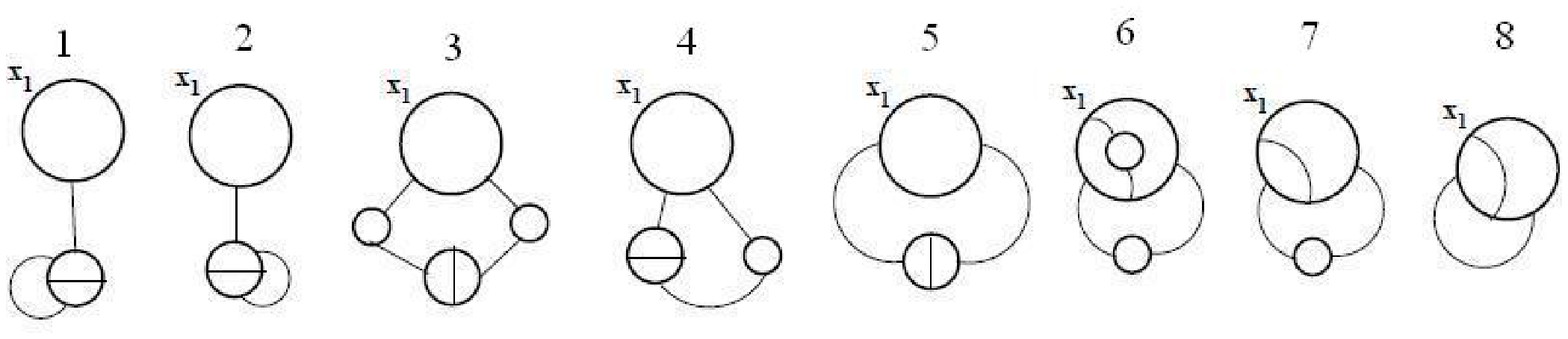}
\caption{Kern von $D$ (Abbildung aus \cite{szabo})}
\label{kern1}
\end{figure}
\begin{figure}[htbp]
\centering
\includegraphics[scale=0.5]{graph.27}
\caption{Kern von $D$}
\label{kern2}
\end{figure}
In den ersten 2 Fällen aus Abbildung \ref{kern1} und im Fall aus Abbildung \ref{kern2} enthält $D$ nach Rotation eine der Konfigurationen $M_1$-$M_5$. Betrachte nun die Fälle 3-8. $D$ enthalte, auch nach Rotation, keine der lokalen Konfigurationen aus Abbildung \ref{mabb}. Dann ist $D$ bis auf Rotation und Tausch wie in Abbildung \ref{t3d1} und \ref{t3d2} (wobei wir die Orientierung der Bögen vergessen), Grad 1 Kreise und die zugehörigen Bögen können entfernt werden, es müssen aber stets mindestens 4 Bögen übrig bleiben.
\begin{figure}[htbp]
\centering
\includegraphics[scale=0.4]{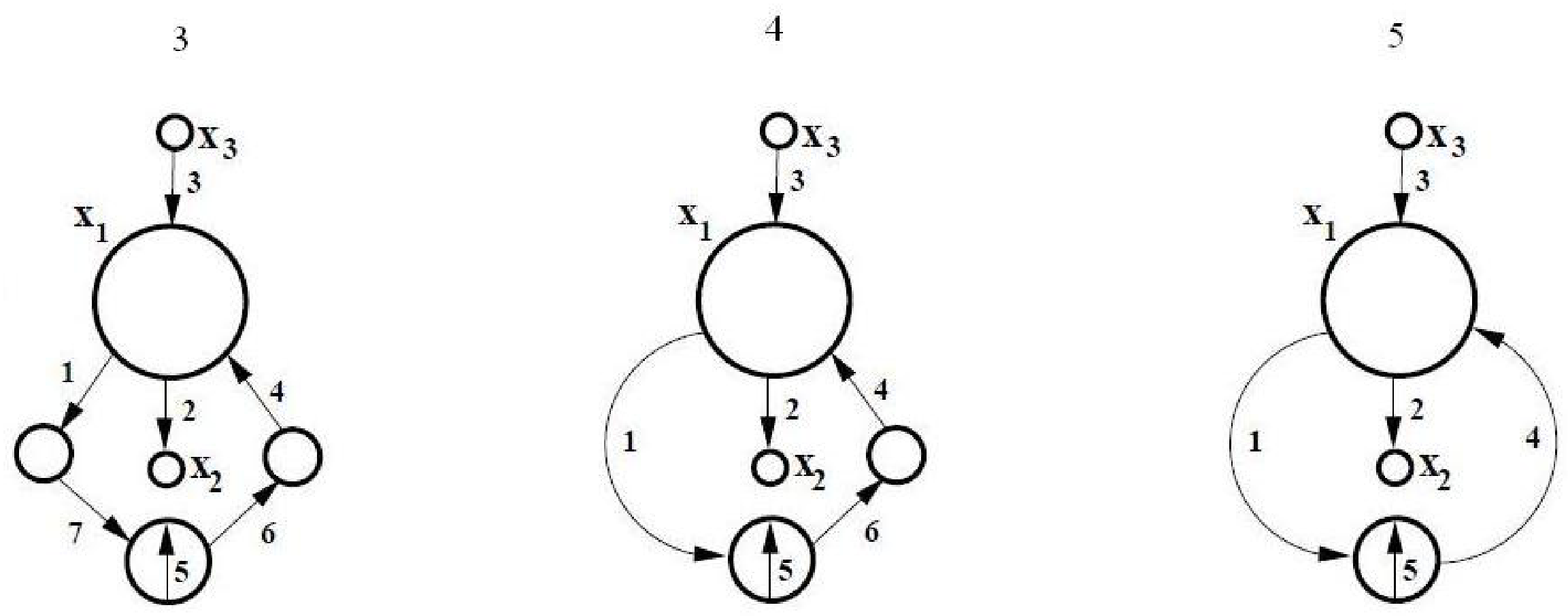}
\caption{aus \cite{szabo}}
\label{t3d1}
\end{figure}
\begin{figure}[htbp]
\centering
\includegraphics[scale=0.4]{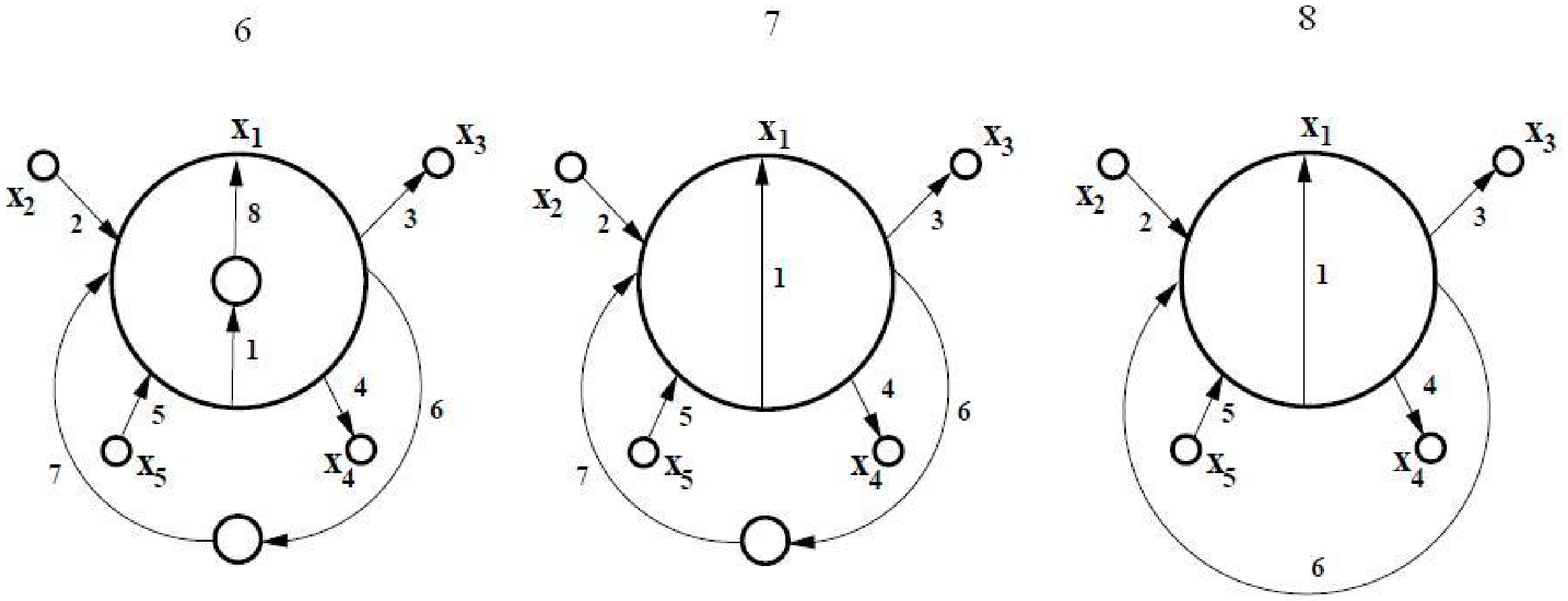}
\caption{aus \cite{szabo}}
\label{t3d2}
\end{figure}
Nach dem Beweis von Lemma \ref{trlem} genügt es den Fall zu betrachten, dass $\alpha$ durch alle Kreise von $D$ teilbar ist, außer durch $x_1$. Sei $C$ wie in den Abbildung \ref{t3d1} und \ref{t3d2}.
Sei $s$ Kantenzuordnung vom Typ Y bezüglich $C$. Sei $a\in\sum(k,i)$ mit $i\in\{1,\dots,k-1\}$ und $b\in\sum(k,k-i)$ mit $b^0=a^1$. In den Fällen 3, 4, 5 gilt: Aus $d_{\g(C,a),s|_a}(\alpha)\ne 0$ folgt $i=1$ oder $a$ entspricht $\gammaup_1$, $\gammaup_2$ oder $a$ entspricht $\gammaup_3$, $\gammaup_4$. Dann ist aber $\g(C,b)$ von keinem Typ aus T.\\
Für 6 gilt: Aus $d_{\g(C,a),s|_a}(\alpha)\ne 0$ folgt $i=1$ oder $a$ entspricht einer Teilmenge von $\{\gammaup_1,\ \gammaup_3,\ \gammaup_4,\ \gammaup_6\}$ oder einer Teilmenge von $\{\gammaup_2,\ \gammaup_5,\ \gammaup_7,\ \gammaup_8\}$. Dann ist aber $\g(C,b)$ von keinem Typ aus T.\\
Für 7 gilt: Aus $d_{\g(C,a),s|_a}(\alpha)\ne 0$ folgt $i=1$ oder $a$ entspricht einer Teilmenge von $\{\gammaup_3,\ \gammaup_4,\ \gammaup_6\}$ oder einer Teilmenge von $\{\gammaup_2,\ \gammaup_5,\ \gammaup_7\}$. Dann ist aber $\g(C,b)$ von keinem Typ aus T.\\
Für 8 gilt: Aus $d_{\g(C,a),s|_a}(\alpha)\ne 0$ folgt $i=1$ oder $a$ entspricht $\gammaup_2$, $\gammaup_5$ oder $\gammaup_3$, $\gammaup_4$ oder $\gammaup_5$, $\gammaup_6$ oder $\gammaup_3$, $\gammaup_6$. Aus $d_{\g(C,b),s|_b}\circ d_{\g(C,a),s|_a}(\alpha)\ne 0$ folgt, dass es für $a$ und $b$ folgende Möglichkeiten gibt:
\begin{description}
\item[-]
$a$ entspricht $\gammaup_1$ und $b$ entspricht $\gammaup_6$, $\gammaup_3$, $\gammaup_4$.
\item[-]
$a$ entspricht $\gammaup_1$ und $b$ entspricht $\gammaup_6$, $\gammaup_2$, $\gammaup_5$.
\item[-]
$a$ entspricht $\gammaup_5$, $\gammaup_6$ und $b$ entspricht $\gammaup_1$, $\gammaup_2$.
\item[-]
$a$ entspricht $\gammaup_3$, $\gammaup_6$ und $b$ entspricht $\gammaup_1$, $\gammaup_4$.
\end{description}
Es bleiben also für $D$ noch die beiden Konfigurationen aus Abbildung \ref{t3d91} zu betrachten.
\begin{figure}[htbp]
\centering
\includegraphics{1bild.16}\qquad\qquad\qquad\qquad
\includegraphics{1bild.17}
\caption{}
\label{t3d91}
\end{figure}
Die beiden sind aber gleich und gehen durch Rotation in die Konfiguration aus Abbildung \ref{t3d92} über, welche bereits behandelt wurde.
\begin{figure}[htbp]
\centering
\includegraphics{1bild.18}
\caption{}
\label{t3d92}
\end{figure}
Es folgt der Fall $Q\in\{\textup{C}_{p,q}\ \ |\ \ p\le q\in\mathbb{N}\}\times\textup{T}_3$.\\
Sei $Q\in\textup{T}_3\times\textup{T}_3$. Sei $r$ die Anzahl der Kreise von $\g(D,\eps^1)$ die sowohl in $(\g(D,\eps))^*$, als auch in $\g(D,\etaup)$ aktiv ist. Da $D$ zusammenhängend ist, ist $r\ge 1$. Es gibt genau einen aktiven Kreis von $(\g(D,\eps))^*$ durch den $z$ teilbar ist. Es gibt genau einen aktiven Kreis von $\g(D,\etaup)$ durch den $z$ nicht teilbar ist. Es folgt $r\le 2$. Ist $r=1$, so enthält $D$ nach Rotation eine der Konfigurationen $M_1$-$M_7$. Sei nun $r=2$. Sei $x_1$ der Kreis von $D$, der von mindestens zwei Bögen aus $\g(D,\eps)$ getroffen wird. Dann müssen wir nur den Fall betrachten, dass $\alpha$ durch alle Kreise von $D$ außer durch $x_1$ teilbar ist. $D$ enthalte, auch nach Rotation, keine der lokalen Konfigurationen aus Abbildung \ref{mabb}. Dann ist $D$ bis auf Rotation und Tausch wie in Abbildung \ref{t3t3}, wobei Grad 1 Kreise und die zugehörigen Bögen entfernt werden können, es müssen aber stets mindestens 4 Bögen übrig bleiben.
\begin{figure}[htbp]
\centering
\includegraphics{1bild.19}\qquad\qquad\qquad\qquad
\includegraphics{1bild.20}\\\ \\\ \\\ \\
\includegraphics{1bild.21}\qquad\qquad\qquad\qquad
\includegraphics{1bild.22}
\caption{}
\label{t3t3}
\end{figure}
Nach dem Beweis von Lemma \ref{trlem} genügt es den Fall zu betrachten, dass $\alpha$ durch alle Kreise von $D$ teilbar ist, außer durch $x_1$. Die beiden oberen Fälle wurden bereits behandelt. Für den unteren linken Fall ist $m(D^*)$ bis auf Rotation und Tausch gleich dem $D$ aus dem oberen rechten Fall. Für den uteren rechten Fall sei $C$ wie in Abbildung \ref{t3t3or}.
\begin{figure}[htbp]
\centering
\includegraphics{1bild.23}
\caption{}
\label{t3t3or}
\end{figure}
Dies lässt sich analog zum Fall 7 aus Abbildung \ref{t3d2} behandeln.
\end{proof}
Damit ist Satz \ref{exddn} bewiesen.
\end{proof}

\section{Invarianz der Spektralsequenz}
\begin{proof}[Beweis von Satz \ref{spekinv}]
Es genügt, den Fall zu betrachten, dass sich $C$ und $D$ nur in der Orientierung des $i$-ten Bogens unterscheiden. Es seien $f,g:\Gamma(C)\to\Gamma(C)$ definiert durch $f\defeq\textup{id}_{\Gamma(C)}+H_i(C,s)$ und $g\defeq\textup{id}_{\Gamma(C)}-H_i(C,s)$. Dann gilt $f\circ g=\textup{id}_{\Gamma(C)}=g\circ f$. Nach Satz \ref{dhs} gibt es eine Kantenzuordnung $u$ vom Typ Y bezüglich $D$, so dass
\[d(C,s)-d(D,u)=d(C,s)\circ H_i(C,s)-H_i(C,s)\circ d(C,s).\]
Es folgt:
\begin{align*}
&d(D,u)\circ f(x)\\
=&(d(C,s)-d(C,s)\circ H_i(C,s)+H_i(C,s)\circ d(C,s))(x+H_i(C,s)(x))\\
=&d(C,s)(x)+H_i(C,s)\circ d(C,s)(x)-d(C,s)\circ\underbrace{H_i(C,s)\circ H_i(C,s)}_{=0}(x)+\underbrace{H_i(C,s)\circ d(C,s)\circ H_i(C,s)}_{=0}(x)\\
=&f\circ d(C,s)(x)
\end{align*}
Also sind $\Omega(C,s)$ und $\Omega(D,u)$ isomorph. Der Beweis von Lemma \ref{inv1} liefert einen Isomorphismus zwischen $\Omega(D,u)$ und $\Omega(D,t)$.
\end{proof}
\begin{proof}[Beweis von Satz \ref{spekrminv}]
Die Aussage folgt aus den Beweisen von Proposition \ref{p1}, Proposition \ref{p2} und Proposition \ref{p3}. Für Reidemeister-Bewegung 3 beachte: Sei $C$ eine orientierte mindestens zweidimensionale Konfiguration, die eine lokale Konfiguration wie in Abbildung \ref{spekrm3} enthält.
\begin{figure}[htbp]
\centering
\includegraphics{1bild.24}\qquad\qquad
\includegraphics{1bild.25}\qquad\qquad
\includegraphics{1bild.26}
\caption{}
\label{spekrm3}
\end{figure}
Dann gilt für jede Kantenzuordnung $s$ vom Typ Y bezüglich $C$ und für jedes Monom $\alpha\in\Lambda\V(C)$: Ist $\alpha$ nicht durch $x$ teilbar, so ist $d_{C,s}(\alpha)=0$.
\end{proof}

\newpage

\bibliographystyle{unsrt}
\bibliography{biblio}

\begin{thebibliography}{1}

\bibitem{kh}
Mikhail Khovanov.
\newblock A categorification of the jones polynomial.
\newblock {\em DUKE MATH.J}, 3:359, 2000.

\bibitem{barnat}
Dror Bar-Natan.
\newblock {On Khovanov's categorification of the Jones polynomial.}
\newblock {\em Algebr. Geom. Topol.}, 2:337--370, 2002.

\bibitem{ors}
Peter Ozsvath, Jacob Rasmussen, and Zoltan Szabo.
\newblock Odd khovanov homology.
\newblock \href{http://arxiv.org/abs/0710.4300v1}{arXiv:0710.4300v1}, 2007.

\bibitem{szabo}
Zoltan Szabo.
\newblock A geometric spectral sequence in khovanov homology.
\newblock \href{http://arxiv.org/abs/1010.4252v1}{arXiv:1010.4252v1}, 2010.

\bibitem{weibel}
Charles~A. Weibel.
\newblock {\em {An introduction to homological algebra.}}
\newblock {Cambridge Studies in Advanced Mathematics. 38. Cambridge: Cambridge
  University Press}, 1994.

\end{thebibliography}

\end{document}